\documentclass{amsart}
\usepackage[oldstylemath,fulloldstylenums]{kpfonts}%font package stolen from tslil

\usepackage{amsmath}
\usepackage{amsxtra}
\usepackage{amsthm}

\usepackage{mathrsfs}
\usepackage{tikz-cd}
\usepackage{mathtools}

\usepackage[bbgreekl]{mathbbol}
\usepackage{stmaryrd}
\usepackage{enumitem} %for [resume]
\usepackage{xcolor}
\definecolor{darkgreen}{rgb}{0,0.45,0} 
\definecolor{darkred}{rgb}{0.75,0,0}
\definecolor{darkblue}{rgb}{0,0,0.6} 
\usepackage[pdfborder=0,pagebackref,colorlinks,citecolor=darkgreen,linkcolor=darkgreen,urlcolor=darkblue]{hyperref}
\usepackage{comment}

\newcommand{\cosk}{\mathrm{cosk}}
\newcommand{\latch}[1]{\widehat{\ell}^{#1}}
\newcommand{\match}[1]{\widehat{m}^{#1}}

\newcommand{\leib}[1]{\mathbin{\hat{#1}}}

\newcommand{\utimes}{\underline{\times}}

\setlist{}
\setenumerate{leftmargin=*,labelindent=0\parindent}
\setitemize{leftmargin=\parindent}%,labelindent=0.5\parindent}
\newtheorem{thm}{Theorem}[subsection]
\newtheorem{lem}[thm]{Lemma}
\newtheorem{prop}[thm]{Proposition}
\newtheorem{cor}[thm]{Corollary}

\theoremstyle{definition}
\newtheorem{defn}[thm]{Definition}
\newtheorem{ex}[thm]{Example}

\theoremstyle{remark}
\newtheorem{rmk}[thm]{Remark}
\newtheorem{exc}[thm]{Exercise}

\makeatletter
\@addtoreset{section}{part}
\makeatother  

\makeatletter
\let\c@equation\c@thm
\makeatother
\numberwithin{equation}{subsection}

\newcommand{\op}{\mathrm{op}}
\newcommand{\co}{\mathrm{co}}

\newcommand{\id}{\mathrm{id}}
\newcommand{\ob}{\mathrm{ob}}
\newcommand{\mor}{\mathrm{mor}}
\newcommand{\To}{\Rightarrow}

\newcommand{\ev}{\mathrm{ev}}
\newcommand{\proj}{\mathrm{proj}}
\newcommand{\incl}{\mathrm{incl}}

\newcommand{\Ran}{\mathrm{Ran}}
\newcommand{\Lan}{\mathrm{Lan}}

\newcommand{\dom}{\mathrm{dom}}
\newcommand{\cod}{\mathrm{cod}}
\newcommand{\colim}{\mathrm{colim}}

\newcommand{\wcolim}{\mathrm{wcolim}}

\newcommand{\cat}[1]{\textup{\textsf{#1}}}% for categories
\newcommand{\fun}[1]{\textup{#1}}%for functors
\newcommand{\LL}{\mathbb{L}}

\newcommand{\RR}{\mathbb{R}}

\DeclareMathAlphabet{\mathbbe}{U}{bbold}{m}{n}

\newcommand{\1}{\mathbbe{1}}
\newcommand{\2}{\mathbbe{2}}
\newcommand{\3}{\mathbbe{3}}
\newcommand{\iso}{\mathbb{I}}

\newcommand{\cA}{\mathsf{A}}

\newcommand{\cC}{\mathsf{C}}
\newcommand{\cD}{\mathsf{D}}
\newcommand{\cE}{\mathsf{E}}
\newcommand{\cF}{\mathsf{F}}

\newcommand{\cK}{\mathsf{K}}
\newcommand{\cL}{\mathsf{L}}
\newcommand{\cM}{\mathsf{M}}

\newcommand{\cN}{\mathsf{N}}

\newcommand{\cS}{\mathsf{S}}

\newcommand{\cV}{\mathsf{V}}

\newcommand{\DDelta}{\mathbbe{\Delta}}

\newcommand{\we}{\mathcal{W}}
\newcommand{\cof}{\mathcal{C}}
\newcommand{\fib}{\mathcal{F}}

\newcommand{\sA}{\mathcal{A}}
\newcommand{\sB}{\mathcal{B}}
\newcommand{\sC}{\mathcal{C}}
\newcommand{\sJ}{\mathcal{J}}
\newcommand{\sL}{\mathcal{L}}
\newcommand{\sR}{\mathcal{R}}

\newcommand{\bL}{\mathbf{L}}
\newcommand{\bR}{\mathbf{R}}

\newcommand{\gC}{\mathfrak{C}}
\newcommand{\gN}{\mathfrak{N}}

\font\maljapanese=dmjhira at 2ex % you can change this "2ex" value
\def\yo{\textrm{\maljapanese\char"48}}

\newcommand{\ho}{\mathfrak{h}}

\makeatletter
\def\makeslashed#1#2#3#4#5{#1{\mathpalette{\sla@{#2}{#3}{#4}}{#5}}}

\def\@mathlower#1#2#3{\setbox0=\hbox{$\m@th#2#3$}\lower#1\ht0\box0}
\def\mathlower#1#2{\mathpalette{\@mathlower{#1}}{#2}}
\makeatother

\newcommand\dhxrightarrow[2][]{%
  \mathrel{\ooalign{$\xrightarrow[#1\mkern4mu]{#2\mkern4mu}$\cr%
  \hidewidth$\rightarrow\mkern4mu$}}
}

\newcommand\tailxrightarrow[2][]{%
  \mathrel{\ooalign{$\xrightarrow[#1\mkern4mu]{#2\mkern4mu}$\cr%
  \hidewidth$\Yright\mkern14mu$}}
}

\newcommand{\fto}{\twoheadrightarrow}
\newcommand{\cto}{\rightarrowtail}
\newcommand{\wto}{\xrightarrow{{\smash{\mathlower{0.8}{\sim}}}}}
\newcommand{\cwto}{\tailxrightarrow{{\smash{\mathlower{0.8}{\sim}}}}}
\newcommand{\fwto}{\dhxrightarrow{{\smash{\mathlower{0.8}{\sim}}}}}

\newcommand{\Map}{\mathord{\text{\normalfont{\textsf{Map}}}}}
\newcommand{\Fun}{\mathord{\text{\normalfont{\textsf{Fun}}}}}
\newcommand{\Hom}{\mathord{\text{\normalfont{\textsf{Hom}}}}}
\newcommand{\Ho}{\mathord{\text{\normalfont{\textsf{Ho}}}}}
\newcommand{\h}{\cat{h}}
\newcommand{\sk}{\mathrm{sk}}

\begin{document}

\title{Homotopical categories: from model categories to $(\infty,1)$-categories}
\author{Emily Riehl}
\date{\today} 

\thanks{The author wishes to thank Andrew Blumberg, Teena Gerhardt, and Mike Hill for putting together this volume and inviting her to contribute. Daniel Fuentes-Keuthan gave detailed comments on a draft version of this chapter, and Yu Zhang and 
Chris Kapulkin pointed out key eleventh hour typos. She was supported by the National Science Foundation via the grants DMS-1551129 and DMS-1652600.}

\address{Department of Mathematics\\Johns Hopkins University \\ 3400 N Charles Street \\ Baltimore, MD 21218}
\email{eriehl@math.jhu.edu}

\begin{abstract}  
This chapter, written for \emph{Stable categories and structured ring spectra,} edited by Andrew J.~Blumberg, Teena Gerhardt, and Michael A.~Hill, surveys the history of homotopical categories, from Gabriel and Zisman's categories of fractions to Quillen's model categories, through Dwyer and Kan's simplicial localizations and culminating in $(\infty,1)$-categories, first introduced through concrete models and later re-conceptualized in a model-independent framework. This reader is not presumed to have prior acquaintance with any of these concepts. Suggested exercises are included to fertilize intuitions and copious references point to external sources with more details. A running theme of homotopy limits and colimits is included to explain the kinds of problems homotopical categories are designed to solve as well as technical approaches to these problems.
\end{abstract}

\maketitle

\setcounter{tocdepth}{2}
\tableofcontents

\section{The history of homotopical categories}

A \emph{homotopical category} is a category equipped with some collection of morphisms traditionally called ``weak equivalences'' that somewhat resemble isomorphisms but fail to be invertible in any reasonable sense, and might in fact not even be reversible:  that is the presence of a weak equivalence $X \wto Y$ need not imply the presence of a weak equivalence $Y \wto X$. Frequently, the weak equivalences are defined as the class of morphisms in a category $\cK$ that are ``inverted by a functor'' $F \colon \cK \to \cL$, in the sense of being precisely those morphisms in $\cK$ that are sent to isomorphisms in $\cL$. For instance.
\begin{itemize}
\item Weak homotopy equivalences of spaces or spectra are those maps inverted by the homotopy group functors $\pi_* \colon \cat{Top} \to \cat{GrSet}$ or $\pi_* \colon \cat{Spectra} \to \cat{GrAb}$.
\item Quasi-isomorphisms of chain complexes are those maps inverted by the homology functor $H_* \colon \cat{Ch} \to \cat{GrAb}$;
\item Equivariant weak homotopy equivalences of $G$-spaces are those maps inverted by the homotopy functors on the fixed point subspaces for each compact subgroup of $G$.
\end{itemize}

The term used to describe the equivalence class represented by a topological space up to weak homotopy equivalence is a \emph{homotopy type}. In view of the fact that the weak homotopy equivalence relation is created by the functor $\pi_*$,  a homotopy type can loosely be thought of as a collection of ``algebraic invariants'' of the space $X$, as encoded by the homotopy groups $\pi_*X$. Homotopy types live in a category called the \emph{homotopy category of spaces}, which is related to the classical category of spaces as follows: a genuine continuous function $X \to Y$ certainly represents a map (graded homomorphism) between homotopy types. But a weak homotopy equivalence of spaces, defining an isomorphism of homotopy types, should now be regarded as formally invertible.

In their 1967 manuscript \emph{Calculus of fractions and homotopy theory}, Gabriel and Zisman \cite{GZ} formalized the construction of what they call the \emph{category of fractions} associated to any class of morphisms in any category together with an associated localization functor $\pi \colon \cK \to \cK[\we^{-1}]$ that is universal among functors with domain $\cK$ that invert the class $\we$ of weak equivalences.  This construction and its universal property is presented in \S\ref{sec:fractions}. For instance, the homotopy category of spaces arises as the category of fractions associated to the weak homotopy equivalences of spaces.

There is another classical model of the homotopy category of spaces that defines an equivalence category. The objects in this category are the \emph{CW-complexes}, spaces built by gluing disks along their boundary spheres, and the morphisms are now taken to be homotopy classes of maps. By construction the isomorphisms in this category are the homotopy equivalences of CW-complexes. Because any space is weak homotopy equivalent to a CW-complex and because Whitehead's theorem proves that the weak homotopy equivalences between CW-complexes are precisely the homotopy equivalences, it can be shown this new homotopy category is equivalent to the Gabriel-Zisman category of fractions.

Quillen introduced a formal framework to which draws attention to the essential features of these equivalent constructions. His axiomatization of an abstract ``homotopy theory'' was motivated by the following question: when does it make sense to invert a class of morphisms in a category and call the result a homotopy category, rather than simply a localization? In the introduction to his 1967 manuscript \emph{Homotopical Algebra} \cite{quillen}, Quillen reports that Kan's theorem that the homotopy theory of simplicial groups is equivalent to the homotopy theory of connected pointed spaces \cite{kan-group} suggested to Quillen that simplicial objects over a suitable category $\cA$ might form a homotopy theory analogous to classical  homotopy theory in algebraic topology. In pursing this analogy he observed
\begin{quote}
there were a large number of arguments which were formally similar to well-known ones in algebraic topology, so it was decided to define the notion of a homotopy theory in sufficient generality to cover in a uniform way the different homotopy theories encountered.
\end{quote}

Quillen named these homotopy theories \emph{model categories}, meaning ``categories of models for a homotopy theory.'' Quillen entitled his explorations ``homotopical algebra'' as they describe both a generalization of and a close analogy to  homological algebra --- in which the relationship between an abelian category and its derived category parallels the relationship between a model category and its homotopy category. We introduce Quillen's model categories and his construction of their homotopy categories as a category of ``homotopy'' classes of maps between sufficiently ``fat'' objects in \S\ref{sec:model}. A theorem of Quillen proven as Theorem \ref{thm:model-saturation} below shows that the weak equivalences in any model category are precisely those morphisms inverted by the Gabriel-Zisman localization functor to the homotopy category. In particular, in the homotopical categories that we will most frequently encounter, the weak equivalences satisfy a number of closure properties to be introduced in Definition \ref{defn:we-properties}.

To a large extent, homological algebra is motivated by the problem of constructing ``derived'' versions of functors between categories of chain complexes that fail to preserve weak equivalences. A similar question arises in Quillen's model categories. Because natural transformations can point either to or from a given functor, derived functors come with a ``handedness'': either left or right. In \S\ref{sec:derived}, we introduce dual notions of left and right Quillen functors between model categories and construct their derived functors via a slightly unusual route that demands a stricter (but in our view improved) definition of derived functors than the conventional one. In parallel, we study the additional properties borne by Quillen's original model structure on simplicial sets later axiomatized by Hovey \cite{hovey} in the notion of a monoidal or enriched model category, which derives to define monoidal structures or enrichments on the homotopy category.

These considerations also permit us to describe when two ``homotopy theories'' are equivalent. For instance, the analogy between homological and homotopical algebra is solidified by a homotopical reinterpretation of the Dold-Kan theorem as an equivalence between the homotopy theory of simplicial objects of modules and chain complexes of modules presented in Theorem \ref{thm:dold-kan}.

As an application of the theory of derived functors, in \S\ref{sec:holim} we study homotopy limits and colimits, which correct for the defect that classically defined limit and colimit constructions frequently fail to be weak equivalence invariant. We begin by observing that the homotopy category admits few strict limits. It does admit weak ones, as we see in Theorem \ref{thm:vogt-w-limits}, but their construction requires higher homotopical information which will soon become a primary focus.

By convention, a full Quillen model structure can only be borne by a category possessing all limits and colimits, and hence the homotopy limits and homotopy colimits introduced in \S\ref{sec:holim} are also guaranteed to exist. This supports the point of view that a model category is a presentation of a homotopy theory with all homotopy limits and homotopy colimits. In a series of papers \cite{DK-simplicial, DK-calculating, DK-function} published by Dwyer and Kan in 1980, Dwyer and Kan describe more general ``homotopy theories'' as \emph{simplicial localizations} of categories with weak equivalences, which augment the Gabriel-Zisman category of fractions with homotopy types of the mapping spaces between any pair of objects. The \emph{hammock localization}  construction described in \S\ref{sec:hammock} is very intuitive, allowing us to re-conceptualize the construction of the category of fractions not by ``imposing relations'' in the same dimension, but by adding maps, in the next dimension ---  ``imposing homotopy relations'' if you will.

The hammock localization defines a simplicially enriched category associated to any homotopical category. A simplicially enriched category is a non-prototypical exemplification of  the notion of an $(\infty,1)$-\emph{category}, that is, a category weakly enriched over $\infty$-groupoids or homotopy types. Model categories also equip each pair of their objects with a well-defined homotopy type of maps, and hence also present $(\infty,1)$-categories.  Before exploring the schematic notion of $(\infty,1)$-\emph{category} in a systematic way, in \S\ref{sec:quasi} we introduce the most popular model, the \emph{quasi-categories} first defined in 1973 by Boardman and Vogt \cite{BV} and further developed by Joyal \cite{joyal-quasi,joyal-theory} and Lurie \cite{lurie-topos}. 

In \S\ref{sec:models} we turn our attention to other models of $(\infty,1)$-categories, studying six models in total: quasi-categories, Segal categories, complete Segal spaces, naturally marked quasi-categories, simplicial categories, and relative categories. The last two models are strictly-defined objects, which are quite easy to define, but the model categories in which they live are poorly behaved. By contrast, the first four of these models live in model categories that have many pleasant properties, that are collected together in a new axiomatic notion of an $\infty$-\emph{cosmos}.

After introducing this abstract definition, in \S\ref{sec:indep} we see how the $\infty$-cosmos axiomatization allows us to develop the basic theory of these four models of $(\infty,1)$-categories ``model-independently,'' that is simultaneously and uniformly across these models. Specifically, we study adjunctions and equivalences between $(\infty,1)$-categories and limits and colimits in an $(\infty,1)$-category to provide points of comparison for the corresponding notions of Quillen adjunction, Quillen equivalence, and homotopy limits and colimits developed for model categories in \S\ref{sec:derived} and \S\ref{sec:holim}. A brief epilogue \S\ref{sec:epilogue} contains a few closing thoughts and anticipates future chapters in this volume.

\section{Categories of fractions and localization}\label{sec:fractions}

In one of the first textbook accounts of abstract homotopy theory \cite{GZ}, Gabriel and Zisman construct the universal category that inverts a collection of morphisms together with accompanying ``calculi-of-fractions'' techniques for calculating this categorical ``localization.'' Gabriel and Zisman prove that a class of morphisms in a category with finite colimits admits a ``calculus of left fractions'' if and only if the corresponding localization preserves them, which then implies that the category of fractions also admits finite colimits \cite[\S 1.3]{GZ}; dual results relate finite limits to their ``calculus of right fractions.'' For this reason, their calculii of fractions fail to exist in the examples of greatest interest to modern homotopy theorists, and so we decline to introduce them here, focused instead in \S\ref{ssec:GZ-cat} on the general construction of the category of fractions.

\subsection{The Gabriel--Zisman category of fractions}\label{ssec:GZ-cat}

For any class of morphisms $\we$ in a category $\cK$, the category of fractions $\cK[\we^{-1}]$ is the universal category equipped with a functor $\iota \colon \cK \to \cK[\we^{-1}]$ that inverts $\we$, in the sense of sending each morphism to an isomorphism. Its objects are the same as the objects of $\cK$ and its morphisms are finite zig-zags of morphisms in $\cK$, with all ``backwards'' arrows finite composites of arrows belonging to $\we$, modulo a few relations which convert the canonical graph morphism $\iota \colon \cK \to \cK[\we^{-1}]$ into a functor and stipulate that the ``backwards'' copies of each arrow in $\we$ define two-sided inverses to the morphisms in $\we$.

\begin{defn}[{category of fractions \cite[1.1]{GZ}}]\label{defn:cat-of-fractions} 
For any class of morphisms $\we$ in a category $\cK$, the \textbf{category of fractions} $\cK[\we^{-1}]$ is a quotient of the free category on the directed graph obtained by adding ``backwards'' copies of the morphisms in $\we$ to the underlying graph of the category $\cK$ modulo the relations:
\begin{itemize}
\item Adjacent arrows pointing ``forwards'' can be composed.
\item Forward-pointing identities may be removed.
\item Adjacent pairs of zig-zags \[ \begin{tikzcd} x \arrow[r, "s"] &y &x\arrow[l, "s"'] \end{tikzcd}\quad \mathrm{or} \quad \begin{tikzcd} y&x\arrow[l, "s"'] \arrow[r, "s"] & y \end{tikzcd}\] indexed by any $s \in \we$ can be removed.\footnote{It follows that adjacent arrows in $\we$ pointing ``backwards'' can also be composed whenever their composite in $\cK$ also lies in $\we$.}
\end{itemize}
The image of the functor $\iota \colon \cK \to \cK[\we^{-1}]$ is comprised of those morphisms that can be represented by unary zig-zags pointing ``forwards.''
 \end{defn}

The following proposition expresses the 2-categorical universal property of the category of fractions construction in terms of categories $\Fun(\cK,\cM)$ of functors and natural transformations:

\begin{prop}[{the universal property of localization\cite[1.2]{GZ}}]\label{prop:fractions-UP} For any category $\cM$, restriction along $\iota$ defines a fully faithful embedding
\[
\begin{tikzcd}[row sep=tiny]
\Fun(\cK[\we^{-1}],\cM) \arrow[dr, "\cong"'] \arrow[rr, hook, "-\circ \iota"] & & \Fun(\cK,\cM) \\ &  \underset{\we \mapsto \cong}{\Fun}(\cK, \cM) \arrow[ur, hook]
\end{tikzcd}
\]
defining an isomorphism $\Fun(\cK[\we^{-1}],\cM) \cong \underset{\we \mapsto \cong}{\Fun}(\cK, \cM)$ of categories onto its essential image, the full subcategory  spanned by those functors that invert $\we$.
\end{prop}
\begin{proof}
As in the analogous case of rings, the functor $\iota\colon \cK \to \cK[\we^{-1}]$ is an epimorphism and so any functor $F \colon \cK \to \cM$ admits at most one extension along $\iota$. To show that any functor $F \colon \cK \to \cM$ that inverts $\we$ does extend to $\cK[\we^{-1}]$, note that we may define a graph morphism from the graph described in Definition \ref{defn:cat-of-fractions} to $\cM$ by sending the backwards copy of $s$ to the isomorphism $(Fs)^{-1}$ and thus a functor from the free category generated by this graph to $\cM$. Functoriality of $F$ ensures that the enumerated relations are respected by this functor, which therefore defines an extension $\hat{F} \colon \cK[\we^{-1}] \to \cM$ as claimed.

The 2-dimensional aspect of this universal property follows from the 1-dimen\-sion\-al one by considering functors valued in arrow categories \cite[\S 3]{kelly-2-limit}.
\end{proof}
%A natural transformation between functors $\cK[\we^{-1}] \to \cM$ is encoded by a functor $\cK[\we^{-1}] \to \cM^\2$ valued in the arrow category of $\cM$ that sends each object $k \in \ob\cK = \ob\cK[\we^{-1}]$ to the $k$-indexed component of the natural transformation. By the universal property just established, such functors correspond to functors $\cK \to \cM^\2$ that invert the morphisms in $\we$, i.e., to natural transformations between functors that both invert $\we$. This establishes the 2-categorical universal property of the statement.

\begin{ex}[groupoid reflection] When all of the morphisms in $\cK$ are inverted, the universal property of Proposition \ref{prop:fractions-UP} establishes an isomorphism $\Fun(\cK[\cK^{-1}], \cM) \cong \Fun(\cK,\fun{core}\cM)$ between functors from the category of fractions of $\cK$ to functors valued in the \textbf{groupoid core}, which is the maximal subgroupoid contained in $\cM$. In this way, the category of fractions construction specializes to define a left adjoint\footnote{More precisely, this left adjoint takes values in a larger universe of groupoids, since the category of fractions $\cK[\cK^{-1}]$ associated to a locally small category $\cK$ need not be locally small. Toy examples illustrating this phenomenon are easy to describe. For instance let $\cK$ be a category with a proper class of objects whose morphisms define a ``double asterisk'': each non-identity morphism has a common domain object and for each other object there are precisely two non-identity morphisms with that codomain.}
 to the inclusion of groupoids into categories:
\[
\begin{tikzcd}[column sep=large]
\cat{Cat} \arrow[r, bend left, "\fun{fractions}"] \arrow[r, bend right, "\fun{core}"' pos = .55 ] & \cat{Gpd} \arrow[l, hook', "\perp", "\perp"'] 
\end{tikzcd}
\]
\end{ex}

The universal property of Proposition \ref{prop:fractions-UP} applies to the class of morphisms inverted by any functor admitting a fully faithful right adjoint \cite[1.3]{GZ}. In this case, the category of fractions defines a \emph{reflective subcategory} of $\cK$, which admits a variety of useful characterizations, for instance as the \emph{local} objects orthogonal to the class of morphisms being inverted \cite[4.5.12, 4.5.vii, 5.3.3, 5.3.i]{riehl-context}. For instance, if $R \to R[S^{-1}]$ is the localization of a commutative ring at a multiplicatively closed set, then the category of $R[S^{-1}]$-modules defines a reflective subcategory of the category of $R$-modules \cite[4.5.14]{riehl-context}, and hence the extension of scalars functor $R[S^{-1}] \otimes_R -$ can be understood as a Gabriel-Zisman localization.

However reflective subcategories inherit all limits and colimits present in the larger category \cite[4.5.15]{riehl-context}, which is not typical behavior for categories of fractions that are ``homotopy categories'' in a sense to be discussed in \S\ref{sec:model}. With the question ``when is a category of fractions a homotopy category'' in mind, we now turn our attention to Quillen's homotopical algebra.

\section{Model category presentations of homotopical categories}\label{sec:model}

A question that motivated Quillen's introduction of model categories \cite{quillen} and also Dwyer, Kan, Hirschhorn, and Smith's later generalization \cite{DHKS} is: when is a category of fractions a homotopy category? Certainly, the localization functor must invert some class of morphisms that are suitably thought of as ``weak equivalences.'' Perhaps these weak equivalences coincide with a more structured class of ``homotopy equivalences'' on a suitable subcategory of ``fat'' objects that spans each weak equivalence class --- such as given in the classical case by Whitehead's theorem that any weak homotopy equivalence between CW complexes admits a homotopy inverse --- in such a way that the homotopy category is equivalent to the category of homotopy classes of maps in this full subcategory. Finally, one might ask that the homotopy category admit certain derived constructions, such as the loop and suspension functors definable on the homotopy category of based spaces. On account of this final desideratum, we will impose the blanket requirement that a category that bears a model structure must be complete and cocomplete.

A class of morphisms $\we$ denoted by ``$\wto$'' in a category $\cM$ might reasonably be referred to as ``weak equivalences'' if they somewhat resemble isomorphisms, aside from failing to be ``invertible'' in any reasonable sense. The meaning of ``somewhat resembling isomorphisms'' may be made precise via any of the following axioms, all of which are satisfied by the isomorphisms in any category.

\begin{defn}\label{defn:we-properties} The following hypotheses are commonly applied to a class of ``weak equivalences'' $\we$ in a category $\cM$:
\begin{itemize}
\item The \textbf{two-of-three property}: for any composable pair of morphisms if any two of $f$, $g$, and $gf$ is in $\we$ then so is the third.
\item The \textbf{two-of-six property}: for any composable triple of morphisms
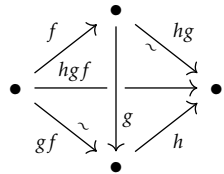
\[
\begin{tikzcd}
& \bullet  \arrow[dr, "hg", "\sim"' sloped] \\ \bullet \arrow[ur, "f"] \arrow[dr, "gf"', "\sim" sloped] \arrow[rr, "hgf" near start] & & \bullet \\ & \bullet \arrow[from = uu, crossing over, "g" near end] \arrow[ur, "h"'] 
\end{tikzcd}
\]
if $gf, hg \in \we$ then $f,g,h, hgf \in \we$.
\item The class $\we$ is \textbf{closed under retracts} in the arrow category: given a commutative diagram
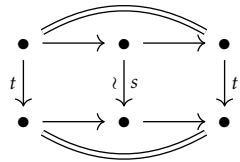
\[
\begin{tikzcd}
\bullet \arrow[d, "t"'] \arrow[r] \arrow[rr, bend left, equals] & \bullet \arrow[d, "s", "\wr"'] \arrow[r] & \bullet \arrow[d, "t"] \\ \bullet \arrow[r] \arrow[rr, bend right, equals] & \bullet \arrow[r] & \bullet
\end{tikzcd}
\]
if $s$ is in $\we$ then so is its retract $t$.
\item The class $\we$ might define a \textbf{wide subcategory}, meaning $\we$ is closed under composition and contains all identity morphisms.
\item More prosaically, it is reasonable to suppose that $\we$  contains the isomorphisms.
\item At bare minimum, one might at least insist that $\we$ contains all of the identities.
\end{itemize} 
\end{defn}

\begin{lem}\label{lem:moral-we-defn} Let $\we$ be the class of morphisms in $\cM$ inverted by a functor $F \colon \cM \to \cK$. Then $\we$ satisfies each of the closure properties just enumerated.
\end{lem}
\begin{proof} This follows immediately from the axioms of functoriality.
\end{proof}

In practice, most classes of weak equivalences arise as in Lemma \ref{lem:moral-we-defn}. For instance, the quasi-isomorphisms are those chain maps inverted by the homology functor $H_\bullet$ from chain complexes to graded modules while the weak homotopy equivalences are those continuous functions inverted by the homotopy group functors $\pi_\bullet$. Rather than adopt a universal set of axioms that may or may not fit the specific situation at hand, we will use the term \textbf{homotopical category} to refer to any pair $(\cM,\we)$ comprised of a category and a class of morphisms and enumerate the specific properties we need for each result or construction. When the homotopical category $(\cM,\we)$ underlies a model category structure, to be defined, Theorem \ref{thm:model-saturation} below proves $\we$ is precisely the class of morphisms inverted by the Gabriel-Zisman localization functor and hence satisfies all of the enumerated closure properties. 

The data of a \emph{model structure} borne by a homotopical category is given by two additional classes of morphisms ---  the \emph{cofibrations} $\cof$ denoted  ``$\rightarrowtail$'', and the \emph{fibrations} $\fib$ denoted ``$\twoheadrightarrow$'' --- 
satisfying axioms to be enumerated. In \S\ref{ssec:wfs}, we present a modern reformulation of Quillen's axioms that more clearly highlight the central features of a model structure borne by a complete and cocomplete category. In \S\ref{ssec:functoriality}, we discuss the delicate question of the functoriality of the factorizations in a model category with the aim of justifying our view that this condition is harmless to assume in practice.

In \S\ref{ssec:htpy-relation}, we explain what it means for a parallel pair of morphisms in a model category to be \emph{homotopic}; more precisely, we introduce distinct \emph{left homotopy} and \emph{right homotopy} relations that define a common equivalence relation when the domain  is \emph{cofibrant} and the codomain is \emph{fibrant}. The homotopy relation is used in \S\ref{ssec:htpy-cat-of-model} to construct and compare three equivalent models for the homotopy category of a model category: the Gabriel-Zisman category of fractions $\cM[\we^{-1}]$ defined by formally inverting the weak equivalences, the category $\h\cM_{\mathrm{cf}}$ of fibrant-cofibrant objects in $\cM$ and homotopy classes of maps, and an intermediary $\Ho\cM$ which has the objects of the former and hom-sets of the later, designed to facilitate the comparison. Finally, \S\ref{ssec:simp-sets} presents a fundamental example: Quillen's model structure on the category of simplicial sets.

\subsection{Model category structures via weak factorization systems}\label{ssec:wfs}

When Quillen first introduces the definition of a model category in the introduction to ``Chapter I.~Axiomatic Homotopy Theory'' \cite{quillen}, he highlights the factorization and lifting axioms as being the most important. These axioms are most clearly encapsulated in the categorical notion of a weak factorization system, a concept which was codified later.

\begin{defn}\label{defn:wfs} A \textbf{weak factorization system} $(\sL,\sR)$ on a category $\cM$ is comprised of two classes of morphisms $\sL$ and $\sR$ so that
\begin{enumerate}
\item\label{itm:wfs-factor} Every morphism in $\cM$ may be factored as a morphism in $\sL$ followed by a morphism in $\sR$.
\[
\begin{tikzcd}[row sep=tiny] 
\bullet \arrow[rr, "f"] \arrow[dr, "\sL\ni \ell"'] & & \bullet \\ & \bullet \arrow[ur, "r \in \sR"'] 
\end{tikzcd}
\]
\item\label{itm:wfs-lift} The maps in $\sL$ have the \textbf{left lifting property} with respect to each map in $\sR$ and equivalently the maps in $\sR$ have the \textbf{right lifting property} with respect to each map in $\sL$: that is, any commutative square
\[ 
\begin{tikzcd}
\bullet \arrow[d, "\sL \ni \ell"'] \arrow[r] & \bullet \arrow[d, "r \in \sR"] \\ \bullet \arrow[r] \arrow[ur, dashed] & \bullet
\end{tikzcd}
\]
admits a diagonal filler as indicated, making both triangles commute.
\item\label{itm:wfs-retract} The classes $\sL$ and $\sR$ are each closed under retracts in the arrow category: given a commutative diagram
\[
\begin{tikzcd}
\bullet \arrow[d, "t"'] \arrow[r] \arrow[rr, bend left, equals] & \bullet \arrow[d, "s"] \arrow[r] & \bullet \arrow[d, "t"] \\ \bullet \arrow[r] \arrow[rr, bend right, equals] & \bullet \arrow[r] & \bullet
\end{tikzcd}
\]
if $s$ is in that class then so is its retract $t$.
\end{enumerate}
\end{defn}

The following reformulation of Quillen's definition \cite[I.5.1]{quillen} was given by Joyal and Tierney \cite[7.7]{JT}, who prove that a homotopical category $(\cM,\we)$, with the weak equivalences satisfying the two-of-three property, admits a model structure just when there exist classes $\cof$ and $\fib$ that define a pair of weak factorization systems:

\begin{defn}[model category]\label{defn:model-cat} A \textbf{model structure} on a homotopical category $(\cM,\we)$ consists of three classes of maps --- the \textbf{weak equivalences} $\we$ denoted ``$\wto$'' which must satisfy the two-of-three property\footnote{The standard definition of a model category also requires the weak equivalences to be closed under retracts, but this is a consequence of the axioms given here \cite[7.8]{JT}.},  the \textbf{cofibrations} $\cof$ denoted  ``$\rightarrowtail$'', and the \textbf{fibrations} $\fib$ denoted ``$\twoheadrightarrow$'' respectively --- so that $(\cof,\fib\cap\we)$ and $(\cof\cap\we, \fib)$ each define weak factorization systems on $\cM$.
\end{defn}

\begin{rmk}[on self-duality]
Note that definitions \ref{defn:wfs} and \ref{defn:model-cat} are self-dual: if $(\sL,\sR)$ defines a weak factorization system on $\cM$ then $(\sR,\sL)$ defines a weak factorization system on $\cM^\op$. Thus the statements we prove about the left classes $\cof$ of cofibrations and $\cof\cap\we$ of \textbf{trivial cofibrations} ``$\cwto$'' will have dual statements involving the right classes  $\fib$  of fibrations and $\fib \cap \we$ of \textbf{trivial fibrations} ``$\fwto$.''
\end{rmk}

The third axiom \eqref{itm:wfs-retract} of Definition \ref{defn:wfs} was missing from Quillen's original definition of a model category; he referred to those model categories that have the retract closure property as ``closed model categories.'' The importance of this closure property is that it implies that the left class of a weak factorization system is comprised of all of those maps that have the left lifting property with respect to the right class and dually, that the right class is comprised of all of those maps that have the right lifting property with respect to the left class. %\footnote{It also implies that the weak equivalences are closed under retracts \cite[7.8]{JT}, a fact which is used in Quillen's Theorem \ref{thm:model-saturation} to show that the weak equivalences are precisely the class of morphisms inverted by the localization functor $\cM \to \cM[\we^{-1}]$.}
  These results follow as a direct corollary of the famous ``retract argument.''

\begin{lem}[retract argument]\label{lem:retract} Suppose $f = r \circ \ell$ and $f$ has the left lifting property with respect to its right factor $r$. Then $f$ is a retract of its left factor $\ell$.
\end{lem}
\begin{proof}
The solution to the lifting problem displayed below left
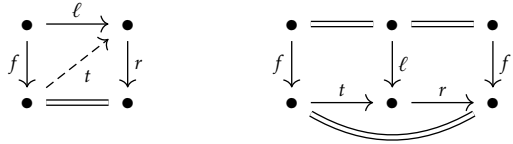
\[
\begin{tikzcd}
\bullet \arrow[d, "f"'] \arrow[r, "\ell"] & \bullet \arrow[d, "r"] & &  \bullet \arrow[d, "f"'] \arrow[r, equals] & \bullet \arrow[d, "\ell"] \arrow[r, equals] & \bullet \arrow[d, "f"] \\ \bullet \arrow[ur, dashed, "t"'] \arrow[r, equals]  &  \bullet & &  \bullet \arrow[r, "t"] \arrow[rr, bend right, equals] & \bullet \arrow[r, "r"] & \bullet
\end{tikzcd}
\]
defines the retract diagram displayed above right.
\end{proof}

\begin{cor}\label{cor:wfs-by-half} Either class of a weak factorization system determines the other: the left class consists of  those morphisms that have the left lifting property with respect to the right class, and the right class consists of  those morphisms that have the right lifting property with respect to the left class.
\end{cor}
\begin{proof}
Any map with the left lifting property with respect to the right class of a weak factorization system in particular lifts against its right factor of the factorization guaranteed by \ref{defn:wfs}\eqref{itm:wfs-factor} and so belongs to the left class by \ref{defn:wfs}\eqref{itm:wfs-retract}. 
\end{proof}

In particular the trivial cofibrations  can be defined without reference to either the cofibrations or weak equivalences as those maps that have the left lifting property with respect to the fibrations, and dually the trivial fibrations are precisely those maps that have the right lifting property with respect to the cofibrations.

\begin{exc}\label{exc:model-given-by} Verify that a model structure on $\cM$, if it exists, is uniquely determined by any of the following data:
\begin{enumerate}
\item\label{itm:model-i} The cofibrations and weak equivalences.
\item\label{itm:model-ii} The fibrations and weak equivalences.
\item\label{itm:model-iii} The cofibrations and fibrations.
\end{enumerate}
By a more delicate observation of Joyal \cite[E.1.10]{joyal-theory} using terminology to be introduced in Definition \ref{defn:fib-cof}, a model structure is also uniquely determined by 
\begin{enumerate}[resume]
\item\label{itm:model-iv} The cofibrations and fibrant objects.
\item\label{itm:model-v} The fibrations and cofibrant objects.
\end{enumerate}
\end{exc}

As a further consequence of the characterizations of the cofibrations, trivial cofibrations, fibrations, and trivial fibrations by lifting properties,  each class automatically enjoys certain closure properties.

\begin{lem}\label{lem:closure}
Let $\sL$ be any class of maps characterized by a left lifting property with respect to a fixed class of maps $\sR$. Then $\sL$ contains the isomorphisms and is closed under coproduct, pushout, retract, and (transfinite) composition.
\end{lem}
\begin{proof}
We prove the cases of pushout and transfinite composition to clarify the meaning of these terms, the other arguments being  similar. Let $k$ be a pushout of a morphism $\ell \in \sL$ as  displayed below left, and consider a lifting problem against a morphism $r \in \sR$ as presented by the square  below right:
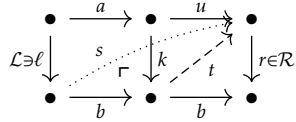
\[
\begin{tikzcd}%\bullet \arrow[d, "\sL\ni\ell"'] \arrow[r, "a"] \arrow[dr, phantom, "\ulcorner" very near end] & \bullet \arrow[d, "k"]  & & \bullet \arrow[d, "k"'] \arrow[r, "u"] & \bullet \arrow[d, "r \in \sR"]  & &
%\bullet \arrow[r, "b"'] & \bullet & & \bullet \arrow[r, "v"'] & \bullet & & 
 \bullet \arrow[d, "\sL\ni\ell"'] \arrow[r, "a"] \arrow[dr, phantom, "\ulcorner" very near end] & \bullet \arrow[d, "k"]  \arrow[r, "u"] & \bullet \arrow[d, "r \in \sR"] 
\\ \bullet \arrow[r, "b"'] \arrow[urr, dotted, bend left=10, "s" near start] & \bullet  \arrow[r, "b"'] \arrow[ur, dashed, "t"']  & \bullet
\end{tikzcd}
\]
Then there exists a lift $s$ in the composite rectangle and this lift together with $u$ define a cone under the pushout diagram, inducing the desired lift $t$.

Now let $\bbalpha$ denote any ordinal category. The \textbf{transfinite composite} of a diagram $\bbalpha \to \cM$ is the leg $\ell_\alpha$ of the colimit cone from the initial object in this diagram to its colimit. To see that this morphism lies in $\sL$ under the hypothesis that the generating morphisms $\ell_i$ in the diagram do, it suffices to construct the solution to any lifting problem against a map $r \in \sR$. 
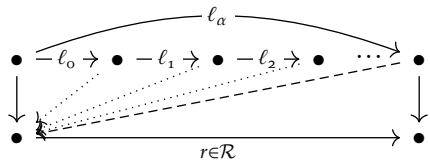
\[
\begin{tikzcd}
\bullet \arrow[r, "\ell_0" description] \arrow[d] \arrow[rrrr, bend left=20, "\ell_\alpha" description] & \bullet \arrow[r, "\ell_1" description]  \arrow[dl, dotted] & \bullet \arrow[r, "\ell_2" description]  \arrow[dll, dotted] & \bullet \arrow[r, phantom, "\cdots"] \arrow[dlll, dotted] & \bullet \arrow[d] \arrow[dllll, dashed] \\ \bullet \arrow[rrrr, "r \in \sR"'] & & & & \bullet
\end{tikzcd}
\]
By the universal property of the colimit object, this dashed morphism exists once the commutative cone of dotted lifts do, and these may be constructed sequentially starting by lifting $\ell_0$ against $r$.
\end{proof}

\begin{exc} Verify that the class of morphisms $\sL$ characterized by the left lifting property against a fixed class of morphisms $\sR$ is closed under coproducts, closed under retracts, and contains the isomorphisms.
\end{exc}

\begin{defn}\label{defn:cell-complex} Let $\sJ$ be any class of maps. A $\sJ$-\textbf{cell complex} is a map built as a transfinite composite of pushouts of coproducts of maps in $\sJ$, which may then be referred to in this context as the basic \textbf{cells}.
\[
\begin{tikzcd}
\bullet \arrow[r, "\coprod j \in \sJ"] \arrow[d] \arrow[dr, phantom, "\ulcorner" very near end] & \bullet \arrow[d] &  \bullet \arrow[r, "\coprod j \in \sJ"] \arrow[d] \arrow[dr, phantom, "\ulcorner" very near end] & \bullet \arrow[d] \\ \bullet  \arrow[rrrr, bend right=20, dashed, "\in \sJ\text{-cell}"' near end] \arrow[r] & \bullet \arrow[r] & \bullet \arrow[r] &  \bullet \arrow[r, dotted] & \bullet \\ & \bullet \arrow[u] \arrow[r, "\coprod j \in \sJ"'] \arrow[ur, phantom, "\llcorner" very near end] & \bullet \arrow[u] 
\end{tikzcd}
\]
\end{defn}

Lemma \ref{lem:closure} implies that the left class of a weak factorization is closed under the formation of cell complexes.

\begin{exc} Explore the reason why the class of morphisms $\sL$ characterized by the left lifting property against a fixed class of morphisms $\sR$ may fail to be closed under coequalizers, formed in the arrow category.\footnote{Note however if the maps in $\sL$ are equipped with specified solutions to every lifting problem posed by $\sR$ and if the squares in the coequalizer diagram commute with these specified lifts, then the coequalizer inherits canonically-defined solutions to every lifting problem posed by $\sR$ and is consequently in the class $\sL$.}
\end{exc}

\subsection{On functoriality of  factorizations}\label{ssec:functoriality}

The weak factorization systems that arise in practice, such as those that define the components of a model category, tend to admit \emph{functorial} factorizations in the following sense.

\begin{defn}\label{defn:fun-fact} A \textbf{functorial factorization} on a category $\cM$ is given by a functor $\cM^\2 \to \cM^\3$ from the category of arrows in $\cM$ to the category of composable pairs of arrows in $\cM$ that defines a section to the  composition functor $\circ \colon \cM^\3 \to \cM^\2$. The action of this functor on objects in $\cM^\2$ (which are arrows, displayed vertically) and morphisms in $\cM^\2$ (which are commutative squares) is displayed below:
\[
\begin{tikzcd} X \arrow[d, "f"'] \arrow[r, "u"] & Z \arrow[d, "g"] \\ Y \arrow[r, "v"'] & W \end{tikzcd}
\qquad \mapsto\qquad
\begin{tikzcd} 
X \arrow[dd, bend right, "f"'] \arrow[r, "u"] \arrow[d, "Lf"] & Z \arrow[d, "Lg"'] \arrow[dd, bend left, "g"] \\ Ef \arrow[r, "{E(u,v)}"] \arrow[d, "Rf"]  & Eg \arrow[d, "Rg"'] \\ Y \arrow[r, "v"'] & W \end{tikzcd}
\]
This data is equivalently presented by a pair of endofunctors $L,R \colon \cM^\2 \rightrightarrows \cM^\2$ satisfying compatibility conditions
relative to the domain and codomain projections $\dom,\cod \colon \cM^\2 \rightrightarrows \cM$, namely that  \[\dom L = \dom, \quad \cod R =\cod,\quad \mathrm{and}\quad E:=\cod L = \dom R\] as functors $\cM^\2 \to \cM$.
\end{defn}

The functoriality of Definition \ref{defn:fun-fact} is with respect to (horizontal) composition of squares and is encapsulated most clearly by the functor $E$ which carries a square $(u,v)$ to the morphism $E(u,v)$ between the objects through which $f$ and $g$ factor. Even without assuming the existence of functorial factorizations, in any category with a weak factorization system $(\sL,\sR)$, commutative squares may be factored into a square between morphisms in $\sL$ on top of a square between morphisms in $\sR$
\[
\begin{tikzcd} X \arrow[d, "f"'] \arrow[r, "u"] & Z \arrow[d, "g"] \\ Y \arrow[r, "v"'] & W \end{tikzcd}
\qquad \mapsto\qquad
\begin{tikzcd} 
X \arrow[dd, bend right, "f"'] \arrow[r, "u"] \arrow[d, "\ell \in \sL"] & Z \arrow[d, "\sL\ni\ell'"'] \arrow[dd, bend left, "g"] \\ E \arrow[r, dashed, "{e}"] \arrow[d, "r \in \sR"]  & F \arrow[d, "\sR\ni r'"'] \\ Y \arrow[r, "v"'] & W \end{tikzcd}
\]
with the dotted horizontal morphism defined by lifting  $\ell$ against $r'$. These factorizations will not be strictly functorial because the solutions to the lifting problems postulated by \ref{defn:wfs}\eqref{itm:wfs-lift} are not unique. However, for either of the weak factorizations systems in a model category, any two solutions to a lifting problem are \emph{homotopic} in a sense defined by Quillen which appears as Definition \ref{defn:htpy}. As homotopic maps are identified in the homotopy category, this means that any model category has functorial factorizations up to homotopy, which suffices for most purposes.\footnote{While the derived functors constructed in Corollary \ref{cor:quillen-derived} make use of explicit point-set level functorial factorizations, their total derived functors in the sense of Definition \ref{defn:total} are well-defined without strict functoriality.} Despite the moral sufficiency of the standard axioms, for economy of language we tacitly assume that our models categories have functorial factorizations henceforth and take comfort in the fact that it seems to be exceedingly difficult to find model categories that fail to satisfy this condition.

\subsection{The homotopy relation on arrows}\label{ssec:htpy-relation}

Our aim now is to define Quillen's homotopy relation, which will be used to construct a relatively concrete model $\h\cM_{\mathrm{cf}}$ for  the homotopy category of the model category $\cM$, which is equivalent to the Gabriel-Zisman category of fractions $\cM[\we^{-1}]$ but provides better control over the sets of morphisms between each pair of objects. Quillen's key observation appears as Proposition \ref{prop:structured-we}, which shows that the weak equivalences between objects of $\cM$ that are both \emph{fibrant} and \emph{cofibrant}, in a sense to be defined momentarily, are more structured, 
 always admitting a homotopy inverse for a suitable notion of homotopy. The homotopy relation is respected by pre- and post-composition, which means that $\h\cM_{\mathrm{cf}}$ may be defined simply to be the category of fibrant-cofibrant objects and homotopy classes of maps. In this section, we give all of these definitions. In \S\ref{ssec:htpy-cat-of-model}, we construct the category $\h\cM_{\mathrm{cf}}$ sketched above and prove its equivalence with the category of fractions $\cM[\we^{-1}]$.

\begin{defn}\label{defn:fib-cof} An object $X$ in a model category $\cM$ is \textbf{fibrant} just when the unique map $X \to *$ to the terminal object is a fibration and \textbf{cofibrant} just when the unique map $\emptyset \to X$ from the initial object is a cofibration.
\end{defn}

Objects that are not fibrant or cofibrant can always be replaced by weakly equivalent objects that are by factoring the maps to the terminal object or from the initial object, as appropriate.

\begin{exc}[fibrant and cofibrant replacement]\label{exc:fib-cof} Assuming the functorial factorizations of \S\ref{ssec:functoriality}, define a \textbf{fibrant replacement} functor $R \colon \cM \to \cM$ and a \textbf{cofibrant replacement} functor $Q \colon \cM \to \cM$ equipped with natural weak equivalences
\[ \begin{tikzcd}[column sep=small] \eta \colon \id_\cM \arrow[r, Rightarrow,  "\sim" pos=.4] & R \end{tikzcd} \qquad \mathrm{and} \qquad \begin{tikzcd}[column sep=small] \epsilon \colon Q \arrow[r, Rightarrow, "\sim" pos=.4] & \id_\cM.\end{tikzcd}\]
\end{exc}

Applying both constructions, one obtains a \textbf{fibrant-cofibrant} replacement of any object $X$ as either $RQX$ or $QRX$. In the diagram
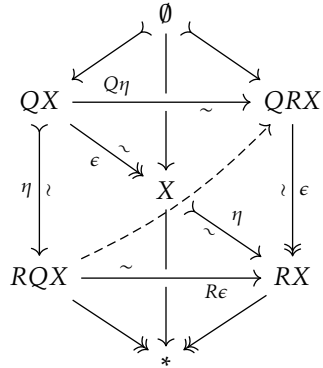
\[
\begin{tikzcd}
& \emptyset \arrow[dr, tail] \arrow[dl, tail] \arrow[dd] \\ QX \arrow[rr, crossing over,  "Q \eta" near start, "\sim"' near end] \arrow[dd, tail, "\eta"', "\wr"] \arrow[dr, two heads, "\epsilon"', "\sim" sloped] & & QRX \arrow[dd, two heads, "\epsilon", "\wr"'] \\ & X \arrow[dd] \arrow[dr, tail, "\eta", "\sim"' sloped] \\ RQX \arrow[rr, crossing over, "R \epsilon"' near end, "\sim" near start] \arrow[dr, two heads] \arrow[uurr, dashed, bend right=10] & & RX \arrow[dl, two heads] \\ & {\ast}
\end{tikzcd}
\] the middle square commutes because its two subdivided triangles do, by naturality of the maps $\eta$ and $\epsilon$ of Exercise \ref{exc:fib-cof}. In particular, this induces a direct comparison weak equivalence $RQX \overset{\sim}{\longrightarrow} QRX$ by lifting $\eta_{QX}$ against $\epsilon_{RX}$.

\begin{exc}\label{exc:replacing-maps}
Show that any map in a model category may be replaced, up to a zig zag of weak equivalences, by one between fibrant-cofibrant objects that moreover may be taken to be either a fibration or a cofibration, as desired.\footnote{Exercise \ref{exc:replacing-maps} reveals that the notions of ``cofibration'' and ``fibration'' are not homotopically meaningful: up to isomorphism in $\cM[\we^{-1}]$, any map in a model category can be taken to be either a fibration or a cofibration.}
\end{exc}

The reason for our particular interest in the subcategory of fibrant-cofibrant objects in a model category is that between such objects, the weak equivalences become more structured, coinciding with a class of ``homotopy equivalences'' in a sense we now define.

\begin{defn}\label{defn:path-cyl}
Let $A$ be an object in a model category. A \textbf{cylinder object} for $A$ is given by a factorization of the fold map
\[
\begin{tikzcd}[row sep=small] A \coprod A \arrow[dr, tail, "{(i_0, i_1)}"'] \arrow[rr, "{(1_A,1_A)}"] &  & A \\ &  \fun{cyl}(A) \arrow[ur, two heads, "\sim" sloped, "q"'] 
\end{tikzcd}
\]
into a cofibration followed by a trivial fibration. Dually, a \textbf{path object} for $A$ is given by any factorization of the diagonal map
\[
\begin{tikzcd}[row sep=small]
 & \fun{path}(A) \arrow[dr, two heads, "{(p_0,p_1)}"] \\ A \arrow[ur, tail, "j", "\sim"' sloped] \ar[rr, "{(1_A,1_A)}"']& & A \times A
\end{tikzcd}
\]
into a trivial cofibration followed by a fibration.
\end{defn}

\begin{rmk} For many purposes it suffices to drop the hypotheses that the maps in the cylinder and path object factorizations are cofibrations and fibrations, and retain only the hypothesis that the second and first factors, respectively, are weak equivalences. The standard terminology for the cylinder and path objects defined here adds the adulation ``very good.'' But since ``very good'' cylinder and path objects always exist, we eschew the usual convention and adopt these as the default notions.
\end{rmk}

\begin{defn}\label{defn:htpy} Consider a parallel pair of maps $f,g \colon A \rightrightarrows B$ in a model category. A \textbf{left homotopy} $H$ from $f$ to $g$ is given by a map from a cylinder object of $A$ to $B$ extending $(f,g) \colon A \coprod A \to B$
\[
\begin{tikzcd} A \arrow[r, "i_0"] \arrow[dr, "f"'] & \fun{cyl}(A) \arrow[d, "H"] & A \arrow[l, "i_1"'] \arrow[dl,"g"] \\ & B
\end{tikzcd}
\]
in which case one writes $f \sim_\ell g$ and says that $f$ and $g$ are \textbf{left homotopic}. 

A \textbf{right homotopy} $K$ from $f$ to $g$ is given by a map from $A$ to a path object for $B$ extending $(f,g) \colon A \to B \times B$
\[
\begin{tikzcd}
& A \arrow[dl, "f"'] \arrow[dr, "g"] \arrow[d, "K"] \\ B & \fun{path}(B) \arrow[l, "p_0"] \arrow[r, "p_1"'] & B
\end{tikzcd}
\]
in which case one writes $f \sim_r g$ and says that $f$ and $g$ are \textbf{right homotopic}.
\end{defn}

\begin{exc}\label{exc:cylinder-inclusions} Prove that the endpoint inclusions $i_0,i_1 \colon A \rightrightarrows \fun{cyl}(A)$ into a cylinder object are weak equivalences always and also cofibrations if $A$ is cofibrant. Conclude that if $f \sim_\ell g$ then $f$ is a weak equivalence if and only if $g$ is. Dually, the projections $p_0,p_1\colon \fun{path}(B) \rightrightarrows B$ are weak equivalences always and also fibrations if $B$ is fibrant, and if $f \sim_r g$ then $f$ is a weak equivalence if and only if $g$ is.
\end{exc}

A much more fine-grained analysis of the left and right homotopy relations is presented in a classic expository paper ``Homotopy theories and model categories'' of Dwyer and Spalinski \cite{DS}. Here we focus on only the essential facts for understanding the homotopy relation on maps between cofibrant and fibrant objects.

\begin{prop}\label{prop:htpy}
If $A$ is cofibrant and $B$ is fibrant then left and right homotopy define equivalence relations on the set $\Hom(A,B)$ of arrows and moreover these relations coincide.
\end{prop}

In light of Proposition \ref{prop:htpy}, we say that maps $f,g \colon A \rightrightarrows B$ from a cofibrant object to a fibrant one are \textbf{homotopic} and write $f \sim g$ to mean that they are left or equivalently right homotopic.

\begin{proof}
The left homotopy relation is reflexive and symmetric without any cofibrancy or fibrancy hypotheses on the domains or codomains. To prove transitivity, consider a pair of left homotopies $H \colon \fun{cyl}(A) \to B$ from $f$ to $g$ and $K \colon \fun{cyl}'(A) \to B$ from $g$ to $h$, possibly constructed using different cylinder objects for $A$.
By cofibrancy of $A$ and Exercise \ref{exc:cylinder-inclusions}, a new cylinder object $\fun{cyl}''(A)$ for $A$ may be constructed by factoring the map from the following pushout $C$ to $A$
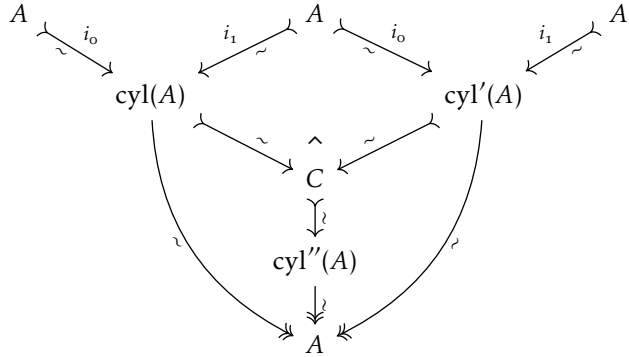
\[
\begin{tikzcd}[row sep=1.5em]
A \arrow[dr, tail, "i_0","\sim"' sloped] & & A \arrow[dl, tail, "i_1"',"\sim" sloped ]  \arrow[dd, phantom, "\rotatebox{135}{$\lrcorner$}" very near end] \arrow[dr, tail, "i_0","\sim"' sloped]  & & A\arrow[dl, tail, "i_1"',"\sim" sloped ]  \\
& \fun{cyl}(A) \arrow[dr, tail, "\sim" sloped] \arrow[dddr, bend right, two heads, "\sim"' sloped] & & \fun{cyl}'(A) \arrow[dl, tail, "\sim"' sloped] \arrow[dddl, bend left, two heads, "\sim" sloped] \\ & & C \arrow[d, tail, "\wr"] & & 
\\ & & \fun{cyl}''(A) \arrow[d, two heads, "\wr"] \\ & & A
\end{tikzcd}
\]
The homotopies $H$ and $K$ define a cone under the pushout diagram inducing a map $H \cup_A K \colon C \to B$. By fibrancy of $B$, this map may be extended along the trivial cofibration $C \cwto \fun{cyl}''(A)$ to define a homotopy $\fun{cyl}''(A) \to B$ from $f$ to $h$. This proves that left homotopy is an equivalence relation.

Finally we argue that if $H \colon \fun{cyl}(A) \to B$ defines a left homotopy from $f$ to $g$ then $f \sim_r g$. 
The desired right homotopy from $f$ to $g$ is constructed as the restriction of the displayed lift
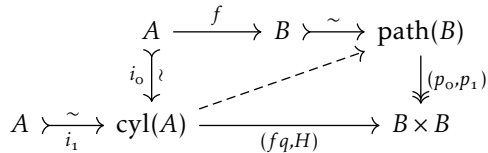
\[
\begin{tikzcd}
& A \arrow[d, tail, "i_0"', "\wr"] \arrow[r, "f"] & B \arrow[r, tail, "\sim"] & \fun{path}(B) \arrow[d, two heads, "{(p_0,p_1)}"d] \\ A \arrow[r, tail, "i_1"', "\sim"] & \fun{cyl}(A) \arrow[rr, "{(fq, H)}"'] \arrow[urr, dashed] & & B \times B
\end{tikzcd}
\]
along the endpoint inclusion $i_1$. The remaining assertions are dual to ones already proven.
\end{proof}

Moreover, the homotopy relation is respected by pre- and post-composition.

\begin{prop}\label{prop:htpy-comp} Suppose $f,g \colon A \rightrightarrows B$ are left or right homotopic maps and consider any maps $h \colon A' \to A$ and $k\colon B \to B'$. Then $kfh, kgh \colon A' \rightrightarrows B'$ are again left or right homotopic, respectively.
\end{prop}
\begin{proof}
%Suppose $f \sim_\ell g$ via a left homotopy $H \colon \fun{cyl}(A) \to B$. Then $kH \colon \fun{cyl}(A) \to B'$ defines a left homotopy from $kf$ to $kg$. 
By lifting the endpoint inclusion $(i_0,i_1) \colon A' \coprod A' \cto \fun{cyl}(A')$ against the projection $\fun{cyl}(A) \fwto A$ --- or by functoriality of the cylinder construction in the sense discussed in \S\ref{ssec:functoriality} --- there is a map $\fun{cyl}(h) \colon \fun{cyl}(A') \to \fun{cyl}(A)$. Now for any left homotopy $H \colon \fun{cyl}(A) \to B$ from $f$ to $g$, the horizontal composite then defines a left homotopy $kfh \sim_\ell kgh$.
\[
\begin{tikzcd} &  {A' \coprod A'} \arrow[dl, tail] \arrow[r, "h \sqcup h"] & A \coprod A \arrow[dl, tail] \arrow[dr, "{(f,g)}"] \\ {\fun{cyl}(A')} \arrow[r, dashed, "{\fun{cyl}(h)}"] \arrow[d, two heads, "\wr"'] & \fun{cyl}(A) \arrow[d, two heads, "\wr"'] \arrow[rr,  "H"'] & & B \arrow[r, "k"] & {B'} \\ {A'} \arrow[r, "h"] & A
\end{tikzcd} \qedhere
\]
\end{proof}
%OLD PROOF ASSUMING co/fibrancy: By Proposition \ref{prop:htpy}, there exists a right homotopy $A \to \fun{path}(B)$ from $f$ to $g$, which may be restricted along $h$ to define a  right homotopy $fh \sim_r gh$. By Proposition \ref{prop:htpy} again, there is now a left homotopy $\fun{cyl}(A') \to B$ from $fh$ to $gh$ that may be composed with $k$ to define a left homotopy from $kfh$ to $kgh$. Thus $kfh \sim kgh$.

%It follows from Definition \ref{defn:model-cat} that the weak equivalences are closed under retracts \cite[7.8]{JT} and also have the two-of-six property \cite[??]{DHKS}. {\color{darkred} or prove both locally} The closure conditions for the weak equivalences are the key ingredients in the proof of the model categorical generalization of the classical Whitehead theorem, that a map between CW complexes is a weak homotopy equivalence if and only if it is a homotopy equivalence:\footnote{Unlike the Whitehead theorem, both directions are non-trivial.}

\begin{prop}\label{prop:structured-we}
Let $f \colon A \to B$ be a map between objects that are both fibrant and cofibrant. Then $f$ is a weak equivalence if and only if it has a homotopy inverse.
\end{prop}
\begin{proof}
For both implications we make use of the fact that any map between fibrant-cofibrant objects may be factored as a trivial cofibration followed by a fibration
\[
\begin{tikzcd}
& P \arrow[dr, two heads, "p"] \\ A \arrow[rr, "f"'] \arrow[ur, tail, "j"', "\sim" sloped] & & B
\end{tikzcd}
\]
through an object that is again fibrant-cofibrant. If $f$ is a weak equivalence then $p$ is a trivial fibration. We argue that any trivial fibration $p$ between fibrant-cofibrant objects extends to a deformation retraction: admitting a right inverse that is also a left homotopy inverse. A dual argument proves that the trivial cofibration $j$ admits a left inverse that is also a right homotopy inverse. These homotopy equivalences compose in the sense of Proposition \ref{prop:htpy-comp}  to define a homotopy inverse for $f$.

If $p$ is a trivial fibration, then cofibrancy of $B$ implies that it admits a right inverse $i$. The homotopy constructed in the lifting problem
\[
\begin{tikzcd} 
\emptyset \arrow[r] \arrow[d, tail] & P \arrow[d, two heads,  "p"] & & P \coprod P \arrow[rr, "{(1_P,ip)}"] \arrow[d, tail] & & P \arrow[d, two heads, "p", "\wr"'] \\ B \arrow[r, equals] \arrow[ur, dashed, "i"] & B & &  {\fun{cyl}(P)} \arrow[r, two heads, "\sim"] \arrow[urr, dashed] & P \arrow[r, "p"] & B
\end{tikzcd}
\] 
proves that $ip \sim 1_P$ as desired.

For the converse we suppose that $f$ admits a homotopy inverse $g$. To prove that $f$ is a weak equivalence it suffices to prove that $p$ is a weak equivalence. A right inverse $i$ to $p$ may be found by lifting the endpoint of the homotopy $H \colon fg \sim 1_B$
\[
\begin{tikzcd} 
& B  \arrow[r, "g"] \arrow[d, tail, "i_0", "\wr"'] & A \arrow[r, "j"] & P \arrow[d, two heads, "p"] \\  B \arrow[r, tail, "i_1", "\sim"'] &\fun{cyl}(B) \arrow[rr, "H"] \arrow[urr, dashed] & & B
\end{tikzcd}
\]
and then restricting this lift along $i_1$. By construction this section $i$ is homotopic to $jg$.  The argument of the previous paragraph applies to the trivial cofibration $j$ to prove that it has a left inverse and right homotopy inverse $q$. Composing the homotopies $1_P\sim jq$, $i \sim jg$,  and $gf \sim 1_A$ we see that
\[ ip \sim ip jq = i f q \sim jgfq \sim jq \sim 1_P \] By Exercise \ref{exc:cylinder-inclusions} we conclude that $ip$ is a weak equivalence. But by construction $p$ is a retract of $ip$
\[
\begin{tikzcd}
P \arrow[d, two heads, "p"'] \arrow[r,equals] & P \arrow[d, "ip", "\wr"'] \arrow[r, equals] & P \arrow[d, two heads, "p"] \\ B \arrow[r, "i"] & P \arrow[r, "p"] & B
\end{tikzcd}
\]
so it follows from the retract stability of the weak equivalences \cite[7.8]{JT} that $p$ is a weak equivalence as desired.
\end{proof}

\subsection{The homotopy category of a model category}\label{ssec:htpy-cat-of-model}

In this section, we prove that the  category of fractions $\cM[\we^{-1}]$, defined by formally inverting the weak equivalences, is equivalent to the category $\h\cM_{\mathrm{cf}}$ of fibrant-cofibrant objects and homotopy classes of maps. Our proof appeals to the universal property of Proposition \ref{prop:fractions-UP}, which characterizes those categories that are \emph{isomorphic} to the category of fractions. For categories to be isomorphic, they must have the same object sets, so we define a larger version of the homotopy category $\Ho\cM$, which has the same objects as $\cM[\we^{-1}]$ and is  equivalent to its full subcategory $\h\cM_{\mathrm{cf}}$.

\begin{defn}\label{defn:nice-htpy-cat} For any model category $\cM$, there is a category $\h\cM_{\mathrm{cf}}$ whose:
\begin{itemize}
\item objects are the fibrant-cofibrant objects in $\cM$ and 
\item in which the set of morphisms from $A$ to $B$ is taken to be the set of homotopy classes of maps 
\[ [A,B] := \Hom(A,B)_{/\sim}.\]
\end{itemize}
\end{defn}

Proposition \ref{prop:htpy-comp} ensures that composition in $\h\cM_{\mathrm{cf}}$ is well-defined.

\begin{defn} The \textbf{homotopy category} $\Ho\cM$ of a model category $\cM$ is defined by applying the (bijective-on-objects, fully faithful) factorization %\footnote{The category of categories has an \textbf{orthogonal factorization system} --- a weak factorization system with unique solutions to lifting problems and hence also unique factorizations up to isomorphism ---  whose left class is comprised of the bijective-on-objects functors and whose right class is comprised of the fully faithful functors. The functorial factorization is constructed by taking the left functor to be identity-on-objects and defining the hom-sets in its codomain to agree with the hom-sets in the target category, making the right functor fully faithful by construction.}
 to the composite functor
\begin{equation}\label{eq:ho-cat-defn}
\begin{tikzcd}[row sep=tiny] \cM \arrow[r, "RQ"] \arrow[dr, dashed, "\mathrm{bij~obj}"', "\gamma" near end] & \cM_{\mathrm{cf}} \arrow[r, "\pi"] & \h\cM_{\mathrm{cf}} \\ & \Ho\cM \arrow[ur, dashed, "\mathrm{f+f}"', "\nu" near start] 
\end{tikzcd}
\end{equation}
 That is the objects in $\Ho\cM$ are the objects in $\cM$ and \[ \Ho\cM(A,B) := \cM(RQA,RQB)_{/\sim}.\]
\end{defn}

\begin{exc}[$\Ho\cM\simeq\h\cM_{\mathrm{cf}}$]$\quad$
\begin{enumerate}
\item Verify that the category $\h\cM_{\mathrm{cf}}$ defined by Definition \ref{defn:nice-htpy-cat} is equivalent to the full subcategory of $\Ho\cM$ spanned by the fibrant-cofibrant objects of $\cM$. 
\item Show that every object in $\cM$ is isomorphic in $\Ho\cM$  to a fibrant-cofibrant object.
\item Conclude that the categories $\Ho\cM$ and $\h\cM_{\mathrm{cf}}$ are equivalent.
\end{enumerate}
\end{exc}

\begin{thm}[Quillen]\label{thm:equivalent-ho} For any model category $\cM$, the category of fractions $\cM[\we^{-1}]$ obtained by formally inverting the weak equivalences is isomorphic to the homotopy category $\Ho\cM$.
\end{thm}
\begin{proof} 
We will prove that $\gamma \colon \cM \to \Ho\cM$ satisfies the universal property of Proposition \ref{prop:fractions-UP} that characterizes the category of fractions $\cM[\we^{-1}]$. First we must verify that $\gamma$ inverts the weak equivalences. The functor $RQ$ carries weak equivalences in $\cM$ to weak equivalences between fibrant-cofibrant objects. Proposition \ref{prop:structured-we} then implies that these admit homotopy inverses and thus become isomorphisms in $\h\cM_{\mathrm{cf}}$. This proves that the composite horizontal functor of \eqref{eq:ho-cat-defn} inverts the weak equivalences. By fully-faithfulness of $\nu$, the functor $\gamma \colon \cM \to \Ho\cM$ also inverts weak equivalences. 

It remains to verify that any functor  $F \colon \cM \to \cE$ that inverts the weak equivalences factors uniquely through $\gamma$
\[
\begin{tikzcd}[row sep=tiny] \cM \arrow[rr, "F"] \arrow[dr, "\gamma"'] & & \cE \\ & \Ho\cM \arrow[ur, dashed, "\bar{F}"']
\end{tikzcd}
\]
Since $\gamma$ is identity-on-objects, we must define $\bar{F}$ to agree with $F$ on objects. Recall that the fibrant and cofibrant replacement functors come with natural weak equivalences $\epsilon_X \colon QX \wto X$ and $\eta_X \colon X \wto RX$. Because $F$ inverts weak equivalences, these natural transformations define a natural isomorphism $\alpha \colon F \To FRQ$ of functors from $\cM$ to $\cE$. By the definition $\Ho\cM(X,Y) := \cM(RQX,RQY)_{/\sim}$, the morphisms from $X$ to $Y$ in $\Ho\cM$ correspond to homotopy classes of morphisms from $RQX$ to $RQY$ in $\cM$. Choose any representative $h \colon RQX \to RQY$ for the corresponding homotopy class of maps and define its image to be the composite
\[
\begin{tikzcd}
\bar{F}h \colon FX \arrow[r, "\alpha_X"] & FRQX \arrow[r, "Fh"] & FRQY \arrow[r, "\alpha_Y^{-1}"] & FY
\end{tikzcd}
\]
This is well-defined because if $h \sim h'$ then there exists a left homotopy so that $H i_0 = h$ and $H i_1 = h$, where $i_0$ and $i_1$ are both sections to a common weak equivalence (the projection from the cylinder). Since $F$ inverts weak equivalences, $Fi_0$ and $Fi_1$ are both right inverses to a common isomorphism, so it follows that $Fi_0 = Fi_1$ and hence $Fh = Fh'$. 

Functoriality of $\bar{F}$ follows immediately from naturality of $\alpha$ and functoriality of $FRQ$. To see that $\bar{F}\gamma = F$, recall that for any $f \colon X \to Y$ in $\cM$, $\gamma(f)$ is defined to be the map in $\Ho\cM(X,Y)$ represented by the homotopy class $RQf \colon RQX \to RQY$. By naturality of $\alpha$, $\bar{F}\gamma(f) = Ff$, so that the triangle of functors commutes.

Finally, to verify that $\bar{F}$ is unique observe that from the following commutative diagram in $\cM$ any map $h \in \Ho\cM(X,Y)$, as represented by the map on the left below, is isomorphic in $\Ho\cM$ to a map in the image of $\gamma$, the vertical morphism displayed on the right:
\[
\begin{tikzcd} RQX \arrow[d, "h"']& QRQX \arrow[r, "\eta_{QRQX}"] \arrow[d, "Qh"] \arrow[l, "\epsilon_{RQX}"'] & RQRQX \arrow[d, "RQh = \gamma(h)"] \\ 
RQY & QRQY \arrow[r, "\eta_{QRQY}"] \arrow[l, "\epsilon_{RQY}"'] & RQRQY
\end{tikzcd}
\]
Since the image of $\bar{F}$ on the right vertical morphism is uniquely determined and the top and bottom morphisms are isomorphisms, the image of $\bar{F}$ on the left vertical morphism is also uniquely determined.
\end{proof}

\begin{rmk} The universal property of $\h\cM_{\mathrm{cf}}$ is slightly weaker than the universal property described in Proposition \ref{prop:fractions-UP} for the category of fractions $\cM[\we^{-1}]$. For any category $\cE$, restriction along $\gamma \colon \cM \to \h\cM_{\mathrm{cf}}$ defines a fully faithful embedding $\Fun(\h\cM_{\mathrm{cf}},\cE) \hookrightarrow \Fun(\cM,\cE)$ and equivalence onto the full subcategory of functors from $\cM$ to $\cE$ that carry weak equivalences to isomorphisms. The difference is that a given homotopical functor on $\cM$ may not factor strictly through $\h\cM_{\mathrm{cf}}$ but may only factor up to natural isomorphism. In practice, this presents no serious difficulty.
\end{rmk}

As a corollary, it is now easy to see that the only maps inverted by the localization functor are weak equivalences. By Lemma \ref{lem:moral-we-defn}, this proves that the weak equivalences in a model category have all of the closure properties enumerated at the outset of this section.

\begin{thm}[{\cite[5.1]{quillen}}]\label{thm:model-saturation} A morphism in a model category $\cM$ is inverted by the localization functor \[ \cM \to \cM[\we^{-1}]\] if and only if it is a weak equivalence.
\end{thm}
\begin{proof}
Cofibrantly and then fibrantly replacing the map it suffices to consider a map between fibrant-cofibrant objects. By Theorem \ref{thm:equivalent-ho} we may prove this result for $\cM_{\mathrm{cf}} \to \h\cM_{\mathrm{cf}}$ instead. But now this is clear by construction: since morphisms in $\h\cM_{\mathrm{cf}}$ are homotopy classes of maps, the isomorphisms are the homotopy equivalences, which coincide exactly with the weak equivalences between fibrant-cofibrant objects by Proposition \ref{prop:structured-we}.
\end{proof}

\subsection{Quillen's model structure on simplicial sets}\label{ssec:simp-sets}

We conclude this section with a prototypical example. Quillen's original model structure is borne by the category of \textbf{simplicial sets}, presheaves on the category $\DDelta$ of finite non-empty ordinals $[n] = \{0 < 1 < \cdots < n\}$ and order-preserving maps. A \textbf{simplicial set} $X\colon\DDelta^\op \to \cat{Set}$ is a graded set $\{X_n\}_{n \geq 0}$ --- where elements of $X_n$ are called ``$n$-simplices'' ---  equipped with dimension-decreasing ``face'' maps $X_n \to X_m$ arising from monomorphisms $[m] \rightarrowtail [n] \in \DDelta$ and dimension-increasing ``degeneracy'' maps $X_m \to X_n$ arising from epimorphisms $[n] \twoheadrightarrow [m] \in \DDelta$. An $n$-simplex has $n+1$ codimension-one faces, each of which avoids one of its $n+1$ vertices. 

There is a \textbf{geometric realization} functor $|-| \colon \cat{sSet} \to \cat{Top}$ that produces a topological space $|X|$ from a simplicial set $X$ by gluing together topological $n$-simplices for each non-degenerate $n$-simplex along its lower-dimensional faces. The simplicial set represented by $[n]$ defines the standard $n$-simplex $\Delta^n$. Its boundary $\partial\Delta^n$ is the union of its codimension-one faces, while a horn $\Lambda^n_k$ is the further subspace formed by omitting the face opposite the vertex $k \in [n]$.

\begin{thm}[Quillen]\label{thm:sset-model} The category $\cat{sSet}$ admits a model structure whose:
\begin{itemize}
\item weak equivalences are those maps $f \colon X \to Y$ that induce a weak homotopy equivalence $f \colon |X| \to |Y|$ on geometric realizations
\item cofibrations are monomorphisms
\item fibrations are the \textbf{Kan fibrations}, which are characterized by the left lifting property with respect to the set of all horn inclusions
\[
\begin{tikzcd}\Lambda^n_k \arrow[r] \arrow[d, tail, "\wr"'] & X \arrow[d, two heads] \\ \Delta^n \arrow[r] \arrow[ur, dashed] & Y
\end{tikzcd}
\]
\end{itemize}
\end{thm}

All objects are cofibrant. The fibrant objects are the \textbf{Kan complexes}, those simplicial sets in which all horns can be filled. The fibrant objects are those simplicial sets that most closely resemble topological  spaces. In particular, two vertices in a Kan complex lie in the same path component if and only if they are connected by a single 1-simplex, with may be chosen to point in either direction. By Proposition \ref{prop:structured-we} a weak equivalence between Kan complexes is a homotopy equivalence where the notion of homotopy is defined with respect to the interval $\Delta^1$ using $\Delta^1 \times X$ as a cylinder object or $X^{\Delta^1}$ as a path object.

Quillen's model category of simplicial sets is of interested because, one the one hand, the category of simplicial sets is very well behaved and, on the other hand, the geometric realization functor defines an ``equivalence of homotopy theories'': in particular, the homotopy category of simplicial sets gives another model for the homotopy category of spaces. To explain this, we turn our focus to derived functors and derived equivalences between model categories, the subject of \S\ref{sec:derived}.

\section{Derived functors between model categories}\label{sec:derived}

Quillen's model category axioms allow us to conjure a homotopy relation between parallel maps in any model category, whatever the objects of that category might be. For this reason, model categories are often regarded as ``abstract homotopy theories.''  We will now zoom out to consider functors comparing such homotopy theories.

More generally, we might consider functors between homotopical categoriesequipped with weak equivalences that at least satisfy the two-of-three property. A great deal of the subtlety in ``category theory up to weak equivalence'' has to do with the fact that functors between homotopical categories need not necessarily preserve weak equivalences. In the case where a functor fails to preserve weak equivalence the next best hope is that it admits a universal approximation by a functor which does, where the approximation is either ``from the left'' or ``from the right.'' Such approximations are referred as \emph{left} or \emph{right derived functors}. 

The universal properties of left or right derived functors exists at the level of homotopy categories though the derived functors of greatest utility, and the ones that are most easily constructed in practice, can be constructed at the ``point-set level.'' One of the selling points of  Quillen's theory of model categories is that they highlight classes of functors---the left or right Quillen functors---whose left or right derived functors can be constructed in a uniform way making the passage to total derived functors pseudofunctorial. However, it turns out a full model structure is not necessary for this construction; morally speaking, all that matters for the specification of derived functors is the weak equivalences.

In  \S\ref{ssec:derived}, we give a non-standard and in our view greatly improved presentation of the theory of derived functors guided by a recent axiomatization of Dwyer-Hirschhorn-Kan-Smith \cite{DHKS} paired with a result of Maltsiniotis \cite{maltsiniotis}. The key point of difference is that we give a much stronger definition of what constitutes a derived functor than the usual one. In \S\ref{ssec:quillen-fun} we introduce left and right Quillen functors between model categories and show that such functors have a left or right derived functor satisfying this stronger property. Then in \S\ref{ssec:derived-adj}, we see that the abstract theory of this stronger class of derived functors is considerably better than the theory of the weaker ones. A highlight is an efficient expression of the properties of composite or adjoint derived functors proven by Shulman \cite{shulman-comparing} and reproduced as Theorem \ref{thm:double-ho}.

In \S\ref{ssec:monoidal-model}, we extend the theory of derived functors to allow functors of two variables with the aim of proving that the homotopy category of spaces is cartesian closed, inheriting an internal hom defined as the derived functor of the point-set level mapping spaces. Implicit in our approach to the proof of this statement is a result promised at the end of \S\ref{ssec:simp-sets}. In \S\ref{ssec:quillen-equiv}, we define a precise notion of equivalence between abstract homotopy theories encoded by model categories, which specializes to establish an equivalence between the homotopy theory of spaces and the homotopy theory of simplicial sets. Finally, in \S\ref{ssec:extending} we briefly sketch the connection between \emph{homotopical algebra} and \emph{homological algebra} by considering suitable model structures appropriate for a theory of derived functors between chain complexes.

\subsection{Derived functors and equivalence of homotopy theories}\label{ssec:derived}

As a warning to the reader, this definition of a derived functor is stronger than the usual one in two ways:
\begin{itemize}
\item We explicitly require our derived functors to be defined ``at the point-set level'' rather than simply as functors between homotopy categories.
\item We require the universal property of the corresponding ``total derived functors'' between homotopy categories to define \emph{absolute} Kan extensions.
\end{itemize}

Before defining our derived functors we should explain the general meaning of absolute Kan extensions.

\begin{defn}\label{defn:Kanext} A \textbf{left Kan extension} of $F \colon \cC \to \cE$ along $K \colon \cC \to \cD$ is a functor $\Lan_KF\colon \cD \to \cE$ together with a natural transformation $\eta\colon F \To \Lan_KF \cdot  K$ such that for any other such pair $(G \colon \cD \to \cE, \gamma \colon F \To GK)$, $\gamma$ factors uniquely through $\eta$ as illustrated.\footnote{Writing $\alpha$ for the natural transformation $\Lan_KF \To G$, the right-hand pasting diagrams express the equality $\gamma = \alpha K \cdot \eta$, i.e., that $\gamma$ factors as $\begin{tikzcd}[ampersand replacement=\&] F \arrow[r, Rightarrow,  "\eta"] \& {\Lan_K F \cdot K} \arrow[r, Rightarrow, "{\alpha K}"] \& GK \end{tikzcd}$.} 
\[
\begin{tikzcd} \cC \arrow[rr, "F"] \arrow[dr, "K"']  & \arrow[d, phantom, "\scriptstyle{\Downarrow \eta}"] & \cE
& \cC \arrow[rr, "F"] \arrow[dr, "K"']  & \arrow[d, phantom, "\scriptstyle{\Downarrow \gamma}"] & \cE \arrow[dr, phantom, "="] &  \cC \arrow[rr, "F"] \arrow[dr, "K"']  & \arrow[dl, phantom, "\scriptstyle{\Downarrow \eta}" near start] & \cE
 \\ & \cD \arrow[ur, dashed, "{\Lan_KF}"'] & & {~} &  \cD \arrow[ur, "G"'] && {~} & \cD \arrow[ur, bend left, "{\Lan_KF}" description] \arrow[ur, phantom, "\scriptstyle {\exists !}{\Downarrow}" near start] \arrow[ur, bend right, "G"']
\end{tikzcd}
\]
Dually, a \textbf{right Kan extension} of $F \colon \cC \to \cE$ along $K \colon \cC \to \cD$  is a functor $\Ran_K F \colon \cD \to \cE$ together with a natural transformation $\epsilon \colon \Ran_KF \cdot K\To F$ such that for any $(G \colon \cD \to \cE, \delta \colon GK \To F)$, $\delta$ factors uniquely through $\epsilon$ as illustrated.
\[
\begin{tikzcd} \cC \arrow[rr, "F"] \arrow[dr, "K"']  & \arrow[d, phantom, "\scriptstyle{\Uparrow \epsilon}"] & \cE
& \cC \arrow[rr, "F"] \arrow[dr, "K"']  & \arrow[d, phantom, "\scriptstyle{\Uparrow \delta}"] & \cE \arrow[dr, phantom, "="] &  \cC \arrow[rr, "F"] \arrow[dr, "K"']  & \arrow[dl, phantom, "\scriptstyle{\Uparrow \epsilon}" near start] & \cE
 \\ & \cD \arrow[ur, dashed, "{\Ran_KF}"'] & & {~} &  \cD \arrow[ur, "G"'] && {~} & \cD \arrow[ur, bend left, "{\Ran_KF}" description] \arrow[ur, phantom, "\scriptstyle {\exists !}{\Uparrow}" near start] \arrow[ur, bend right, "G"']
\end{tikzcd}
\]
A left or right Kan extension  is \textbf{absolute} if for any functor $H \colon \cE \to \cF$, the whiskered composite $(H\Lan_KF \colon \cD \to \cE, H\eta)$ or $(H\Ran_KF\colon\cD \to\cE,H\epsilon)$ defines the left or right Kan extension of $HF$ along $K$.
\end{defn}

A functor between homotopical categories is a \textbf{homotopical functor} if it preserves the classes of weak equivalences, or carries the weak equivalences in the domain to isomorphisms in the codomain in the case where no class of weak equivalences is specified. Derived functors can be understood as universal homotopical approximations to a given functor in a sense we now define.

\begin{defn}[{derived functors}]\label{defn:derived-functor} Let $\cM$ and $\cK$ be homotopical categories with weak equivalences satisfying the two-of-three property and localization functors $\gamma \colon \cM \to \Ho\cM$ and $\delta \colon \cK \to \Ho\cK$.
\begin{itemize}
\item A \textbf{left derived functor} of $F \colon \cM \to \cK$ is a homotopical functor $\LL F \colon \cM \to \cK$  equipped with a natural transformation $\lambda \colon \LL F \To F$ so that $\delta \LL F$  and $\delta \lambda \colon \delta  \LL F \To \delta F$ define an absolute \emph{right} Kan extension of $\delta F$ along $\gamma$.
\[
\begin{tikzcd} \cM \arrow[r, bend left, "F"] \arrow[r, bend right, "\LL F"'] \arrow[r, phantom, "\scriptstyle\Uparrow\lambda"] & \cK 
\end{tikzcd}
 \qquad \leftrightsquigarrow \qquad
 \begin{tikzcd}
 \cM \arrow[dr, phantom,"\scriptstyle\Uparrow\delta \lambda"] \arrow[r, "F"]  \arrow[d, "\gamma"'] & \cK \arrow[d, "\delta"]  \\ \Ho\cM \ar[r, "\delta\LL F"'] & \Ho\cK
 \end{tikzcd}\]
\item A \textbf{right derived functor} of $F \colon \cM \to \cK$ is a homotopical functor $\RR F \colon \cM \to \cK$  equipped with a natural transformation $\rho \colon F \To \RR F$ so that $\delta \RR F$  and $\delta \rho \colon \delta  F \To  \delta \RR F$ define an absolute \emph{left} Kan extension of $\delta F$ along $\gamma$.
\[
\begin{tikzcd} \cM \arrow[r, bend left, "F"] \arrow[r, bend right, "\RR F"'] \arrow[r, phantom, "\scriptstyle\Downarrow\rho"] & \cK 
\end{tikzcd}
 \qquad \leftrightsquigarrow \qquad
 \begin{tikzcd}
 \cM \arrow[dr, phantom,"\scriptstyle\Downarrow\delta \rho"] \arrow[r, "F"]  \arrow[d, "\gamma"'] & \cK \arrow[d, "\delta"]  \\ \Ho\cM \ar[r, "\delta\RR F"'] & \Ho\cK
 \end{tikzcd}\]
\end{itemize}
\end{defn}

\begin{rmk}
Absolute Kan extensions are in particular ``pointwise'' Kan extensions, these being the left or right Kan extensions that are preserved by representable functors. The pointwise left or right Kan extensions are those definable as colimits or limits in the target category \cite[\S 6.3]{riehl-context}, so it is somewhat surprising that these conditions are appropriate to require for functors valued in homotopy categories, which have few limits and colimits.\footnote{With the exception of products and coproducts, the so-called ``homotopy limits'' and ``homotopy colimits'' introduced in \S\ref{sec:holim} do not define limits and colimits in the homotopy category.}
\end{rmk}

As a consequence of Proposition \ref{prop:fractions-UP}, the homotopical functors \[\delta \LL F, \delta \RR F \colon \cM \rightrightarrows \Ho\cK\] factor uniquely through $\gamma$ and so may be equally regarded as functors \[\delta \LL F, \delta \RR F \colon \Ho \cM \rightrightarrows\Ho\cK,\] as appearing in the displayed diagrams of Definition \ref{defn:derived-functor}.

\begin{defn}[total derived functors]\label{defn:total}
The  \textbf{total left} or \textbf{right derived functors} of $F$ 
are the functors \[\delta \LL F, \delta \RR F \colon \Ho \cM \rightrightarrows\Ho\cK,\] defined as absolute Kan extensions in Definition \ref{defn:derived-functor} and henceforth  denoted by \[\bL F, \bR F \colon \Ho\cM \rightrightarrows \Ho\cK.\]
\end{defn}

There is a common setting in which derived functors exist and admit a simple construction. Such categories have a subcategory of ``good'' objects on which the functor of interest becomes homotopical and a functorial reflection into this full subcategory. The details are encoded in the following axiomatization due to \cite{DHKS} and exposed in \cite{shulman-homotopy}, though we diverge from their terminology to more thoroughly ground our intuition in the model categorical case.

\begin{defn}\label{defn:leftdef} A \textbf{left deformation} on a homotopical category $\cM$ consists of an endofunctor $Q$ together with a natural weak equivalence $q \colon Q \stackrel{\sim}{\To} 1$.
\end{defn}

The functor $Q$ is necessarily homotopical. Let $\cM_{\mathrm{c}}$ be any full subcategory of $\cM$ containing the image of $Q$. The inclusion $\cM_{\mathrm{c}} \to \cM$ and the left deformation $Q \colon \cM \to \cM_{\mathrm{c}}$ induce an equivalence between $\Ho\cM$ and $\Ho\cM_{\mathrm{c}}$. As our notation suggests, any model category $\cM$ admits a left deformation defined by cofibrant replacement. Accordingly, we refer to $\cM_{\mathrm{c}}$ as the subcategory of \textbf{cofibrant objects}, trusting the reader to understand that when we have not specified any model structures, Quillen's technical definition is not what we require.

\begin{defn}\label{defn:funleftdef} A functor $F \colon \cM \to \cK$ between homotopical categories is \textbf{left deformable} if there exists a left deformation on $\cM$ such that $F$ is homotopical on an associated subcategory of cofibrant objects.
\end{defn}

Our first main result proves that left deformations can be used to construct left derived functors. The basic framework of left deformations was set up by \cite{DHKS} while the observation that such derived functors are absolute Kan extensions is due to \cite{maltsiniotis}.

\begin{thm}[{\cite[41.2-5]{DHKS}, \cite{maltsiniotis}}]\label{thm:leftderived} If $F\colon \cM \to \cK$ has a left deformation $q \colon Q \stackrel{\sim}{\To} 1$, then $\LL F = FQ$ is a left derived functor of $F$.
\end{thm}
\begin{proof} 
Write $\delta \colon \cK \to \Ho\cK$ for the localization. To show that $(FQ, Fq)$ is a point-set left derived functor, we must show that the functor $\delta FQ$ and natural transformation $\delta Fq \colon \delta FQ \To \delta F$ define a right Kan extension. The verification makes use of Proposition \ref{prop:fractions-UP}, which identifies the functor category $\Fun(\Ho\cM,\Ho\cK)$ with the full subcategory of $\Fun(\cM,\Ho\cK)$ spanned by the homotopical functors. Suppose $G \colon \cM \to \Ho\cK$ is homotopical and consider $\alpha \colon G \To \delta F$. Because $G$ is homotopical and $q \colon Q \To 1_\cM$ is a natural weak equivalence, $Gq \colon GQ \To G$ is a natural isomorphism. Using naturality of $\alpha$, it follows that $\alpha$ factors through $\delta F q$ as
\[
\begin{tikzcd}
G \arrow[r, Rightarrow, "{(Gq)^{-1}}"] & GQ \arrow[r, Rightarrow, "{\alpha_Q}"] & \delta F Q \arrow[r, Rightarrow, "{\delta Fq}"] & \delta F
\end{tikzcd} 
\]

To prove uniqueness, suppose $\alpha$ factors as
\[
\begin{tikzcd}
 G \arrow[r, Rightarrow, "{\beta}"] & \delta FQ \arrow[r, Rightarrow,"{\delta Fq}"] & \delta F
 \end{tikzcd}\]
Naturality of $\beta$ provides a commutative square of natural transformations: 
\[
\begin{tikzcd}
GQ \arrow[r, Rightarrow,"{\beta_Q}"]\arrow[d, Rightarrow,"{Gq}"'] & \delta F Q^2 \arrow[d, Rightarrow,"{\delta FQq}"] \\ G \arrow[r, Rightarrow, "{\beta}"'] & \delta FQ
\end{tikzcd} \] 
Because $q$ is a natural weak equivalence and the functors $G$ and $\delta FQ$ are homotopical, the vertical arrows are natural isomorphisms, so $\beta$ is determined by $\beta_Q$. This restricted natural transformation is  uniquely determined: $q_Q$ is a natural weak equivalence between objects in the image of $Q$. Since $F$ is homotopical on this subcategory, this means that $Fq_Q$ is a natural weak equivalence and thus $\delta Fq_Q$ is an isomorphism, so $\beta_Q$ must equal the composite of the inverse of this natural isomorphism with $\alpha_Q$. 

Finally, to show that this right Kan extension is absolute, our task is to show that for any functor $H \colon \Ho\cK \to \cE$, the pair $(H\delta FQ, H\delta Fq)$ again defines a right Kan extension. Note that $(Q,q)$ also defines a left deformation for $H\delta F$, simply because the functor $H\colon \Ho\cK \to \cE$ preserves isomorphisms. The argument just given now demonstrates that $(H\delta FQ, H\delta Fq)$ is a right Kan extension, as claimed.
\end{proof}

\subsection{Quillen functors}\label{ssec:quillen-fun}

We'll now introduce important classes of functors between model categories that will admit derived functors.

\begin{defn} A functor between model categories is
\begin{itemize}
\item \textbf{left Quillen} if it preserves cofibrations, trivial cofibrations, and cofibrant objects, and
\item  \textbf{right Quillen} if it preserves fibrations, trivial fibrations, and fibrant objects.
\end{itemize}
\end{defn}

Most left Quillen functors are ``cocontinuous,'' preserving all colimits, while most right Quillen functors are ``continuous,'' preserving all limits; when this is the case there is no need to separately assume that cofibrant or fibrant objects are preserved. Importantly, cofibrant replacement defines a left deformation for any left Quillen functor, while fibrant replacement defines a right deformation for any right Quillen functor, as we now demonstrate:

\begin{lem}[Ken Brown's lemma]\label{lem:kb}  $\quad$
\begin{enumerate}
\item\label{itm:kb-i} Any map between fibrant objects in a model category can be factored as a right inverse to a trivial fibration followed by a  fibration.
\begin{equation}\label{eq:kb-construction}
\begin{tikzcd}
& P \arrow[dr, two heads, "p"] \arrow[dl, bend right=45, two heads, "q"',  "\sim" sloped] \\ A \arrow[rr, "f"'] \arrow[ur, tail, "j"', "\sim" sloped] & & B
\end{tikzcd}
\end{equation}
\item\label{itm:kb-ii} Let $F \colon \cM \to \cK$ be a functor from a model category to a category with a class of weak equivalences satisfying the two-of-three property. If $F$ carries trivial fibrations in $\cM$ to weak equivalences in $\cK$, then $F$ carries all weak equivalences between fibrant objects in $\cM$ to weak equivalences in $\cK$.
\end{enumerate}
\end{lem}
\begin{proof}
For \eqref{itm:kb-i}, given any map $f \colon A \to B$ factor its graph $(1_A,f) \colon A \to A \times B$ as a trivial cofibration $j$ followed by a fibration $r$:
\[
\begin{tikzcd}
& & & B \arrow[dr, two heads] \\
A \arrow[rrru, bend left, "f"] \arrow[rrrd, bend right, equals] \arrow[r, "j", tail, "\sim"'] & P \arrow[urr, two heads, "p"] \arrow[drr, two heads, "q"', "\sim" sloped] \arrow[r, two heads, "r" description] & A \times B \arrow[ur, two heads, "\pi_B"'] \arrow[dr, two heads, "\pi_A"] \arrow[rr, phantom, "\rotatebox{45}{$\lrcorner$}" very near start] & & 1 \\
& & & A \arrow[ur, two heads]
\end{tikzcd}
\]
Since $A$ and $B$ are fibrant, the dual of Lemma \ref{lem:closure} implies that the product projections are fibrations, and thus the composite maps $p$ and $q$ are fibrations. By the two-of-three property, $q$ is also a weak equivalence.

To prove \eqref{itm:kb-ii} assume that $f \colon A \to B$ is a weak equivalence in $\cM$ and construct the factorization \eqref{eq:kb-construction}. It follows from the two-of-three property that $p$ is also a trivial fibration, so by hypothesis both $Fp$ and $Fq$ are weak equivalences in $\cK$. Since $Fj$ is right inverse to $Fq$, it must also be a weak equivalence, and thus the closure of weak equivalences under composition implies that $Ff$ is a weak equivalence as desired.
\end{proof}

Specializing Theorem \ref{thm:leftderived} we then have:

\begin{cor}\label{cor:quillen-derived} The left derived functor of any left Quillen functor $F$ exists and is given by $\LL F := FQ$ while the right derived functor of any right Quillen functor $G$ exists and is given by $\RR G := GR$, where $Q$ and $R$ denote any cofibrant and fibrant replacement functors, respectively.
\end{cor}

\subsection{Derived composites and derived adjunctions}\label{ssec:derived-adj}

Left and right Quillen functors frequently occur in adjoint pairs, in which case the left adjoint is left Quillen if and only if the right adjoint is right Quillen:

\begin{defn}\label{defn:quillen-adj} Consider an adjunction between a pair of model categories.
\begin{equation}\label{eq:quillen-adj}
\begin{tikzcd}
\cM \arrow[r, bend left, "F"] \arrow[r, phantom, "\perp"] & \cN \arrow[l, bend left, "G"]
\end{tikzcd}
\end{equation}
Then the following are equivalent, defining a \textbf{Quillen adjunction}.
\begin{enumerate}
\item\label{itm:quillen-adj-i} The left adjoint $F$ is left Quillen.
\item\label{itm:quillen-adj-ii} The right $G$ is right Quillen.
\item\label{itm:quillen-adj-iii} The left adjoint preserves cofibrations and the right adjoint preserves fibrations.
\item\label{itm:quillen-adj-iv} The left adjoint preserves trivial cofibrations and the right adjoint preserves trivial fibrations.
\end{enumerate}
\end{defn}

\begin{exc} Justify the equivalence of Definition \ref{defn:quillen-adj}\eqref{itm:quillen-adj-i}--\eqref{itm:quillen-adj-iv} by proving:
\begin{enumerate}
\item In the presence of any adjunction \eqref{eq:quillen-adj} the lifting problem displayed below left in $\cN$ admits a solution if and only if the transposed lifting problem displayed below right admits a solution in $\cM$.
\[
\begin{tikzcd}
FA \arrow[d, "F\ell"'] \arrow[r, "f^\sharp"] & X \arrow[d, "r"] &  & A \arrow[r, "f^\flat"] \arrow[d, "\ell"'] & GX \arrow[d, "Gr"] \\ FB \arrow[r, "g^\sharp"'] \arrow[ur, dashed, "k^\sharp"] & Y & & B \arrow[ur, dashed, "k^\flat"] \arrow[r, "g^\flat"'] & GY
\end{tikzcd}
\]
\item Conclude that if $\cM$ has a weak factorization system $(\sL,\sR)$ while $\cN$ has a weak factorization system $(\sL',\sR')$ then $F$ preserves the left classes if and only if $G$ preserves the right classes.
\end{enumerate}
\end{exc}

Importantly, the total left and right derived functors of a Quillen pair form an adjunction between the appropriate homotopy categories. 

\begin{thm}[Quillen {\cite[I.3]{quillen}}]\label{thm:derived-adj} If 
\[
\begin{tikzcd}
\cM \arrow[r, bend left, "F"] \arrow[r, phantom, "\perp"] & \cN \arrow[l, bend left, "G"]
\end{tikzcd}
\]
 is a Quillen adjunction, then the total derived functors form an adjunction
 \[
\begin{tikzcd}
\Ho\cM \arrow[r, bend left, "\bL F"] \arrow[r, phantom, "\perp"] & \Ho\cN \arrow[l, bend left, "\bR G"]
\end{tikzcd}
\]
\end{thm}

A particularly elegant proof of Theorem \ref{thm:derived-adj} is due to Maltsiniotis. Once the strategy is known, the details are elementary enough to be left as an exercise:

\begin{exc}[{\cite{maltsiniotis}}]\label{exc:derivedadj} Use the fact that the total derived functors of a Quillen pair  $F \dashv G$ define \emph{absolute} Kan extensions to prove that ${\bL}F \dashv {\bR}G$. Conclude that Theorem \ref{thm:derived-adj} applies more generally to any pair of adjoint functors that are deformable in the sense of Definition \ref{defn:funleftdef} \cite[44.2]{DHKS}. 
\end{exc}

A double categorical theorem of Shulman \cite{shulman-comparing} consolidates the adjointness of the total derived functors of a Quillen adjunction, the pseudo-functoriality of the construction of total derived functors of Quillen functors, and one further result about functors that are simultaneously left and right Quillen into a single statement. A \textbf{double category} is a category internal to $\cat{Cat}$: it has a set of objects, a category of horizontal morphisms, a category of vertical morphisms, and a set of squares that are composable in both vertical and horizontal directions, defining the arrows in a pair of categories with the horizontal and vertical morphisms as objects, respectively \cite{kelly-street}.

For instance, $\mathbb{C}\cat{at}$ is the double category of categories, functors, functors, and natural transformations inhabiting squares and pointing southwest. There is another double category $\mathbb{M}\cat{odel}$ whose objects are model categories, whose vertical morphisms are left Quillen functors, whose horizontal morphisms are right Quillen functors, and whose squares are natural transformations pointing southwest. The following theorem and a generalization, with deformable functors in place of Quillen functors \cite[8.10]{shulman-comparing}, is due to Shulman.

\begin{thm}[{\cite[7.6]{shulman-comparing}}]\label{thm:double-ho}
The map that sends  a model category to its homotopy category and a left or right Quillen functor to its total left or right derived functor defines a double pseudofunctor $\Ho\colon \mathbb{M}\cat{odel} \to \mathbb{C}\cat{at}$.
\end{thm}

The essential content of the pseudo-functoriality statement is that the composite of the left derived functors of a pair of left Quillen functors is coherently naturally weakly equivalent to the left derived functor of their composite. 
Explicitly, given a composable pair of left Quillen functors $\begin{tikzcd}
\cM \arrow[r, "F"] & \cL \arrow[r, "G"] & \cK
\end{tikzcd}$, the map
\[ \LL G \circ \LL F := GQ \circ FQ \xrightarrow{G \epsilon_{FQ}} GFQ =: \LL (GF),\] defines a comparison natural transformation. Since $Q \colon \cM \to \cM_{\mathrm{c}}$ lands in the subcategory of cofibrant objects and $F$ preserves cofibrant objects, $\epsilon_{FQ} \colon QFQ \To FQ$ is a weak equivalence between cofibrant objects. Lemma \ref{lem:kb}\eqref{itm:kb-ii} then implies that $G\epsilon_{FQ} \colon GQFQ \to GFQ$ defines a natural weak equivalence $\LL{G} \circ \LL{F} \to \LL{GF}$. Given a composable triple of left Quillen functors, there is a commutative square of natural weak equivalences $\LL{H} \circ \LL{G}\circ \LL{F} \to \LL(H \circ G \circ F)$. If we compose with the Gabriel-Zisman localizations to pass to homotopy categories and total left derived functors these coherent natural weak equivalences become coherent natural isomorphisms, defining the claimed pseudofunctor. 

Quillen adjunctions are encoded in the double category $\mathbb{M}\cat{odel}$ as ``conjoint'' relationships between vertical and horizontal 1-cells; in this way Theorem \ref{thm:double-ho} subsumes Theorem \ref{thm:derived-adj}. Similarly, functors that are simultaneously left and right Quillen are presented as vertical and horizontal ``companion'' pairs. The double pseudofunctoriality of Theorem \ref{thm:double-ho} contains a further result:  if a functor is both left and right Quillen, then its total left and right derived functors are isomorphic.

\subsection{Monoidal and enriched model categories}\label{ssec:monoidal-model}

If $\cM$ has a model structure and a monoidal structure it is natural to ask that these be compatible in some way, but what sort of compatibility should be required? In the most common examples, the monoidal product is \emph{closed} --- that is, the functors $A \otimes -$ and $- \otimes A$ admit right adjoints\footnote{Very frequently a monoidal structure is symmetric, in which case these functors are naturally isomorphic, and a single right adjoint suffices.}  and consequently preserve colimits in each variable separately. This situation is summarized and generalized by the notion of a two-variable adjunction, which we introducing using notation that will suggest the most common examples.

\begin{defn}\label{defn:two-var}
A triple of bifunctors 
\[ \cK \times \cL \xrightarrow{\otimes} \cM\rlap{{\,},} \quad \cK^\op \times \cM \xrightarrow{\{,\}} \cL\rlap{{\,},} \quad \cL^\op \times \cM \xrightarrow{\Map}\cK\] equipped with a natural isomorphism \[  \cM(K \otimes L,M) \cong \cL(L,\{K,M\})\cong \cK(K,\Map(L,M)) \] defines a \textbf{two-variable adjunction}. 
\end{defn}

\begin{ex} A symmetric monoidal category is \textbf{closed} just when its monoidal product $- \otimes - \colon \cV \times \cV \to \cV$ defines the left adjoint of a two-variable adjunction
 \[  \cV(A \otimes B, C) \cong \cV(B, \Map(A,C)), \cV(A,\Map(B,C)),\]
 the right adjoint $\Map \colon\cV^\op \times \cV \to \cV$ defining an \textbf{internal hom}.
\end{ex}

\begin{ex}\label{ex:tensored-cotensored} A category $\cM$ that is enriched over a monoidal category is \textbf{tensored} and \textbf{cotensored} just when the enriched hom functor $\Map \colon \cM^\op \times \cM \to \cV$ is one of the right adjoints of a two-variable adjunction
\[ \cM(V \otimes M, N) \cong \cM(M, \{V,N\}) \cong \cV(V, \Map(M,N)),\]
the other two adjoints defining the \textbf{tensor} $V \otimes M$ and \textbf{cotensor} $\{V,N\}$
of an object $V \in \cV$ with objects $M,N \in \cM$.\footnote{As stated this definition is a little too weak: one needs to ask in addition that (i) the tensors are associative relative to the monoidal product in $\cV$, (ii) dually that the cotensors are associative relative to the monoidal product in $\cV$, and (iii) that the two-variable adjunction is enriched in $\cV$. Any of these three conditions implies the other two.}
\end{ex}

A Quillen two-variable adjunction is a two-variable adjunction in which the left adjoint is a left Quillen bifunctor while the right adjoints are both right Quillen bifunctors, any one of these conditions implying the other two. To state these definitions, we must introduce the following construction. The ``pushout-product'' of a bifunctor  $-\otimes- \colon \cK \times \cL \to \cM$
defines a bifunctor $-\leib\otimes-\colon \cK^\2 \times \cL^\2 \to \cM^\2$ that we refer to as the ``Leibniz tensor'' (when the bifunctor $\otimes$ is called a ``tensor''). The ``Leibniz cotensor''  and ``Leibniz hom'' \[\widehat{\{-,-\}}\colon (\cK^\2)^\op \times \cM^\2 \to \cL^\2 \qquad \mathrm{and} \qquad \widehat{\Map}(-,-) \colon (\cL^\2)^\op \times \cM^\2 \to \cK^\2\]
are defined dually, using pullbacks in $\cL$ and $\cK$ respectively. 

\begin{defn}[the Leibniz construction] Given a bifunctor $-\otimes- \colon \cK \times \cL \to \cM$ valued in a category with pushouts, the \textbf{Leibniz tensor} of a map $k \colon I \to J$ in $\cK$ and a map $\ell \colon A \to B$ in $\cL$ is the map $k \leib\otimes \ell$ in $\cM$ induced by the pushout diagram below-left:
\[
\begin{tikzcd} I \otimes A \arrow[r, "I \otimes \ell"] \arrow[d, "k \otimes A"'] \arrow[dr, phantom,  "\ulcorner"  near end] & I \otimes B \arrow[d] \arrow[ddr, bend left, "k \otimes B"] & & {\{ J, X\}} \arrow[drr, bend left, "{\{k,X\}}"] \arrow[ddr, bend right, "{\{J, m\}}"'] \arrow[dr, dashed, "{\widehat{\{ k,m\}}}" description]  \\ J \otimes A \arrow[r] \arrow[drr, bend right, "J \otimes \ell"'] & \bullet \arrow[dr, dashed, "k \leib\otimes \ell" description] & & & \bullet \arrow[r] \arrow[d] \arrow[dr, phantom, "\lrcorner" near start] & {\{ I, X\}} \arrow[d, "{\{I,m\}}"] \\ & & J \otimes B & & {\{ J, Y\}} \arrow[r, "{\{k,Y\}}"'] & {\{ I,Y\}}
\end{tikzcd}
\] In the case of a bifunctor $\{-,-\} \colon \cK^\op \times \cM \to \cL$ contravariant in one of its variables valued in a category with pullbacks, the \textbf{Leibniz cotensor} of a map $k\colon I \to J$ in $\cK$ and a map $m \colon X \to Y$ in $\cM$ is the map $\widehat{\{ k,m\}}$ induced by the pullback diagram above right.
\end{defn}

\begin{prop}\label{prop:leibniz-facts} The Leibniz construction preserves: 
\begin{enumerate}
\item\label{itm:leib-iso} structural isomorphisms: a natural isomorphism
\[ X \ast (Y \otimes Z) \cong (X \times Y) \square Z\] between suitably composable bifunctors extends to a natural isomorphism
\[ f \leib\ast (g \leib\otimes h) \cong (f \leib\times g) \leib\square h\] between the corresponding Leibniz products;
\item\label{itm:leib-adj} adjointness: if $(\otimes, \{,\},\Map)$ define a two-variable adjunction, then the Leibniz bifunctors $(\leib\otimes, \widehat{\{,\}}, \widehat{\Map})$ define a two-variable adjunction between the corresponding arrow categories;
%\item\label{itm:leib-lifting} lifting properties (a two-variable adjunction induces a two-variable adjunction on arrow categories),
\item\label{itm:leib-cocont} colimits in the arrow category: if $\otimes \colon \cK \times \cL \to\cM$ is cocontinuous in either variable, then so is $\leib\otimes \colon \cK^\2 \times \cL^\2 \to \cM^\2$;
\item\label{itm:leib-pushout} pushouts: if $\otimes \colon \cK \times \cL \to\cM$ is cocontinuous in its second variable, and if $g'$ is a pushout of $g$, then $f \leib\otimes g'$ is a pushout of $f \leib\otimes g$;
\item\label{itm:leib-comp} composition, in a sense: the Leibniz tensor $f \leib\otimes (h \cdot g)$ factors as a composite of a pushout of $f \leib\otimes g$ followed by $f \leib\otimes h$
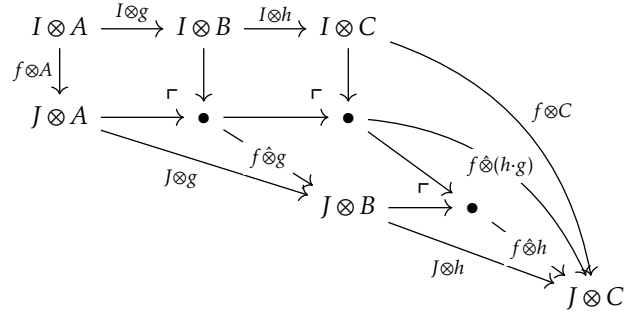
\[
\begin{tikzcd}
I \otimes A \arrow[d, "f \otimes A"'] \arrow[r, "I \otimes g"] \arrow[dr, phantom, "\ulcorner" very near end] & I \otimes B \arrow[r, "I \otimes h"] \arrow[d] \arrow[dr, phantom, "\ulcorner" very near end] & I \otimes C \arrow[d] \arrow[dddrr, bend left, "f \otimes C"] \\
J \otimes A \arrow[r] \arrow[drr, "J \otimes g"'] & \bullet \arrow[dr, "f \leib\otimes g" description] \arrow[r] \arrow[drr, phantom, "\ulcorner" very near end] & \bullet \arrow[dr] \arrow[ddrr, bend left, "f \leib\otimes (h \cdot g)" description] \\  & & J \otimes B \arrow[drr,  "J \otimes h"'] \arrow[r]  & \bullet \arrow[dr, "f \leib\otimes h" description] \\ & &  & & J \otimes C
\end{tikzcd}
\]
\item\label{itm:leib-cell} cell complex structures: if $f$ and $g$ may be presented as cell complexes with cells $f_\alpha$ and $g_\beta$, respectively, and if $\otimes$ is cocontinuous in both variables, then $f \leib\otimes g$ may be presented as a cell complex with cells $f_\alpha \leib\otimes g_\beta$.
\end{enumerate}
\end{prop}

Proofs of these assertions and considerably more details are given in \cite[\S\S 4-5]{RV0}.

\begin{exc}\label{exc:leibniz-lifting} 
Given a two variable adjunction as in Definition \ref{defn:two-var} and classes of maps $\sA, \sB, \sC$  in $\cK, \cL, \cM$, respectively, prove that the following lifting properties are equivalent \[ \sA\hat{\otimes}\sB \boxslash \sC \quad \Leftrightarrow \quad \sB \boxslash \widehat{\{\sA,\sC\}} \quad \Leftrightarrow\quad \sA \boxslash \widehat{\Map}(\sB,\sC).\] 
Here $\sA\hat{\otimes}\sB \boxslash \sC$, for instance, asserts that maps in $\sC$ have the right lifting property with respect to each map in $\sA\hat{\otimes}\sB$.
\end{exc}

Exercise \ref{exc:leibniz-lifting} explains the equivalence between the following three equivalent definitions of a Quillen two-variable adjunction.

\begin{defn}\label{defn:quillen-two-var} A two-variable adjunction 
\[
\cV \times \cM \xrightarrow{\otimes} \cN, \quad \cV^\op \times \cN \xrightarrow{\{-,-\}} \cM, \quad \cM^\op \times \cN \xrightarrow{\Map} \cV\]
between model categories $\cV$, $\cM$, and $\cN$ defines a \textbf{Quillen two-variable adjunction} if any, and hence all, of the following equivalent conditions are satisfied:
\begin{enumerate}
\item The functor $\hat{\otimes} \colon \cV^\2 \times \cM^\2 \to \cN^\2$ carries any pair comprised of a cofibration in $\cV$ and a cofibration in $\cM$ to a cofibration in $\cN$ and furthermore this cofibration is a weak equivalence if either of the domain maps are.
\item The functor $\widehat{\{-,-\}} \colon (\cV^\2)^\op \times \cN^\2 \to \cM^\2$ carries any pair comprised of a 
cofibration in $\cV$ and a fibration in $\cN$ to a fibration in $\cN$ and furthermore this fibration is a weak equivalence if either of the domain maps are.
\item The functor $\widehat{\Map} \colon (\cM^\2)^\op \times \cN^\2 \to \cV^\2$ carries any pair comprised of a 
cofibration in $\cM$ and a fibration in $\cN$ to a fibration in $\cV$ and furthermore this fibration is a weak equivalence if either of the domain maps are.
\end{enumerate}
\end{defn}

\begin{exc} Prove that if $-\otimes - \colon \cV \times \cM \to \cN$ is a left Quillen bifunctor and $V \in \cV$ is cofibrant then $\cV \otimes -\colon \cM \to\cN$ is a left Quillen functor.
\end{exc}

Quillen's axiomatization of the additional properties enjoyed by his model structure on the category of simplicial sets  has been generalized by Hovey \cite[\S 4.2]{hovey}.

\begin{defn}\label{defn:monmodelcat} A (\textbf{closed symmetric}) \textbf{monoidal model category} is a (closed symmetric) monoidal category $(\cV, \otimes, I)$ with a model structure so that the monoidal product and hom define a Quillen two-variable adjunction and furthermore so that the maps \begin{equation}\label{eq:monoidal-unit} QI \otimes v \to I \otimes v \cong v \qquad \mathrm{and} \qquad v \otimes QI \to v \otimes I \cong v\end{equation} are weak equivalences if $v$ is cofibrant.\footnote{If the monoidal product is symmetric then of course these two conditions are equivalent and if it is closed then they are also equivalent to a dual one involving the internal hom \cite[4.2.7]{hovey}.}
\end{defn}

Then

\begin{defn}\label{defn:Vmodelcat} If $\cV$ is a monoidal model category a $\cV$-\textbf{model category} is a model category $\cM$ that is tensored, cotensored, and $\cV$-enriched in such a way that $(\otimes,\{,\},\Map)$ is a Quillen two-variable adjunction and the maps \[ QI \otimes\, m \to I \otimes\, m \cong m \] are weak equivalences if $m$ is cofibrant.
\end{defn}

\begin{exc}\label{exc:unenriched-as-enriched} In a locally small category $\cM$ with products and coproducts the hom bifunctor is part of a two-variable adjunction:
\[ -\ast - \colon \cat{Set} \times \cM \to \cM, \quad \{-,-\} \colon \cat{Set}^\op \times \cM \to \cM, \quad \Hom \colon \cM^\op \times \cM \to \cat{Set}.\]
Equipping $\cat{Set}$ with the model structure whose weak equivalences are all maps, whose cofibrations are monomorphisms, and whose fibrations are epimorphisms, prove that
\begin{enumerate}
\item $\cat{Set}$ is a cartesian monoidal model category.
\item Any model category $\cM$ is a $\cat{Set}$-model category.
\end{enumerate}
\end{exc}

\begin{ex} Quillen's model structure of Theorem \ref{thm:sset-model} is a closed symmetric monoidal model category. The term \textbf{simplicial model category} refers to a model category enriched over this model structure.
\end{ex}

\begin{exc}\label{exc:kan-enriched-simp-model} Show that if $\cM$ is a simplicial model category then the full simplicial subcategory $\cM_{\mathrm{cf}}$ is Kan-complex enriched.
\end{exc}

The conditions \eqref{eq:monoidal-unit} on the cofibrant replacement of the monoidal unit are implied by the Quillen two-variable adjunction if the monoidal unit is cofibrant and  are necessary for the proof of  Theorem \ref{thm:hoVmonoidal}, which shows that the homotopy categories are again closed monoidal and enriched, respectively. 

\begin{thm}[{\cite[4.3.2,4]{hovey}}]\label{thm:hoVmonoidal} $\quad$
\begin{enumerate}
\item The homotopy category of a closed symmetric monoidal model category is a closed monoidal category with tensor and hom given by the derived adjunction
 \[ (\LL\otimes,\RR\Map,\RR\Map) \colon \Ho\cV \times \Ho\cV \to \Ho\cV \] and monoidal unit $QI$.
\item If $\cM$ is a $\cV$-model category, then $\Ho\cM$ is the underlying category of a $\Ho\cV$-enriched, tensored, and cotensored category  with enrichment given by the total derived two-variable adjunction \[ (\LL\otimes,\RR\{\},\RR\Map) \colon \Ho\cV \times \Ho\cM \to \Ho\cM. \]
\end{enumerate}
\end{thm}

In particular:

\begin{cor} 
The homotopy category of spaces is cartesian closed. Moreover, if $\cM$ is a simplicial model category, then $\Ho\cM$ is enriched, tensored, and cotensored over the homotopy category of spaces.
\end{cor}

\subsection{Quillen equivalences between homotopy theories}\label{ssec:quillen-equiv}

Two model categories present equivalent homotopy theories if there exists a finite sequence of model categories and a zig-zag of Quillen equivalences between them, in a sense we now define. A Quillen adjunction defines a Quillen equivalence just when the derived adjunction of Theorem \ref{thm:derived-adj} defines an \textbf{adjoint equivalence}: an adjunction with invertible unit and counit. There are several equivalent characterizations of this situation.

\begin{defn}[{\cite[\S I.4]{quillen}}]\label{defn:quillen-equiv} A Quillen adjunction between a pair of model categories.
\[
\begin{tikzcd}
\cM \arrow[r, bend left, "F"] \arrow[r, phantom, "\perp"] & \cN \arrow[l, bend left, "G"]
\end{tikzcd}
\]
defines a \textbf{Quillen equivalence} if any, and hence all, of the following equivalent conditions is satisfied:
\begin{enumerate}
\item The total left derived functor $\bL{F} \colon \Ho\cM \to \Ho\cN$ defines an equivalence of categories.
\item The total right derived functor $\bR{G} \colon \Ho\cN \to \Ho\cM$ defines an equivalence of categories.
\item For every cofibrant object $A \in \cM$ and every fibrant object $X \in \cN$, a map $f^\sharp\colon FA \to X$ is a weak equivalence in $\cN$ if and only if its transpose $f^\flat \colon A \to GX$ is a weak equivalence in $\cN$.
\item For every cofibrant object $A \in \cM$, the composite $A \to GFA \to GRFA$ of the unit with fibrant replacement is a weak equivalence in $\cM$, and for every fibrant object $X \in \cN$, the composite $FQGX \to FGX \to X$ of the counit with cofibrant replacement is a weak equivalence in $\cN$.
\end{enumerate}
\end{defn}

Famously, the formalism of Quillen equivalences enables a proof that the homotopy theory of spaces is equivalent to the homotopy theory of simplicial sets.

\begin{thm}[{Quillen \cite[\S II.3]{quillen}}]\label{thm:geo-sing-equiv} The homotopy theory of simplicial sets is equivalent to the homotopy theory of topological spaces via the geometric realization $\dashv$ total singular complex adjunction
\[ 
\begin{tikzcd}
\cat{sSet} \arrow[r, bend left, "{|-|}"] \arrow[r, phantom, "\perp"] & \cat{Top} \arrow[l, bend left, "\mathrm{Sing}"]
\end{tikzcd}
\]
\end{thm}

\subsection{Extending homological algebra to homotopical algebra}\label{ssec:extending}

Derived functors are endemic to homological algebra. Quillen's homotopical algebra can be understood to subsume classical homological algebra in the following sense. The category of chain complexes of modules over a fixed ring (or valued in an arbitrary abelian category) admits a homotopical structure where the weak equivalences are quasi-isomorphisms. Relative to an appropriately-defined model structure, the left and right derived functors of homological algebra can be viewed as special cases of the construction of derived functors of left or right Quillen functors in Corollary \ref{cor:quillen-derived} or in the more general context of  Theorem \ref{thm:leftderived}.

The following theorem describes an equivalent presentation of the homotopy theory just discussed.

\begin{thm}[{Schwede-Shipley after Dold-Kan}]\label{thm:dold-kan} The homotopy theory of simplicial modules over a commutative ring, with fibrations and weak equivalences as on underlying simplicial sets, is equivalent to the homotopy theory of non-negatively graded chain complexes of modules, as presented by the projective model structure whose weak equivalences are the quasi-isomorphisms, fibrations are the chain maps which are surjective in positive dimensions, and cofibrations are the monomorphisms with dimensionwise projective cokernel.
\end{thm}
\begin{proof} For details of the model structure on simplicial objects see \cite[II.4, II.6]{quillen} and on chain complexes see \cite[2.3.11, 4.2.13]{hovey}. The proof that the functors $\Gamma, N$ in the Dold-Kan equivalence are each both left and right Quillen equivalences can be found in \cite[\S 4.1]{schwede-shipley} or is safely left as an exercise to the reader.
\end{proof}
 
 The Dold-Kan Quillen equivalence of Theorem \ref{thm:dold-kan} suggests that simplicial methods might replace homological ones in non-abelian contexts. Let $\cM$ be any category of ``algebras'' such as monoids, groups, rings (or their commutative variants), or modules or algebras over a fixed ring; technically $\cM$ may be any category of models for a \emph{Lawvere theory} \cite{lawvere}, which specifies finite operations of any arity and relations between the composites of these operations.
 
 \begin{thm}[{Quillen \cite[\S II.4]{quillen}}] For any category $\cM$ of ``algebras'' --- a category of models for a Lawvere theory --- the category $\cM^{\DDelta^\op}$ of simplicial algebras admits a simplicial model structure whose
 \begin{itemize}
 \item weak equivalences are those maps that are weak homotopy equivalences on underlying simplicial sets
 \item fibrations are those maps that are Kan fibrations on underlying simplicial sets
 \item cofibrations are retracts of free maps.
 \end{itemize}
 \end{thm}
 
% \begin{defn} \textbf{Andr\'{e}-Quillen cohomology} \end{defn}

% {\color{darkblue} Also look at the intro to Chapter II of \cite{quillen}.}
 
% {\color{darkblue} If you want something else look at A.K. Bousfield, Localization of spaces with respect to homology, Topology 14 (1975), 133--150 as summarized in 11.5 of Dwyer-Spalinski: fibrant replacement as constructing the localization of a simplicial set relative to an arbitrary homology theory.}

%(from MO ``what are non-categorical applications of homotopical algebra?''; Qiaochu Yuan)

%The Cech nerve or a cover $U \to X$ gives a simplicial object that's like a free resolution. In the abelian case this gives Cech cohomology. In the non-abelian case each cocycle descrip at principle bundle?

\section{Homotopy limits and colimits}\label{sec:holim}

Limits and colimits provide fundamental tools for constructing new mathematical objects from existing ones, so it is important to understand these constructions in the homotopical context. There are a variety of possible meanings of a homotopical notion of limit or colimit including:
\begin{enumerate}
\item\label{itm:holim-i} limits or colimits in the homotopy category of a model category;
\item\label{itm:holim-ii} limit or colimit constructions that are ``homotopy invariant,'' with weakly equivalent inputs giving rise to weak\-ly equivalent outputs; 
\item\label{itm:holim-iii} derived functors of the limit or colimit functors; and finally
\item\label{itm:holim-iv} limits or colimits whose universal properties are (perhaps weakly) enriched over simplicial sets or topological spaces.
\end{enumerate}

We will explore all of these possibilities in turn. We begin in  \S\ref{ssec:hocat-lim} by observing that the homotopy category has few genuine limits and colimits but does have ``weak'' ones in the case where the category is enriched, tensored, and cotensored over spaces. For the reason explained in Remark \ref{rmk:holim-no-lim}, homotopy limits or colimits rarely satisfy condition \eqref{itm:holim-i}.

Then in \S\ref{ssec:holim}, we define homotopy limits and colimits as derived functors, which in particular give ``homotopy invariant'' constructions, and introduce hypotheses on the ambient model category that ensure that these homotopy limit and colimit functors always exist. In \S\ref{ssec:reedy} we consider particular diagram shapes, the so-called Reedy categories, for which homotopy limits and colimits exist in any model category. Finally, in \S\ref{ssec:weighted} we permit ourselves a tour through the general theory of weighted limits and colimits as a means of elucidating these results and introducing families of Quillen bifunctors that deserve to be better known. This allows us to finally explain the sense in which homotopy limits or colimits in a simplicial model category satisfy properties \eqref{itm:holim-ii}-\eqref{itm:holim-iv} and in particular have an enriched universal property which may be understood as saying they ``represent homotopy coherent cones'' over or under the diagram.

\subsection{Weak limits and colimits in the homotopy category}\label{ssec:hocat-lim}

Consider a category $\cM$ that is enriched over spaces --- either simplicial sets or topological spaces will do --- meaning that for each pair of objects $x,y$, there is a mapping space $\Map(x,y)$ whose points are the usual set $\cM(x,y)$ of arrows from $x$ to $y$. We may define a homotopy category of $\cM$ using the construction of Definition \ref{defn:nice-htpy-cat}.

\begin{defn} If $\cM$ is a simplicially enriched category its \textbf{homotopy category} $\h\cM$ has
\begin{itemize}
\item objects the same objects as $\cM$ and
\item hom-sets $\h\cM(x,y) := \pi_0\Map(x,y)$ taken to be the path components of the mapping spaces.
\end{itemize}
Thus, a morphism from $x$ to $y$ in $\h\cM$ is a homotopy class of vertices in the simplicial set $\Map(x,y)$, where two vertices are homotopic if and only if they can be connected by a finite zig-zag of 1-simplices.
\end{defn}

A product of a family of objects $m_\alpha$ in a category $\cM$ is given by a representation $m$ for the functor displayed on the right:  \[ \cM(-,m) \stackrel{\cong}{\longrightarrow} \prod_\alpha \cM(-,m_\alpha).\] By the Yoneda lemma,  a representation consists of an object $m \in \cM$ together with maps $m \to m_\alpha$ for each $\alpha$ that are universal in the sense that for any collection $x \to m_\alpha \in \cM$, each of these arrows factors uniquely along a common map $x \to m$. But if $\cM$ is enriched over spaces, we might instead require only that the triangles 
\begin{equation}\label{eq:homotopytriangle} 
\begin{tikzcd}  x \arrow[d, dashed, "{\exists}"'] \arrow[dr] & {~} \\ m \arrow[r] \arrow[ur, phantom, "\simeq" near start] & m_\alpha
\end{tikzcd}
\end{equation} commute ``up to homotopy'' in the sense of a path in the space $\Map(x,m_\alpha)$ whose underlying set of points is $\cM(x,m_\alpha)$. Now we can define the \textbf{homotopy product} to be an object $m$ equipped with a natural weak homotopy equivalence \[ \Map(x,m) \to \prod_\alpha \Map(x,m_\alpha)\] for each $x \in \cM$. Surjectivity on path components implies the existence and homotopy commutativity of the triangles \eqref{eq:homotopytriangle}.

\begin{exc} Use the fact that $\pi_0$ commutes with products and is homotopical to show, unusually for homotopy limits, that the homotopy product is a product in the homotopy category $\h\cM$. Similarly, a homotopy coproduct is a coproduct in the homotopy category. 
\end{exc}

For non-discrete diagram shapes, the homotopy category of a category enriched in spaces\footnote{Here we can take our enrichment over topological spaces or over simplicial sets, the latter being more general \cite[3.7.15-16]{riehl-cathtpy}.} 
 will no longer have genuine limits or colimits but in the presence of tensors in the colimit case and cotensors in the limit case it will have weak ones.
 
\begin{thm}[{\cite[11.1]{vogt}}]\label{thm:vogt-w-limits} If $\cM$ is cocomplete and also enriched and tensored over spaces, its homotopy category $\h\cM$ has all weak colimits: given any small diagram $F \colon \cD \to \h\cM$, there is a cone under $F$ through which every other cone factors, although not necessarily uniquely.
\end{thm}

In general, the colimit of a diagram $F$ of shape $\cD$ may be constructed as the reflexive coequalizer of the diagram
\[
\begin{tikzcd}[column sep=large]
\coprod\limits_{a,b \in \cD} \cD(a,b) \times Fa \arrow[r, shift left=0.75em, "\ev"] \arrow[r, shift right=0.75em, "\proj"'] & \coprod\limits_{a \in \cD} Fa \arrow[l, "\id" description] 
\end{tikzcd}
\]
Note that this construction does not actually require the diagram $F$ to be a functor; it suffices for the diagram to define a reflexive directed graph in the target category. In the case of a diagram valued in $\h\cM$, the weak colimit will be constructed as a ``homotopy reflexive coequalizer''\footnote{Succinctly, it may be defined as the weighted colimit of this reflexive coequalizer diagram weighted by the truncated cosimplicial object $\begin{tikzcd}[ampersand replacement=\&] {\ast} \arrow[r, shift left=0.5em] \arrow[r, shift right=0.5em] \& I \arrow[l] \end{tikzcd}$ whose leftwards maps are the endpoint inclusions into the closed interval $I$; see \S\ref{ssec:weighted}.} of a lifted reflexive directed graph in $\cM$.

\begin{proof}
Any diagram $F \colon \cD \to \h\cM$ may be lifted to a reflexive directed graph $F \colon \cD \to \cM$, choosing representatives for each homotopy class of morphisms in such a way that identities are chosen to represent identities. Using these lifted maps and writing $I$ for the interval, define the weak colimit of $F\colon \cD \to \h\cM$ to be a quotient of the coproduct
\[ \left( \coprod\limits_{a,b \in \cD} \cD(a,b) \times I \times Fa \right) \sqcup \left( \coprod\limits_{a \in \cD} Fa\right)\] modulo three identifications
\[
\begin{tikzcd}[column sep=4.5em]
 \coprod\limits_{a,b \in \cD}  \left( \cD(a,b) \times \{0\}  \times Fa \sqcup  \cD(a,b) \times \{1\}  \times Fa \right) \sqcup \coprod\limits_{a \in \cD} I \times Fa \arrow[r, "(\ev \sqcup\proj)\sqcup\proj"] \arrow[d, "(\incl\sqcup\incl)\sqcup\id"'] \arrow[dr, phantom, "\ulcorner" very near end] & \coprod\limits_{a \in \cD} Fa \arrow[d, dashed] \\ \coprod\limits_{a,b \in \cD} \cD(a,b) \times I \times Fa  \arrow[r, dashed] & \wcolim{F}
 \end{tikzcd}
 \]
The right hand vertical map defines the legs of the colimit cone, which commute in $\h\cM$ via the witnessing homotopies given by the bottom horizontal map.

Now consider a cone in $\h\cM$ under $F$ with nadir $X$. We may regard the data of this cone as a diagram $\cD \times \2 \to \h\cM$ that restricts along $\{0\} \hookrightarrow \2$ to $F$ and along $\{1\}\hookrightarrow \2$ to the constant diagram at $X$. This data may be lifted to a reflexive directed graph $\cD\times\2 \to \cM$ whose lift over $0$ agrees with the previously specified lift $F$ and whose lift over $1$ is constant at $X$. This defines a cone under the pushout diagram, inducing the required map $\wcolim{F} \to X$.
\end{proof}

\subsection{Homotopy limits and colimits of general shapes}\label{ssec:holim}

In general, limit and colimit constructions in a homotopical category fail to be weak equivalence invariant. Famously the $n$-sphere can be formed by gluing together two disks along their boundary spheres $S^n \cong D^n \cup_{S^{n-1}} D^n$. The diagram 
\begin{equation}\label{eq:nhocolim}
\begin{tikzcd}
D^n \arrow[d, "{\rotatebox{270}{$\sim$}}"'] & {~}S^{n-1}{~} \arrow[r, hook]\arrow[l, hook'] \arrow[d, equals] & D^n \arrow[d, "{\rotatebox{90}{$\sim$}}"] \\ {*} & S^{n-1} \arrow[l] \arrow[r] & {*}
\end{tikzcd}
\end{equation} reveals that the pushout functor fails to preserves componentwise homotopy equivalences.

When a functor fails to be homotopical, the next best option is to replace it by a derived functor. Because colimits are left adjoints, one might hope that $\colim \colon \cM^\cD \to \cM$ has a left derived functor and dually that $\lim \colon \cM^\cD \to \cM$ has a right derived functor, leading us to the following definition:

\begin{defn} Let $\cM$ be a homotopical category and let $\cD$ be a small category. The \textbf{homotopy colimit functor}, when it is exists, is a left derived functor $\LL\colim\colon \cM^\cD \to \cM$ while the \textbf{homotopy limit functor}, when it exists, is a right derived functor $\RR\!\lim\colon \cM^\cD \to \cM$.
\end{defn}

We always take the weak equivalences in the category $\cM^\cD$ of diagrams of shape $\cD$ in a homotopical category $\cM$ to be defined pointwise. By the universal property of localization, there is a canonical map 
\begin{equation}\label{eq:coherent-vs-commutative} 
\begin{tikzcd} \cM^\cD \arrow[d, "{\gamma}"'] \arrow[r, "{\gamma^\cD}"] & (\Ho\cM)^\cD \\ \Ho(\cM^\cD) \arrow[ur, dashed]
\end{tikzcd}
\end{equation} 
but it is not typically an equivalence of categories. Indeed, some of the pioneering forays into abstract homotopy theory \cite{vogt,CP-vogt,DKS-homotopy}. were motivated by attempts to understand the essential image of the functor $\Ho(\cM^\cD) \to (\Ho\cM)^\cD$, the objects in $(\Ho\cM)^\cD$ being homotopy commutative diagrams while the isomorphism classes of objects in $\Ho(\cM^\cD)$ being somewhat more mysterious; see \S\ref{ssec:bergner}.

\begin{rmk}\label{rmk:holim-no-lim}
The diagonal functor $\Delta \colon \cM \to \cM^\cD$ is homotopical and hence acts as its own left and right derived functors. By Theorem \ref{thm:derived-adj} applied to a Quillen adjunction to be constructed in the proof of Theorem \ref{thm:proj-hocolim}, the total derived functor $\bL\colim \colon \Ho(\cM^\cD) \to \Ho\cM$ is left adjoint to $\Delta \colon \Ho \cM \to \Ho(\cM^\cD)$ but unless the comparison of \eqref{eq:coherent-vs-commutative} is an equivalence, this is not the same as the diagonal functor $\Delta \colon \Ho\cM \to \Ho(\cM)^{\cD}$. Hence, homotopy colimits are not typically colimits in the homotopy category.\footnote{The comparison \eqref{eq:coherent-vs-commutative} is an equivalence when $\cD$ is discrete, which is the reason why homotopy products and homotopy coproducts \emph{are} products and coproducts in the homotopy category.}
\end{rmk}

In the presence of suitable model structures, Corollary \ref{cor:quillen-derived} can be used to prove that the homotopy limit and colimit functors exist.

\begin{defn} Let $\cM$ be a model category and let $\cD$ be a small category. 
\begin{enumerate}
\item The \textbf{projective model structure} on $\cM^\cD$ has weak equivalences and fibrations defined pointwise in $\cM$. 
\item The \textbf{injective model structure} on $\cM^\cD$ has weak equivalences and cofibrations defined pointwise in $\cM$.
\end{enumerate}
\end{defn}

When $\cM$ is a \emph{combinatorial model category} both model structures are guaranteed to exist. More generally when $\cM$ is an \emph{accessible model category} these model structures exist \cite[3.4.1]{HKRS}. Of course, the projective and injective model structures might happen to exist on $\cM^\cD$, perhaps for particular diagram shapes $\cD$, in the absence of these hypotheses.

\begin{thm}\label{thm:proj-hocolim} Let $\cM$ be a model category and let $\cD$ be a small category.
\begin{enumerate}
\item Whenever the projective model structure on $\cM^\cD$ exists then the homotopy colimit functor $\LL\colim\colon \cM^\cD \to \cM$ exists and may be computed as the colimit of a projective cofibrant replacement of the original diagram.
\item Whenever the injective model structure on $\cM^\cD$ exists then the homotopy limit functor $\RR\!\lim\colon \cM^\cD \to \cM$ exists and may be computed as the limit of an injective fibrant replacement of the original diagram.
\end{enumerate}
\end{thm}
\begin{proof}
This follows from Corollary \ref{cor:quillen-derived} once we verify that the colimit and limit functors are respectively left and right Quillen with respect to the projective and injective model structures. These functors are, respectively, left and right adjoint to the constant diagram functor $\Delta \colon\cM\to \cM^\cD$, so by Definition \ref{defn:quillen-adj} it suffices to verify that this functor is right Quillen with respect to the projective model structure and also left Quillen with respect to the injective model structure. But these model structures are designed so that this is the case.
\end{proof}

\begin{exc}\label{exc:hopushout}$\quad$
\begin{enumerate}
\item Show that any pushout diagram $\begin{tikzcd} B & A \arrow[r, tail] \arrow[l, tail] & C\end{tikzcd}$ comprised of a pair of cofibrations between cofibrant objects is projectively cofibrant. Conclude that the pushout of cofibrations between cofibrant objects is a homotopy pushout and use this to compute the homotopy pushout of \eqref{eq:nhocolim}.
\item Argue that for a generic pushout diagram $\begin{tikzcd} Y & X \arrow[l] \arrow[r] & Z\end{tikzcd}$, its homotopy pushout may be constructed by taking a cofibrant replacement $q \colon X' \to X$ of $X$ and then factoring the composites $h q$ and $k q$ as a cofibration followed by a trivial fibration.
\begin{tikzcd} {Y'} \arrow[d, two heads, "{\rotatebox{270}{$\sim$}}"'] & {X'} \arrow[d, "{\rotatebox{270}{$\sim$}}"', "q"] \arrow[l, tail] \arrow[r, tail]  & {Z'} \arrow[d, two heads, "{\rotatebox{270}{$\sim$}}"] \\ Y & X \arrow[l] \arrow[r] & Z
\end{tikzcd}
and then taking the ordinary pushout of this projective cofibrant replacement formed by the top row.
\end{enumerate}
\end{exc}

\begin{exc}\label{exc:seqhocolim}$\quad$
\begin{enumerate}
\item Verify that any $\bbomega$-indexed diagram
\[\begin{tikzcd}
A_0 \arrow[r, tail, "f_{01}"] & A_1 \arrow[r, tail,  "f_{12}"] & A_2 \arrow[r, tail, "f_{23}"] & \cdots
\end{tikzcd}
\]
of cofibrations between cofibrant objects is projectively cofibrant. Conclude that the sequential colimit of a diagram of cofibrations between cofibrant objects is a homotopy colimit.
\item Argue that for a generic sequential diagram
\[\begin{tikzcd}
X_0 \arrow[r, "f_{01}"] & X_1 \arrow[r, "f_{12}"] & X_2 \arrow[r, "f_{23}"] & \cdots
\end{tikzcd}
\] its projective cofibrant replacement may be formed by first replacing $X_0$ by a cofibrant object $Q_0$, then inductively factoring the resulting composite map $Q_n \to X_{n+1}$ into a cofibration followed by a trivial fibration:
 \[
 \begin{tikzcd}G \arrow[d, "q"', "{\rotatebox{90}{$\sim$}}"] & Q_0 \arrow[d, "q_0"', "{\rotatebox{90}{$\sim$}}"] \arrow[r, tail, "g_{01}"]& Q_1  \arrow[d, "{q_1}"', "{\rotatebox{90}{$\sim$}}"] \arrow[r, tail, "{g_{12}}"] & Q_2 \arrow[d, "q_2"', "{\rotatebox{90}{$\sim$}}"] \arrow[r, tail, "{g_{23}}"] & \cdots \\  F & X_0 \arrow[r, "f_{01}"] & X_1 \arrow[r, "f_{12}"] & X_2 \arrow[r, "f_{23}"] & \cdots
\end{tikzcd}
\]
Conclude that the homotopy sequential colimit is formed as the sequential colimit of this top row.
\end{enumerate}
\end{exc}

\subsection{Homotopy limits and colimits of Reedy diagrams}\label{ssec:reedy}

In fact, even if the projective model structures do not exist, certain diagram shapes  allow us to construct functorial ``projective cofibrant replacements'' in any model category nonetheless, such as following the prescriptions of Exercise \ref{exc:seqhocolim}. Dual ``injective fibrant replacements'' for pullback or inverse limit diagrams exist similarly.  The reason is because the categories indexing these diagrams are \emph{Reedy categories}.

 If $\cM$ is any model category and $\cD$ is any Reedy category, then category $\cM^\cD$ of Reedy diagrams admits a model structure. If the indexing category $\cD$ satisfies the appropriate half of a dual pair of conditions listed in Proposition \ref{prop:constants}, then the colimit or limit functors $\colim,\lim \colon \cM^\cD \to \cM$ are left or right Quillen. In such contexts, homotopy colimits and homotopy limits can be computed by applying Corollary \ref{cor:quillen-derived}.

The history of the abstract notion of Reedy categories is entertaining. The category $\DDelta^\op$ is an example of what is now called a \emph{Reedy category}. The eponymous model structure on simplicial objects taking values in any model category was introduced in an unpublished but widely disseminated manuscript written by Reedy \cite{reedy}. Reedy notes that a dual model structure exists for cosimplicial objects, which, in the case of cosimplicial simplicial sets, coincides with a model structure introduced by Bousfield and Kan to define homotopy limits \cite[\S X]{bousfield-kan}. The general definition, unifying these examples and many others, is due to Kan and appeared in the early drafts of the book that eventually became \cite{DHKS}. Various draft versions circulated  in the mid 1990s and contributed to the published accounts \cite[chapter 15]{hirschhorn} and \cite[chapter 5]{hovey}. The final \cite{DHKS} in turns references these sources in order to ``review the notion of a Reedy category'' that originated in an early draft of that same manuscript.

\begin{defn}\label{defn:reedy} A \textbf{Reedy structure} on a small category $\cA$ consists of a \textbf{degree function} $\deg \colon \ob{\cA} \to \omega$ together with a pair of wide subcategories $\overrightarrow{\cA}$ and $\overleftarrow{\cA}$ of \textbf{degree-increasing} and \textbf{degree-decreasing} arrows respectively so that 
\begin{enumerate}
\item The degree of the domain of every non-identity morphism in  $\overrightarrow{\cA}$ is strictly less than the degree of the codomain, and the degree of the domain of every non-identity morphism in $\overleftarrow{\cA}$ is strictly greater than the degree of the codomain.
\item Every $f \in \mor{\cA}$ may be factored uniquely as 
\begin{equation}\label{eq:reedy-fact}
\begin{tikzcd}
\bullet \arrow[rr, "f"] \arrow[dr, "{\overleftarrow{\cA}\ni \overleftarrow{f}}"'] & & \bullet \\ & \bullet \arrow[ur, "{\overrightarrow{f} \in \overrightarrow{\cA}}"'] 
\end{tikzcd}
\end{equation}
\end{enumerate}
\end{defn}

\begin{ex} $\quad$
\begin{enumerate}
\item Discrete categories are Reedy categories, with all objects having degree zero
\item If $\cA$ is a  Reedy category, then so is $\cA^\op$: its Reedy structure has the same degree function but has the degree-increasing and degree-decreasing arrows interchanged. 
\item Finite posets are Reedy categories with all morphisms degree-increasing.  Declare any minimal element to have degree zero and define the degree of a generic object $d \in \cD$ to be the length of the maximal-length path of non-identity arrows from an element of degree zero to $d$. This example can be extended without change to include infinite posets such as $\bbomega$ provided that each object has finite degree.
\item The previous example gives the category $b \leftarrow a \rightarrow c$ a Reedy structure in which $\deg(a)=0$ and $\deg(b)= \deg(c)=1$. There is another Reedy category structure in which $\deg(b)=0$, $\deg(a)=1$, and $\deg(c)=2$. 
\item The category $a \rightrightarrows b$ is a Reedy category with $\deg(a)=0$, $\deg(b)=1$, and both non-identity arrows said to strictly raise degrees.
\item   The category $\DDelta$ of finite non-empty ordinals and the category $\DDelta_+$ of finite ordinals and order-preserving maps  both support canonical Reedy category structures, for which we take the degree-increasing maps to be the subcategories of face operators (monomorphisms) and the degree-decreasing maps to be the subcategories of degeneracy operators (epimorphisms).   
\end{enumerate}
\end{ex}

\begin{exc} $\quad$
\begin{enumerate}
\item Show that every morphism $f$ factors uniquely through an object of minimum degree and this factorization is the ``Reedy factorization'' of  \eqref{eq:reedy-fact}.
\item Show that the Reedy category axioms prohibit any non-identity iso\-morph\-isms.
\end{enumerate}
\end{exc}

\begin{rmk}
The notion of Reedy category has been usefully extended by Berger and Moerdijk to include examples such as finite sets or finite pointed sets that do have non-identity automorphisms. All of the results to be described here have analogues in this more general context, but for ease of exposition we leave these details to \cite{BM}.
\end{rmk}

To focus attention on our goal, we now introduce the \emph{Reedy model structure}, which serves as motivation for some auxiliary constructions we  have yet to introduce.

\begin{thm}[{Reedy, Kan \cite[\S 7]{RV0}}] Let $\cM$ be a model category and let $\cD$ be a Reedy category. Then the category $\cM^\cD$ admits a model structure whose
\begin{itemize}
\item weak equivalences are the pointwise weak equivalences
\item weak factorization systems $(\cof \cap \we[\cD],\fib[\cD])$ and $(\cof[\cD], \fib\cap\we[\cD])$ are the Reedy weak factorization systems
\end{itemize}
\end{thm}

In the \textbf{Reedy weak factorization system} $(\sL[\cD],\sR[\cD])$ defined relative to a weak factorization system $(\sL,\sR)$ on $\cM$, a natural transformation $f \colon X \to Y \in \cM^\cD$ is in $\sL[\cD]$ or $\sR[\cD]$, respectively, if and only if, for each $d \in \cD$, the \emph{relative latching map} $X^d \cup_{L^dX} L^dY \to Y^d$ is in $\sL$ or the \emph{relative matching map} $X^d \to Y^d \times_{M^dY} M^dX$ is in $\sR$. The most efficient definitions of these latching and matching objects $L^dX$ and $M^dX$ appearing in Example \ref{ex:latch-match} makes use of the theory of weighted colimits and limits, a subject to which we now turn.

\subsection{Quillen adjunctions for weighted limits and colimits}\label{ssec:weighted}

Ordinary limits and colimits are objects representing the functor of cones with a given summit over or under a fixed diagram. Weighted limits and colimits are defined analogously, except that the cones over or under a diagram might have exotic ``shapes.'' These shapes are allowed to vary with the objects indexing the diagram. More formally, the \textbf{weight} --- a functor which specifies the ``shape'' of a cone over a diagram indexed by $\cD$ or a cone under a diagram indexed by $\cD^\op$ --- takes the form of a functor in $\cat{Set}^\cD$ in the unenriched context or $\cV^\cD$ in the $\cV$-enriched context.

%\begin{ex}[tensors and cotensors]\label{ex:tensor-cotensor}For example, in the case of a diagram of shape $\1$ in a category $\cM$, the shape of a cone might be a set $S \in \cat{Set}$. Writing $X \in \cM$ for the object in the image of the diagram, the $S$-weighted limit of $X$ is an object $\{S,  X\} \in \cM$ satisfying the universal property \[ \cM(M, \{S, X\}) \cong \cat{Set}(S,\cM(M,X))\] while the $S$-weighted colimit of $X$ is an object $S \ast X \in \cM$ satisfying the universal property \[ \cM(S \ast X, M) \cong \cat{Set}(S, \cM(X,M)).\]  For historical reasons, $\{S, X\}$ is called the \textbf{cotensor} and $S \ast X$ is called the \textbf{tensor} of $X \in \cM$ by the set $S$. 

%If $\cM$ has small products and coproducts, in this case guaranteed by our standing assumption that the large categories under consideration are bicomplete,  then $\{S ,X\}$ and $S\ast X$ are, respectively, the $S$-fold product and coproduct of the object $X$ with itself, and cotensors and tensors define bifunctors \[ \{-,-\}  \colon \cat{Set}^\op \times \cM \to \cM \qquad \mathrm{and}\qquad  -\ast- \colon \cat{Set} \times \cM \to \cM.\]   \end{ex}

\begin{defn}[weighted limits and colimits, axiomatically]\label{defn:weighted-i}
For a general small category $\cD$ and bicomplete category $\cM$, the weighted limit and weighted colimit define bifunctors
\[   \{-,-\}^{\cD}\colon (\cat{Set}^{\cD})^\op\times\cM^\cD \to \cM \qquad \mathrm{and} \qquad      -\ast_{\cD}-\colon \cat{Set}^{\cD}\times\cM^{\cD^\op} \to {\cM}\]
which are characterized by the following pair of axioms.
\begin{enumerate}
\item Weighted (co)limits with representable weights evaluate at the representing object: 
\[\{\cD(d,-),X\}^\cD \cong X(d) \qquad \text{and} \qquad \cD(-,d)\ast_\cD Y \cong Y(d).\]
\item The weighted (co)limit bifunctors are cocontinuous in the weight: for any diagram $X \in \cM^\cD$, the functor $ -\ast_\cD X$ preserves colimits, while the functor $\{-,X\}^\cD$ carries colimits to limits.
\end{enumerate}
We interpret axiom (ii) to mean that weights can be ``made-to-order'': a weight constructed as a colimit of representables --- as all $\cat{Set}$-valued functors are --- will stipulate the expected universal property. 
\end{defn}

Let $\cM$ be any locally small category with products and coproducts.  For any set $S$, the $S$-fold product and coproduct define cotensor and tensor bifunctors \[ \{-,-\}  \colon \cat{Set}^\op \times \cM \to \cM \qquad \mathrm{and}\qquad  -\ast- \colon \cat{Set} \times \cM \to \cM,\] which form a two-variable adjunction with $\Hom  \colon \cM^\op \times \cM \to \cat{Set}$; cf.~Exercise \ref{exc:unenriched-as-enriched}.

\begin{defn}[weighted limits and colimits, constructively]\label{defn:weighted-ii}
The weighted colimit is a functor tensor product and the weighted limit is a functor cotensor product: \[ \{W,X\}^\cD \cong \int_{d \in \cD} \{W(d),X(d)\} \qquad \qquad W \ast_\cD Y \cong \int^{d \in \cD} W(d) \ast Y(d).\]
 The  limit $\{W,X\}^\cD$ of the diagram $X$ weighted by $W$ and the colimit $W \ast_\cD Y$  of $Y$ weighted by $W$ are characterized by the universal properties: \[ \cM(M, \{{W},{X}\}^{\cD}) \cong \cat{Set}^{\cD}(W, \cM(M,X)) \qquad \cM(W\ast_{\cD} Y, M) \cong \cat{Set}^{\cD^\op}(W,\cM(Y,M)).\] 
\end{defn}

\begin{ex}\label{ex:latch-match} Let $\cA$ be a Reedy category and write $\cA_{\leq n}$ for the full subcategory of objects of degree at most $n$. Restriction along the inclusion $\cA_{\leq n}\hookrightarrow \cA$ followed by left Kan extension defines an comonad $\sk_n\colon \cat{Set}^\cA \to \cat{Set}^\cA$. 

Let $a \in \cA$ be an object of degree $n$ and define 
\[\partial\cA(a,-) := \sk_{n-1}\cA(a,-) \in \cat{Set}^{\cA}\quad \mathrm{and}\quad\partial\cA(-,a):=\sk_{n-1}\cA(-,a)\in \cat{Set}^{\cA^\op},\] where $\cA(a,-)$ and $\cA(-,a)$ denote the co- and contravariant functors represented by $a$, respectively. %The counit of the skeleton comonad induces natural inclusions $\partial\cA(-,a)\hookrightarrow\cA(-,a)$ and $\partial\cA(a,-)\hookrightarrow \cA(a,-)$.
 Then for any $X \in \cM^\cA$, the \textbf{latching} and \textbf{matching} objects are defined by
\[ L^aX := \partial\cA(-,a) \ast_\cA X \qquad \mathrm{and} \qquad M^aX := \{\partial\cA(a,-),X\}.\]
\end{ex}

\begin{exc}[enriched weighted limits and colimits]\label{exc:enriched-weighted}
For the reader who knows some enriched category theory, generalize Definitions \ref{defn:weighted-i} and \ref{defn:weighted-ii} to the $\cV$-enriched context to define weighted limit and weighted colimit bifunctors 
\[ \{-,-\}^\cA \colon (\cV^\cA)^\op \times \cM^\cA \to \cM \qquad \mathrm{and} \qquad - \otimes_\cA - \colon \cV^\cA \times \cM^{\cA^\op} \to \cM\]
in any $\cV$-enriched, tensored, and cotensored category $\cM$ whose underlying unenriched category is 
 complete and cocomplete.
\end{exc}

Recall the notion of Quillen two-variable adjunction, the prototypical example being the tensor-cotensor-hom of a $\cV$-model category $\cM$. 

\begin{thm}[{\cite[7.1]{riehl-reedy}}]\label{thm:left-leibniz-bifunctor} Let $\cA$ be a Reedy category and let $\otimes \colon \cK \times \cL \to \cM$ be a left Quillen bifunctor between model categories. Then the functor tensor product
\[ \otimes_{\cA} \colon \cK^{\cA^\op} \times \cL^\cA \to \cM\]
is left Quillen with respect to the Reedy model structures.
\end{thm}

A dual result holds for functor cotensor products formed relative to a right Quillen bifunctor. In particular, if $\cM$ is a $\cV$-model category, then its tensor, cotensor, and hom define a Quillen two-variable adjunction, and so in particular:

\begin{cor}\label{cor:quillen-strict-weighted} Let $\cM$ be a $\cV$-model category and let $\cA$ be a Reedy category. Then for any Reedy cofibrant weight $W \in \cV^\cA$, the weighted colimit and weighted limit functors
\[ W \ast_\cA - \colon \cM^\cA \to \cM \qquad \mathrm{and}\qquad \{W,-\}^\cA \colon \cM^{\cA^\op} \to \cM\] are respectively left and right Quillen with respect to the Reedy model structures on $\cM^\cA$ and $\cM^{\cA^\op}$.
\end{cor}

\begin{ex}[geometric realization and totalization] The Yoneda embedding defines a Reedy cofibrant weight $\DDelta^\bullet \in \cat{sSet}^{\DDelta}$. The weighted colimit and weighted limit functors
\[ \DDelta^\bullet \ast_{\DDelta^\op} - \colon \cM^{\DDelta^\op} \to \cM \qquad \mathrm{and} \qquad \{\DDelta^\bullet, - \}^{\DDelta} \colon \cM^{\DDelta} \to \cM\] typically go by the names of \textbf{geometric realization} and \textbf{totalization}. Corollary \ref{cor:quillen-strict-weighted} proves that if $\cM$ is a simplicial model category, then these functors are left and right Quillen.
\end{ex}

By Exercise \ref{exc:unenriched-as-enriched}, Corollary \ref{cor:quillen-strict-weighted} also has implications in the case of an unenriched model category $\cM$ in which case ``Reedy cofibrant'' should be read as ``Reedy monomorphic.'' Ordinary limits and colimits are weighted limits and colimits where the weight is the terminal functor, constant at the singleton set.

\begin{prop}[homotopy limits and colimits of Reedy shape]\label{prop:constants}
$\quad$
\begin{enumerate}
\item  If $\cA$ is a Reedy category with the property that the constant $\cA$-indexed diagram at any cofibrant object in any model category is Reedy cofibrant, then the limit functor $\lim \colon \cM^\cA \to \cM$ is right Quillen.
\item  If $\cA$ is a Reedy category with the property that the constant $\cA$-indexed diagram at any fibrant object in any model category is Reedy fibrant, then the colimit functor $\colim \colon \cM^\cA \to \cM$ is left Quillen.
\end{enumerate}
\end{prop}
\begin{proof}
Taking the terminal weight $1$ in $\cat{Set}^\cA$, the weighted limit reduces to the ordinary limit functor. The functor $1 \in \cat{Set}^\cA$ is Reedy monomorphic just when, for each $a \in\cA$, the category of elements for the weight $\partial\cA(-,a)$ is either empty or connected. This is the case if and only if $\cA$ has ``cofibrant constants,'' meaning that the constant $\cA$-indexed diagram at any cofibrant object in any model category is Reedy cofibrant. Thus, we conclude that if $\cA$ has cofibrant constants, then the limit functor $\lim \colon \cM^\cA \to \cM$ is right Quillen. See \cite[\S9]{RV0} for more discussion.
\end{proof}

There is an analogous result for projective and injective model structures which the author first saw formulated in this way by Gambino in the context of a simplicial model category.

\begin{thm}[{\cite{gambino}}]\label{thm:gambino} If $\cM$ is a $\cV$-model category and $\cD$ is a small category, then the weighted colimit functor \[ -\otimes_\cD - \colon \cV^{\cD} \times \cM^{\cD^\op} \to \cM \] is left Quillen if the domain has the (injective, projective) or (projective, injective) model structure. Similarly, the weighted limit functor \[ \{-,-\}^\cD\colon (\cV^{\cD})^\op \times \cM^\cD \to \cM\] is right Quillen if the domain has the (projective, projective) or (injective, injective) model structure.
\end{thm}
\begin{proof} 
By Definition \ref{defn:quillen-two-var} we can prove both statements in adjoint form. The weighted colimit bifunctor of Exercise \ref{exc:enriched-weighted} has a right adjoint  (used to express the defining universal property of the weighted colimit) \[\Map(-,-) \colon (\cM^{\cD^\op})^\op \times \cM \to \cV^{\cD}\] which sends $F \in \cM^{\cD^\op}$ and $m \in \cM$ to $\Map(F-,m) \in \cV^\cD$. 

To prove the statement when $\cV^{\cD}$ has the projective and $\cM^{\cD^\op}$ has the injective model structure, we must show that this is a right Quillen bifunctor with respect to the pointwise (trivial) cofibrations in $\cM^{\cD^\op}$, (trivial) fibrations in $\cM$, and pointwise (trivial) fibrations in $\cV^{\cD}$. Because the limits involved in the definition of right Quillen bifunctors are also formed pointwise, this follows immediately from the corresponding property of the simplicial hom bifunctor, which was part of the definition of a simplicial model category. The other cases are similar. 
\end{proof}

The upshot of Theorem \ref{thm:gambino} is that there are two approaches to constructing a homotopy colimit: fattening up the diagram, as is achieved by the derived functors of \S\ref{ssec:holim}, or fattening up the weight. The famous Bousfield-Kan formulae for homotopy limits and colimits in the context of a simplicial model category define them to be weighted limits and colimits for a particular weight constructed as a projective cofibrant replacement of the terminal weight; see \cite{bousfield-kan} or \cite[\S 11.5]{riehl-cathtpy}. The Quillen two variable adjunction of Theorem \ref{thm:gambino} can be derived as in Theorem \ref{thm:hoVmonoidal} to express a homotopically-enriched universal property of the weighted limit or colimit, as representing ``homotopy coherent'' cones over or under a diagram, an intuition to be explored in the next section.

\section{Simplicial localizations}\label{sec:hammock}

Quillen's model categories provide a robust axiomatic framework within which to ``do homotopy theory.'' But the constructions of \S\ref{sec:holim} imply that the homotopy theories presented by model categories have all homotopy limits and homotopy colimits, which need not be the case in general. In this section we introduce a framework originally developed by Dwyer and Kan and re-conceptualized by Bergner which allows us to extend our notion of equivalence between homotopy theories introduced in \S\ref{ssec:quillen-equiv} to a more flexible notion of \emph{DK-equivalence} (after Dwyer and Kan) that identifies when any two homotopical categories are equivalent.

A mere equivalence of categories of fractions is insufficient to detect an equivalence of homotopy theories; instead a construction that takes into account the ``higher dimensional'' homotopical structure is required. To that end, Dwyer and Kan build, from any homotopical category $(\cK,\we)$,  a simplicial category $\sL^H(\cK,\we)$ called the \emph{hammock localization} \cite{DK-function} and demonstrate that their construction has a number of good products that we tour in \S\ref{ssec:hammock-loc}:
\begin{itemize}
\item The homotopy category $\h{\sL^H(\cK,\we)}$ is equivalent to the category of fractions $\cK[\we]^{-1}$ (Proposition \ref{prop:hammock-to-fractions}).
\item If $(\cK,\we)$ underlies a simplicial model category then the Kan complex enriched category $\cK_{\mathrm{cf}} \subset \cK$ is DK-equivalent to $\sL^H\cK$ (Proposition \ref{prop:hammock-of-model}).
\item More generally, $\sL^H(\cK,\we)$ provides a not-necessarily simplicial model category $(\cK,\we)$ with function complexes that have the correct mapping type even if the model structure is not simplicial (Proposition \ref{prop:hammock-gives-spaces}).
\item If two model categories are Quillen equivalent, then their hammock localizations are DK-equivalent (Proposition \ref{prop:DK-generalizes-quillen}).
\end{itemize}

The DK-equivalences are those simplicial functors that are bijective on homotopy equivalence classes of objects and define local equivalences of the mapping spaces constructed by the hammock localization. Zooming out a categorical level, the Bergner model structure on simplicially enriched categories gives a presentation of the homotopical category of homotopy theories, with the DK-equivalences as its weak equivalences. This is the subject of \S\ref{ssec:bergner}.

\subsection{The hammock localization}\label{ssec:hammock-loc}

There are two equivalent ways to present the data of a simplicially enriched category, either as a category equipped with a simplicial set of morphisms between each pair of objects, or simplicial diagram of categories $\cK_n$ of $n$-\textbf{arrows}, each of which is equipped with a constant set of objects.

\begin{exc}\label{exc:simplicial-category}
Prove that the following are equivalent:
\begin{enumerate}
\item A simplicially enriched category with objects $\ob\cK$.
\item A simplicial object $\cK_\bullet \colon \DDelta^\op \to \cat{Cat}$ in which each of the categories $\cK_n$ has objects $\ob\cK$ and each functor $\cK_n \to \cK_m$ is the identity on objects.
\end{enumerate}
\end{exc}

We being by introducing the notion of a DK-equivalence between simplicially enriched categories. 

\begin{defn}\label{defn:DK-equiv} A simplicial functor $F \colon \cK \to \cM$ is a \textbf{DK-equivalence} iff
\begin{enumerate}
\item\label{itm:DK-i} It defines an equivalence of homotopy categories $\h{F} \colon \h\cK \to \h\cM$.
\item\label{itm:DK-ii} It defines a local weak equivalence of mapping complexes: for all $X,Y \in \cK$, $\Map_\cK(X,Y) \wto \Map_\cM(FX,FY)$.
\end{enumerate}
\end{defn}

In the case where $F$ is identity on objects, condition \eqref{itm:DK-ii} subsumes condition \eqref{itm:DK-i}.

\begin{defn}[{\cite[2.1]{DK-calculating}}]  Let $\cK$ be a category with a wide subcategory $\we$, containing all the identity arrows. The \textbf{hammock localization} $\sL^H(\cK,\we)$ is a simplicial category with the same objects as $\cK$ and with the mapping complex $\Map(X,Y)$ defined to be the simplicial set whose $k$-simplices are ``reduced hammocks of width $k$'' from $X$ to $Y$, these being commutative diagrams
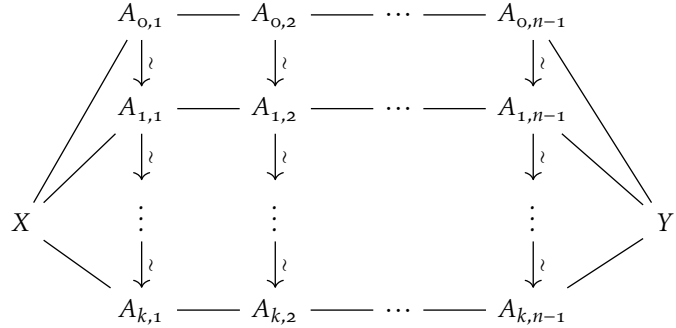
\[
\begin{tikzcd} & A_{0,1} \arrow[d, "\wr"] \arrow[r, no head] & A_{0,2} \arrow[d, "\wr"] \arrow[r, no head] & \cdots \arrow[r, no head] & A_{0,n-1} \arrow[d, "\wr"] \arrow[ddr, no head] \\
& A_{1,1} \arrow[d, "\wr"] \arrow[r, no head] & A_{1,2} \arrow[d, "\wr"] \arrow[r, no head] & \cdots \arrow[r, no head] & A_{1,n-1} \arrow[d, "\wr"]\arrow[dr, no head] \\ X \arrow[uur, no head] \arrow[ur, no head] \arrow[dr, no head] & \vdots\arrow[d, "\wr"] & \vdots\arrow[d, "\wr"] & & \vdots\arrow[d, "\wr"] & Y \\ & A_{k,1} \arrow[r, no head] & A_{k,2} \arrow[r, no head] & \cdots \arrow[r, no head] & A_{k,n-1} \arrow[ur, no head]
\end{tikzcd}
\]
where the length of the hammock is any integer $n \geq 1$ so that
\begin{enumerate}
\item all vertical maps are in $\we$,
\item in each column of horizontal morphisms all maps go in the same direction and if they go left then they are in $\we$, and 
\item the maps in adjacent columns go in different directions.
\end{enumerate}
The graded set of reduced hammocks of width $k$ from $X$ to $Y$ becomes a simplicial set $\Map(X,Y)$ in which
\begin{enumerate}
\item[(i)$'$] faces are defined by omitting rows and 
\item[(ii)$'$] degeneracies are defined by duplicating rows.
\end{enumerate}
Composition is defined by horizontally pasting hammocks and then reducing by 
\begin{enumerate}
\item[(i)$''$] composing adjacent columns whose maps point in the same direction and
\item[(ii)$''$] omitting any column which contains only identity maps.
\end{enumerate}
\end{defn}

There is a canonical functor $\cK \to \sL^H(\cK,\we)$ whose image is comprised of dimension zero length 1 hammocks pointing forwards. 

\begin{exc}\label{exc:hammock-inverts} Verify that the composite of the functor $\cK \to \sL^H(\cK,\we)$ just described with the quotient functor $\sL^H(\cK,\we) \to \h\sL^H(\cK,\we)$ that collapses each mapping space onto its set of path components inverts the weak equivalences in $\cK$, sending each to an isomorphism in the homotopy category  $\h\sL^H(\cK,\we)$.
\end{exc}

In the hammock localization \begin{quote} cancelation in any dimension is achieved not by ``imposing relations'' in the same dimension, but by ``imposing homotopy relations'', i.e.~adding maps, in the next dimension, \cite[\S 3]{DK-calculating}
\end{quote}
in contrast with the category of fractions constructed in \S\ref{sec:fractions}. By considering the effect of these ``homotopy relations,'' it is straightforward to see that the induced functor from the category of fractions to the homotopy category of the hammock localization is an isomorphism of categories.

\begin{prop}[{Dwyer-Kan \cite[3.2]{DK-function}}]\label{prop:hammock-to-fractions} The canonical functor $\cK \to \sL^H(\cK,\we)$ induces an isomorphism of categories $\cK[\we^{-1}] \cong \h\sL^H(\cK,\we)$.
\end{prop}
\begin{proof}
The comparison functor $\cK[\we^{-1}] \to \h\sL^H(\cK,\we)$ induced by Exercise \ref{exc:hammock-inverts} and  the universal property of Proposition \ref{prop:fractions-UP} is clearly bijective on objects and  full, homotopy classes in $\h\sL^H(\cK,\we)$ being represented by zig zags whose ``backwards'' maps lie in $\we$. To see that this functor is faithful it suffices to consider a 1-simplex in $\Map(X,Y)$
\[
\begin{tikzcd}[row sep=tiny] & A_{0,1} \arrow[dd, "\wr"] \arrow[r, no head] & A_{0,2} \arrow[dd, "\wr"] \arrow[r, no head] & \cdots \arrow[r, no head] & A_{0,n-1} \arrow[dd, "\wr"] \arrow[dr, no head] \\
X \arrow[dr, no head] \arrow[ur, no head] & && & & Y \\
& A_{1,1} \arrow[r, no head] & A_{1,2} \arrow[r, no head] & \cdots \arrow[r, no head] & A_{1,n-1} \arrow[ur, no head]
\end{tikzcd}
\] and argue that the top and bottom zig zags define the same morphism in $\cK[\we^{-1}]$. This is an easy exercise in diagram chasing, applying the rules of Definition \ref{defn:cat-of-fractions}.
\end{proof}

%\begin{exc}{\color{darkred} From \cite[\S 11]{DHKS} maps in the homotopy category of a model category are representable as 3-arrow zig zags (in fact there is a 3-arrow calculus) with the equivalence relation between them also as a three-arrow zig-zag of weak equivalences (try to reverse the direction here). This suggests the sense in which the hammock localization presents the homotopy category with data in place of relations }
%\end{exc} IF ADDING THIS BACK, PUT IT AFTER THE NEXT SENTENCE

The previous result applies to a model category $(\cM,\we)$ in which case we see that $\sL^H(\cM,\we)$ is a higher dimensional incarnation of the homotopy category, equipping $\cM[\we^{-1}]$ with mapping spaces whose path components correspond to arrows in the category of fractions.  
 A further justification that the mapping spaces of the hammock localization have the correct homotopy type, not just the correct sets of path components, proceeds as follows.  A \textbf{simplicial resolution} of $Y \in \cM$ is a Reedy fibrant simplicial object $Y_\bullet$ together with a weak equivalence $Y \wto Y_0$. \textbf{Cosimplicial resolutions} $X^\bullet \to X$ are defined dually. Every object has a simplicial and cosimplicial resolution, defined as the Reedy fibrant  replacement of the constant simplicial object in $\cM^{\DDelta^\op}$ and the Reedy cofibrant replacement of the constant cosimplicial object in $\cM^{\DDelta}$, respectively.

\begin{prop}[{Dwyer-Kan \cite[4.4]{DK-function}}]\label{prop:hammock-gives-spaces} For any cosimplicial resolution $X^\bullet \to X$ and simplicial resolution $Y \to Y_\bullet$, the diagonal of the bisimplicial set $\cM(X^\bullet, Y_\bullet)$ has the same homotopy type of $\Map_{\sL^H(\cM,\we)}(X,Y)$, and if $X$ or $Y$ are respectively cofibrant or fibrant the simplicial sets $\cM(X,Y_\bullet)$ and $\cM(X^\bullet, Y)$ do as well.
\end{prop}

As a corollary of this result one can show:

\begin{prop}[{Dwyer-Kan \cite[4.7, 4.8]{DK-function}}]\label{prop:hammock-of-model} Let $(\cM,\we)$ be the homotopical category underlying a simplicial model category $\cM$. Then for cofibrant $X$ and fibrant $Y$, $\Map_{\cM}(X,Y)$ and $\Map_{\sL^H(\cM,\we)}(X,Y)$ have the same homotopy type and hence the simplicial categories $\cM_{\mathrm{cf}}$ and $\sL^H(\cM,\we)$ are DK-equivalent.
\end{prop}

The statement of this result requires some explanation. If $\cK$ is a simplicial category whose underlying category of 0-arrows $\cK_0$ has a subcategory of weak equivalences $\we$, then these weak equivalences degenerate to define homotopical categories $(\cK_n,\we)$ for each category of $n$-arrows in $\cK$. For each $n$ we may form the hammock localization $\sL^H(\cK_n,\we)$.  As $n$ varies, this gives a bisimplicial sets of mapping complexes for each fixed pair of objects of $\cK$. The mapping complexes in the hammock localization $\sL^H(\cK,\we)$ are defined to be the diagonals of these bisimplicial sets. In the case of a simplicial model category $\cM$, the hammock localization $\sL^H(\cM,\we)$ is DK-equivalent to the hammock localization $\sL^H(\cM_0,\we)$ of the underlying unenriched homotopical category.

\begin{prop}[{\cite[5.4]{DK-function}}]\label{prop:DK-generalizes-quillen} A Quillen equivalence
\[\begin{tikzcd}
\cM \arrow[r, bend left, "F"] \arrow[r, phantom, "\perp"] & \cN \arrow[l, bend left, "G"]
\end{tikzcd}
\] induces DK-equivalences
\[ \sL^H(\cM_{\mathrm{c}},\we) \wto \sL^H(\cN_{\mathrm{c}},\we) \qquad \sL^H(\cN_{\mathrm{f}},\we) \wto \sL^H(\cM_{\mathrm{f}},\we)\]
Moreover, for any model category the inclusions
\[\sL^H(\cM_{\mathrm{c}},\cof \cap \we) \wto \sL^H(\cM_{\mathrm{c}},\we_{\mathrm{c}}) \wto \sL^H(\cM,\we)\]
are DK-equivalences  and hence $\sL^H(\cM,\we)$ and $\sL^H(\cN,\we)$ are DK-equivalent.
\end{prop}

\subsection{A model structure for homotopy coherent diagrams}\label{ssec:bergner}

Several of Dwyer and Kan's proofs of the results in the previous subsection make use of a model structure on the category of simplicial categories with a fixed set of objects and with identity-on-objects functors. But this restriction to categories with the same objects is somewhat unnatural. The Bergner model structure is the extension of Dwyer and Kan's model structure that drops that restriction, unifying the notions of DK-equivalence, free simplicial category (also known as ``simplicial computad''), and Kan complex enriched simplicial category, the importance of which will be made clear in \S\ref{sec:quasi}.

\begin{thm}[{Bergner \cite{bergner}}]\label{thm:bergner} There exists a model structure on the category of simplicially enriched categories whose:
\begin{itemize}
\item weak equivalences are the DK equivalences,%: those simplicial functors $F \colon \cC \to \cD$ that are
%\begin{itemize}
%\item define local weak homotopy equivalences $F \colon \cC(x,y) \to \cD(Fx,Fy)$ of simplicial sets for all pairs of objects $x,y \in \cC$
%\item essentially surjective and hence bijective on objects at the level of homotopy categories $F \colon \cat{h}\cC \to \cat{h}\cD$.
%\end{itemize}
\item cofibrant objects are the \textbf{simplicial computads}: those simplicial categories that, when considered as a simplicial object $\cC_\bullet \colon \DDelta^\op \to \cat{Cat}$ have the property that:
\begin{itemize}
\item each category $\cC_n$ is freely generated by the reflexive directed graph of its \textbf{atomic arrows}, those admitting no non-trivial factorizations
\item the degeneracy operators $[m] \twoheadrightarrow [n]$ in $\DDelta$ preserve atomic arrows, and 
\end{itemize}
\item fibrant objects are the \textbf{Kan complex enriched categories}: those simplicial categories whose mapping spaces  are all Kan complexes.
\end{itemize}
\end{thm}

More generally, the cofibrations in the Bergner model structure are retracts of relative simplicial computads and the fibrations are those functors that are local Kan fibrations and define isofibrations at the level of homotopy categories; see \cite{bergner} for more details.

Definition \ref{defn:nice-htpy-cat} tells us that maps in the homotopy category of the Bergner model structure from a simplicial category $\cA$ to a simplicial category $\cK$ are represented by simplicial functors from a cofibrant replacement of $\cA$ to a fibrant replacement of $\cK$. These are classically studied objects. Cordier and Porter after Vogt define such functors to be \textbf{homotopy coherent diagrams} of shape $\cA$ in $\cK$ \cite{CP-vogt}. 

A particular model for the cofibrant replacement of a strict 1-category $\cA$ regarded as a discrete simplicial category gives some intuition for the data involved in defining a homotopy coherent diagram. This construction, introduced by Dwyer and Kan under the name ``standard resolutions'' \cite[2.5]{DK-simplicial}, can be extended to the case where $\cA$ is non-discrete by applying it levelwise and taking diagonals.

\begin{defn}[free resolutions]\label{defn:free-resolution}
There is a comonad $(F,\epsilon,\delta)$ on the category of categories that sends a small category to the free category on its underlying reflexive directed graph. Explicitly $F\cA$ has the same objects as $\cA$ and its non-identity arrows are strings of composable non-identity arrows of $\cA$. 

Adopting the point of view of Exercise \ref{exc:simplicial-category}, we define a simplicial category $\gC\cA_\bullet$ with $\ob\gC\cA=\ob\cA$ and with the category of $n$-arrows $\gC\cA_n := F^{n+1}\cA$. A non-identity $n$-arrow is a string of composable arrows in $\cA$ with each arrow in the string enclosed in exactly $n$ pairs of well-formed parentheses. In the case $n=0$, this recovers the previous description of the non-identity 0-arrows in $F\cA$, strings of composable non-identity arrows of $\cA$. 

The required identity-on-objects functors in the simplicial object $\gC\cA_\bullet$ are defined by evaluating the comonad resolution for $(F, \epsilon, \delta)$ on a small category $\cA$. 
\[
\begin{tikzcd}
\gC\cA_\bullet := &
F\cA\arrow[r, tail] & \arrow[l, shift left=0.75em, two heads] \arrow[l, shift right=0.75em, two heads]  F^2\cA  \arrow[r, shift left=0.75em, tail] \arrow[r, shift right=0.75em, tail] &   \arrow[l, two heads] \arrow[l, two heads, shift right=1.5em] \arrow[l, two heads, shift left=1.5em]   F^3\cA \arrow[r, tail] \arrow[r, tail, shift right=1.5em] \arrow[r, tail, shift left=1.5em]   &\arrow[l, two heads, shift right=0.75em] \arrow[l, two heads, shift right=2.25em] \arrow[l, two heads, shift left=2.25em]  \arrow[l, two heads, shift left=0.75em] F^4\cA  & \cdots
\end{tikzcd}
\]
Explicitly, for $j \geq 1$, the face maps \[F^k\epsilon F^j  \colon F^{k+j+1}\cA \to F^{k+j}\cA\] remove the parentheses that are contained in exactly $k$ others, while $F^{k+j}\epsilon$ composes the morphisms inside the innermost parentheses. For $j \geq 1$, the degeneracy maps \[F^k\delta F^j\colon F^{k+j+1}\cA \to F^{k+j+2}\cA\] double up the parentheses that are contained in exactly $k$ others, while $F^{k+j}\delta$ inserts parentheses around each individual morphism.
\end{defn}

\begin{exc} Explain the sense in which free resolutions define Bergner cofibrant replacements of strict 1-categories by:
\begin{enumerate}
\item verifying that for any $\cA$, the free resolution $\gC\cA_\bullet$ is a simplicial computad, and
\item defining a canonical identity-on-objects augmentation functor $\epsilon\colon \gC\cA \to \cA$ and verifying that it defines a local homotopy equivalence.
\end{enumerate}
\end{exc}

The notation $\gC\cA_\bullet$ for the free resolution is non-standard and will be explained in \S\ref{ssec:qcat-nature}, where we will gain a deeper understanding of the importance of the Bergner model structure from the vantage point of $(\infty,1)$-categories.

\section{Quasi-categories as $(\infty,1)$-categories}\label{sec:quasi}

Any topological space $Y$ has an associated simplicial set $\fun{Sing}(Y)$ called its \textbf{total singular complex}. The vertices in $\fun{Sing}(Y)$ are the points in $Y$ and the 1-simplices are the paths; in general, an $n$-simplex in $\fun{Sing}(Y)$ corresponds to an $n$-simplex in $Y$, that is, to a continuous map $|\Delta^n| \to Y$. In particular, a 2-simplex $|\Delta^2| \to Y$ defines a triangular shaped homotopy from the composite  paths along the spine $\Lambda^2_1 \subset \Delta^2$ of the 2-simplex to the direct path from the 0th to the 2nd vertex that is contained in its 1st face.\footnote{The simplicial $n$-simplex $\Delta^n$, its boundary sphere $\partial\Delta^n$, and its horns $\Lambda^n_k$ are defined in \S\ref{ssec:simp-sets}.} Since the inclusion $|\Lambda^n_k|\to|\Delta^n|$ admits a retraction, $\fun{Sing}(Y)$ is a Kan complex.

The total singular complex is a higher-dimensional incarnation of some of the basic invariants of $Y$, which can be recovered by truncating the total singular complex at some level and replacing the top-dimensional simplices with suitably defined ``homotopy classes'' of such. Its set of path components is the set $\pi_0Y$ of path components in $Y$.  Its homotopy category, in a sense to be defined below, comprised of the vertices and homotopy classes of paths between them, is a groupoid $\pi_1Y$ called the  \emph{fundamental groupoid} of $Y$. By extension, it is reasonable to think of the higher dimensional simplices of $\fun{Sing}(Y)$ as being invertible in a similar sense, with composition relations witnessed by higher cells. In this way, $\fun{Sing}(Y)$ models the $\infty$-groupoid associated to the topological space $Y$ and the Quillen equivalence \ref{thm:geo-sing-equiv} is one incarnation of Grothendieck's famous ``homotopy hypothesis'' (the moniker due to  Baez), that $\infty$-groupoids up to equivalence should model homotopy types \cite{grothendieck}.

In the catalogue of weak higher-dimensional categories, the $\infty$-groupoids define $(\infty,0)$-\textbf{categories}, weak categories with morphisms in each dimensional all of which are weakly invertible. In \S\ref{ssec:quasi}, we  introduce \textbf{quasi-categories}, which provide a particular model for $(\infty,1)$-categories --- infinite-dimensional categories in which every morphism above dimension 1 is invertible --- in parallel with the Kan complex model for $(\infty,0)$-categories. We explain the sense in which quasi-categories, which are defined to be simplicial sets with an inner horn lifting property, model $(\infty,1)$-categories by introducing the homotopy category of a quasi-category and constructing the hom-space between objects in a quasi-category. In \S\ref{ssec:qcat-nature}, we explain how simplicially enriched categories like those considered in \S\ref{sec:hammock} can be converted into quasi-categories. Then in \S\ref{ssec:joyal}, we introduce a model structure whose fibrant objects are the quasi-categories due to Joyal and in this way obtain a suitable notion of (weak) equivalence between quasi-categories.

\subsection{Quasi-categories and their homotopy categories}\label{ssec:quasi}

The \textbf{nerve} of a small category $\cD$ is the simplicial set $\cD_\bullet$ whose vertices $\cD_0$ are the objects of $\cD$, whose 1-simplices $\cD_1$ are the morphisms, and whose set of $n$-simplices $\cD_n$ is the set of $n$ composable pairs of morphisms in $\cD$. The simplicial structure defines a diagram in $\cat{Set}$
\[
\begin{tikzcd}
\cdots & \cD_3  \arrow[r, shift right=.5em] \arrow[r, shift right=1.5em] \arrow[r, shift left=.5em] \arrow[r, shift left=1.5em] & \cD_2 \arrow[l] \arrow[l, shift right=1em] \arrow[l, shift left=1em] \arrow[r] \arrow[r, shift left=1em] \arrow[r, shift right=1em] & \cD_1 \arrow[l, shift left=.5em] \arrow[l, shift right=.5em] \arrow[r, shift left=.5em] \arrow[r, shift right=.5em] & \cD_0 \arrow[l]
\end{tikzcd}
\]
Truncating at level 2 we are left with precisely the data that defines a small category $\cD$ as a category internal to the category of sets and in fact this higher-dimensional data is redundant in a sense: the simplicial set $\cD_\bullet$ is 2-coskeletal, meaning any sphere bounding a hypothetical simplex of dimension at least 3 admits a unique filler.

The description of the nerve as an internal category relies on an isomorphism $\cD_2 \cong \cD_1 \times_{\cD_0} \cD_1$ identifying the set of 2-simplices with the pullback of the domain and codomain maps $\cD_1 \rightrightarrows \cD_0$: a composable pair of arrows is given by a pair of arrows so that the domain of the second equals the codomain of the first. Equivalently, this condition asserts that the map
\[
\begin{tikzcd} \Lambda^2_1 \arrow[r] \arrow[d, tail] & \cD_\bullet \\ \Delta^2 \arrow[ur, dashed, "\exists !"']
\end{tikzcd}
\]
admits a unique filler. In higher dimensions, we can consider the inclusion of the spine
$\Delta^1 \cup_{\Delta^0} \cdots  \cup_{\Delta^0} \Delta^1 \hookrightarrow\Delta^n$ of an $n$-simplex, and similarly the nerve $\cD_\bullet$ will admit unique extensions along these maps. From the perspective of an infinite dimensional category, in which the higher dimensional simplices represent data and not just conditions on the one simplices, it is better to consider extensions along inner horn inclusions $\Lambda^n_k\hookrightarrow\Delta^n$ for the reasons explained by the following exercise.

\begin{exc} Prove that the spine inclusions can be presented as cell complexes (see Definition \ref{defn:cell-complex}) built from the inner horn inclusions $\{\Lambda^k_n\hookrightarrow\Delta^n\}_{n \geq 2, 0 < k< n}$ but demonstrate by example that the inner horn inclusions cannot be presented as cell complexes built from the spine inclusions.
\end{exc}

The original definition of a simplicial set satisfying the ``restricted Kan condition,''  now called a \emph{quasi-category} (following Joyal \cite{joyal-quasi}) or an \emph{$\infty$-category} (following Lurie \cite{lurie-topos}), is due to Boardman and Vogt \cite{BV}. Their motivating example appears as Corollary \ref{cor:coherent-diagram-quasi-category}.

\begin{defn} A \textbf{quasi-category} is a simplicial set $X$ such that $X \to *$ has the right lifting property with respect to the inner horn inclusions  for each $n \geq 2$, $0 < k < n$. \begin{equation}\label{eq:qcatdefn}
\begin{tikzcd}
 \Lambda^n_k \arrow[r] \arrow[d, tail] & X \\ \Delta^n \arrow[ur, dashed]
 \end{tikzcd}
 \end{equation}
\end{defn}

Nerves of categories are quasi-categories; in fact in this case each lift \eqref{eq:qcatdefn} is unique. Tautologically, Kan complexes are quasi-categories. In particular, the total singular complex of a topological space is a Kan complex and hence a quasi-category. More sophisticated examples of (frequently large) quasi-categories are produced by Theorem \ref{thm:nerve-to-quasi} below.

\begin{defn}[{the homotopy category of a quasi-category \cite[4.12]{BV}}]\label{defn:qcat-htpy-cat}
Any quasi-category $X$ has an associated  \textbf{homotopy category} $hX$ whose objects are the vertices of $X$ and whose morphisms are represented by 1-simplices, which we consequently depict as arrows $f \colon x \to y$ from their 0th vertex to their 1st vertex. The degenerate 1-simplices serve as identities in the homotopy category which may be depicted using an equals sign in place of the arrow. 

As the name would suggest, the morphisms in $hX$ are homotopy classes of 1-simplices, where a pair of 1-simplices $f$ and $g$ with common boundary are \textbf{homotopic}  if there exists a 2-simplex whose boundary has any of the following forms: 
\begin{equation}\label{eq:1simphtpy}
\begin{tikzcd}[column sep=small]
~ & \bullet \arrow[dr, "f"] \arrow[d, phantom, "\sim" near end] & & & \bullet \arrow[dr, equals]  \arrow[d, phantom, "\sim" near end] & & & \bullet \arrow[dr, "g"] \arrow[d, phantom, "\sim" near end] & & & \bullet \arrow[dr, equals]  \arrow[d, phantom, "\sim" near end] \\ \bullet \arrow[ur, equals] \arrow[rr, "g"']  & ~& \bullet & \bullet \arrow[ur, "f"] \arrow[rr, "g"'] &~ & \bullet & \bullet \arrow[ur, equals] \arrow[rr, "f"'] & ~& \bullet & \bullet \arrow[ur, "g"] \arrow[rr, "f"'] &~ & \bullet
 \end{tikzcd}
 \end{equation}
Indeed, in a quasi-category, if any of the 2-simplices \eqref{eq:1simphtpy} exists then there exists a 2-simplex of each type.

Generic 2-simplices in $X$ \begin{equation}\label{eq:compwitness}
\begin{tikzcd}[column sep=small] & \bullet\arrow[dr, "g"]  \arrow[d, phantom, "\sim" near end]  \\ \bullet \arrow[ur, "f"] \arrow[rr, "h"'] &~ & \bullet
\end{tikzcd}
\end{equation} witness that $gf=h$ in the homotopy category. Conversely, if $h=gf$ in $hX$ and $f, g, h$ are any 1-simplices representing these homotopy classes, then there exists a 2-simplex \eqref{eq:compwitness} witnessing the composition relation. 
\end{defn}

\begin{exc} $\quad$
\begin{enumerate}
\item Verify the assertions made in Definition \ref{defn:qcat-htpy-cat} or see \cite[\S 1.2.3]{lurie-topos}.
\item Show that $h$ is the left adjoint to the nerve functor:\footnote{In fact, this pair defines a Quillen adjunction between the model structure to be introduced in Theorem \ref{thm:quasimodel} and the ``folk'' model structure on categories \cite[15.3.8]{riehl-cathtpy}.}
\[ 
\begin{tikzcd}
\cat{qCat} \arrow[r, bend left, "h" pos=.45] \arrow[r, phantom, "\perp"] & \cat{Cat} \arrow[l, bend left, "N" pos=.6]
\end{tikzcd}
\]
\end{enumerate}
\end{exc}%Here we have restricted the domain of the left adjoint to $\cat{qCat} \subset \cat{sSet}$, the full subcategory of quasi-categories. The definition of the homotopy category associated to a generic simplicial set  is slightly more complicated. 

The mapping space between two objects of a quasi-category $A$ is modeled by the Kan complex defined via the pullback
\[
\begin{tikzcd}
\Map_A(x,y) \arrow[r] \arrow[d, two heads] \arrow[dr, phantom, "\lrcorner" very near start] & A^{\Delta^1} \arrow[d, two heads] \\ \Delta^0 \arrow[r, "{(x,y)}"'] & A \times A
\end{tikzcd}
\]

The following proposition of Joyal is useful in proving that $\Map_A(x,y)$ is a Kan complex and also characterizes the $\infty$-groupoids in the quasi-categorical model of $(\infty,1)$-categories.

\begin{prop}[{Joyal \cite[1.4]{joyal-quasi}}] A quasi-category is a Kan complex if and only if its homotopy category is a groupoid.
\end{prop}

\begin{defn}[{\cite[1.6]{joyal-quasi}}]\label{defn:quasi-iso} A 1-simplex $f$ in a quasi-category $X$ is an \textbf{isomorphism} if and only if it represents an isomorphism in the homotopy category, or equivalently  if and only it admits a coherent homotopy inverse:
\[ \begin{tikzcd} \2 \arrow[r, "f"] \arrow[d, hook] & X \\ \iso \arrow[ur, dashed]
\end{tikzcd}
\] extending along the map $\2 \hookrightarrow\iso$ including the the nerve of the free-living arrow into the nerve of the free-living isomorphism.
\end{defn}

\subsection{Quasi-categories found in nature}\label{ssec:qcat-nature}

Borrowing notation from the simplex category $\DDelta$, we write $[n] \subset \bbomega$ for the ordinal category ${\mathbb{n}}+1$, the full subcategory spanned by $0,\ldots, n$ in the category that indexes a countable sequence: 
\[
\begin{tikzcd}
{[n]} := & 0 \arrow[r] & 1 \arrow[r] & 2 \arrow[r] & 3 \arrow[r] & \cdots \arrow[r] & n
\end{tikzcd}
\]
 These categories define the objects of a diagram $\DDelta\hookrightarrow\cat{Cat}$ that is a full embedding: the only functors $[m] \to [n]$ are order-preserving maps from $[m] = \{0,\ldots, m\}$ to $[n] = \{0,\ldots, n\}$. Applying the free resolution construction of Definition \ref{defn:free-resolution} to these categories we get a functor $\gC \colon \DDelta \to \cat{sCat}$ where $\gC[n]$ is the full simplicial subcategory of $\gC\bbomega$ spanned by those objects $0,\ldots,n$. 

\begin{defn}[homotopy coherent realization and nerve]\label{defn:hocoh-adj} The homotopy coherent nerve $\gN$ and homotopy coherent realization $\gC$ are the adjoint pair of functors obtained by applying Kan's construction \cite[1.5.1]{riehl-cathtpy} to the functor
    $\gC\colon \DDelta \to \cat{sCat}$ to construct an adjunction
    \[
\begin{tikzcd}
\cat{sSet} \arrow[r,bend left, "\gC"] \arrow[r, phantom, "\perp"]  & \cat{sCat} \arrow[l, bend left, "\gN"]
\end{tikzcd}
\]

The right adjoint,  called the \textbf{homotopy coherent nerve}, converts a simplicial category $\cS$ into a simplicial set $\gN\cS$ whose $n$-simplices  are homotopy coherent diagrams of shape $[n]$ in $\cS$. That is  \[\gN \cS_n := \{ \gC[n] \to \cS\}.\]

The left adjoint is defined by pointwise left Kan extension along the
  Yoneda embedding:
\[
\begin{tikzcd}
      {\DDelta}\arrow[rr, hook, "\yo"] \arrow[dr, "{\gC}"'] &
\arrow[d, phantom, "\cong" ] & {\cat{sSet}}\arrow[dl, dashed, "\gC"] \\
      & {\cat{sCat}} &
    \end{tikzcd}
\]
That is, $\gC\Delta^n$ is defined to be $\gC[n]$ --- a simplicial category that we call the \textbf{homotopy coherent $n$-simplex} --- and for a generic simplicial set $X$, $\gC{X}$ is defined to be a colimit of the homotopy coherent simplices indexed by the category of simplices of $X$.\footnote{The simplicial set $X$ is obtained by gluing in a $\Delta^n$ for each $n$-simplex $\Delta^n \to X$ of $X$. The functor $\gC$ preserves these colimits, so $\gC{X}$ is obtained by gluing in a $\gC[n]$ for each $n$-simplex of $X$.}  Because of the formal similarity with the geometric realization functor, another left adjoint defined by Kan's construction, we refer to $\gC$ as \textbf{homotopy coherent realization}.   
\end{defn}

Many examples of quasi-categories fit into the following paradigm.

\begin{thm}[{\cite[2.1]{CP-vogt}}]\label{thm:nerve-to-quasi}
 If $\cS$ is Kan complex enriched, then $\gN\cS$ is a quasi-category.
 \end{thm}
 
 In particular, in light of Exercise \ref{exc:kan-enriched-simp-model}, the quasi-category associated to a simplicial model category $\cM$ is defined  to be $\gN\cM_{\mathrm{cf}}$.

%\begin{proof}
%By adjunction, to extend along an inner horn inclusion $\Lambda^n_k\hookrightarrow\Delta^n$ mapping into the homotopy coherent nerve $\gN\cS$ is to  extend along simplicial subcomputad inclusions $\gC\Lambda^n_k\hookrightarrow\gC\Delta^n$ mapping into the Kan complex enriched category $\cS$. This is the simplicial subcomputad generated by all arrows whose beads are supported by simplices in $\Lambda^n_k\subset\Delta^n$. The only missing ones are in the mapping space from $0$ to $n$, so we are asked to solve a single lifting problem
%\[
%\begin{tikzcd}
%\gC\Lambda^n_k(0,n) \arrow[r] \arrow[d, hook] & \Map(X_0,X_n) \\ \gC\Delta^n(0,n) \arrow[ur, dashed]
%\end{tikzcd}
%\] 
%The space $\gC\Delta^n(0,n) \cong (\Delta^1)^{n-1}$ is a cube, while $\gC\Lambda^n_k(0,n)$ is a cubical horn. Cubical horn inclusions can be filled in the Kan complex $\Map(X_0,X_n)$, completing the proof; see \cite[4.4.4]{RV6} for more details.
%\end{proof}

Recall from \S\ref{ssec:bergner} that a \textbf{homotopy coherent diagram} of shape $\cA$ in a Kan complex enriched category $\cS$ is a functor $\gC\cA \to \cS$. Similarly,  a \textbf{homotopy coherent natural transformation} $\alpha \colon F \to G$ between homotopy coherent diagrams $F$ and $G$ of shape $\cA$ is a homotopy coherent diagram of shape $\cA \times [1]$ that restricts on the endpoints of $[1]$ to $F$ and $G$ as follows:
\[
\begin{tikzcd} \gC\cA \arrow[r, tail, "0"] \arrow[dr, "F"'] &  \gC(\cA \times [1]) \arrow[d, "\alpha"] & \gC\cA \arrow[l, tail, "1"'] \arrow[dl, "G"] \\ & \cS 
\end{tikzcd}
\]

Note that the data of a pair of homotopy coherent natural transformations $\alpha \colon F \to G$ and $\beta \colon G \to H$ between homotopy coherent diagrams of shape $\cA$ does not uniquely determine a (vertical) ``composite'' homotopy coherent natural transformation $F \to H$ because this data does not define a homotopy coherent diagram of shape $\cA \times [2]$, where $[2] = 0 \to 1 \to 2$. Here $\alpha$ and $\beta$ define a diagram of shape $\gC(\cA\times \Lambda^2_1)$ rather than a diagram of shape $\gC(\cA\times[2])$, where $\Lambda^2_1$ is the shape of the generating reflexive directed graph of the category $[2]$. This observation motivated Boardman and Vogt to define, in place of a \emph{category} of homotopy coherent diagrams and natural transformations of shape $\cA$, a \emph{quasi-category} of homotopy coherent diagrams and natural transformations of shape $\cA$. 

For any category $\cA$, let $\fun{Coh}(\cA,\cS)$ denote the simplicial set whose $n$-simplices are homotopy coherent diagrams of shape $\cA \times [n]$, i.e., are simplicial functors
\[ \gC(\cA\times[n]) \to \cS.\]

\begin{cor}\label{cor:coherent-diagram-quasi-category}
$\fun{Coh}(\cA,\cS) \cong \gN\cS^\cA$  is a quasi-category.
\end{cor}
\begin{proof}
By the adjunction of Definition \ref{defn:hocoh-adj}, a simplicial functor $\gC \cA \to \cS$ is the same as a simplicial map $\cA \to \gN\cS$. So $\fun{Coh}(\cA,\cS) \cong \gN\cS^\cA$ and since the quasi-categories define an exponential ideal  in simplicial sets as a consequence of the cartesian closure of the Joyal model structure of Theorem \ref{thm:quasimodel}, the fact that $\gN\cS$ is a quasi-category implies that $\gN\cS^\cA$ is too.
\end{proof}

\begin{rmk}[all diagrams in homotopy coherent nerves are homotopy coherent]\label{rmk:coherent-diagram-qcat}
This corollary explains that any map of simplicial sets $X \to \gN\cS$ transposes to define a simplicial functor $\gC{X} \to \cS$, a homotopy coherent diagram of shape $X$ in $\cS$. While not every quasi-category is isomorphic to a homotopy coherent nerve of a Kan complex enriched category, every quasi-category is equivalent to a homotopy coherent nerve; one proof appears as \cite[7.2.2]{RV6}. This explains the slogan  that ``all diagrams in quasi-categories are homotopy coherent.''
\end{rmk}

%\begin{ex}[stable quasi-categories of chain complexes]
%\cite{lurie-algebra}
%\end{ex}

\subsection{The Joyal model structure}\label{ssec:joyal}

In analogy with Quillen's model structure of Theorem \ref{thm:sset-model}, in which the fibrant objects are the Kan complexes and the cofibrations are the monomorphisms, we might hope that there is another model structure on $\cat{sSet}$ whose fibrant objects are the quasi-categories and with the monomorphisms as cofibrations, and indeed this is true (and hence by Exercise \ref{exc:model-given-by}\eqref{itm:model-iv} is unique with these properties).

The weak equivalences in this  hoped-for model structure for quasi-categories can be described using a particularly nice cylinder object. Let $\iso$ be the nerve of the free-standing isomorphism $\iso$; the name is selected because $\iso$ is something like an interval. 

\begin{prop}\label{prop:quasi-cyl} For any simplicial set $A$, the evident inclusion and projection maps define a cylinder object
\[
\begin{tikzcd}[row sep=small] A \sqcup A \arrow[dr, tail, "{(i_0, i_1)}"'] \arrow[rr, "{(1_A,1_A)}"] &  & A \\ &  A \times \iso \arrow[ur, two heads, "\sim" sloped, "\pi"'] 
\end{tikzcd}
\]
\end{prop}
\begin{proof}
The map $(i_0,i_1)\colon A \sqcup A \to A \times \iso$ is a monomorphism and hence a cofibration. To see that the projection is a trivial fibration, observe that it is a pullback of $\iso \to *$  as displayed below-left and hence by Lemma \ref{lem:closure} it suffices to prove that this latter map is a trivial fibration. To that end, we must show that there exist solutions to lifting problems displayed on the right
\[ 
\begin{tikzcd}
A \times \iso \arrow[d, "\pi"'] \arrow[r, "\pi"] \arrow[dr, phantom, "\lrcorner" very near start] & \iso \arrow[d] & &  \partial \Delta^n \arrow[d] \arrow[r] & \iso \arrow[d]  \\ A \arrow[r] & {\ast} & & \Delta^n \arrow[ur, dashed] \arrow[r] & {*}
 \end{tikzcd}
 \] 
When $n=0$ this is true because $\iso$ is non-empty. For larger $n$, we use the fact that $\iso \cong \cosk_0 \iso$. By adjunction, it suffices to show that $\iso$ lifts against $\sk_0 \partial\Delta^n \to \sk_0 \Delta^n$, but for $n >0$, the $0$-skeleton of $\Delta^n$ is isomorphic to that of its boundary.
\end{proof}

The proof of Joyal's model structure has been widely circulated in unpublished notes, and can also be found in the following sources \cite[2.2.5.1]{lurie-topos} or \cite[2.13]{ds-mapping}.

\begin{thm}[Joyal]\label{thm:quasimodel} There is a cartesian closed model structure on $\cat{sSet}$ whose \begin{itemize}
\item cofibrations are monomorphisms, 
\item weak equivalences are those maps $f\colon A \to B$ that induce bijections on the sets \[\Hom(B,X)_{/\sim_\ell} \to \Hom(A,X)_{/\sim_\ell}\] of maps into any quasi-category $X$ modulo the left homotopy relation relative to the cylinder just defined,
\item fibrant objects are precisely the quasi-categories,  and 
\item fibrations between fibrant objects are the \textbf{isofibrations}, those maps that lift against the inner horn inclusions and also the map $* \to \iso$. 
\end{itemize}
\end{thm}

By Proposition \ref{prop:structured-we}, a map between quasi-categories is a weak equivalence, or we say simply \textbf{equivalence} of quasi-categories,  if and only if it admits an inverse equivalence $Y \to X$
 together with an ``invertible homotopy equivalence'' using the notion of homotopy defined with the interval $\iso$. A map between nerves of strict 1-categories is an equivalence of quasi-categories if and only if it is an equivalence of categories, as usually defined. In general, every categorical notion for quasi-categories restricts along the full inclusion $\cat{Cat} \subset \cat{qCat}$ to the classical notion. This gives another sense in which quasi-categories model the $(\infty,1)$-categories introduced at the start of this section. However, quasi-categories are not the only model of $(\infty,1)$-categories as we shall now discover.

\section{Models of $(\infty,1)$-categories}\label{sec:models}

An $(\infty,1)$-category should have a set of objects $X_0$, a space of morphisms $X_1$, together with composition and identities that are at least weakly associative and unital. One idea of how this might be presented, due to Segal  \cite{segal-classifying}, is to ask that $X \in \cat{sSet}^{\DDelta^\op}$ is a simplicial space
\[
\begin{tikzcd}
\cdots & X_3  \arrow[r, shift right=.5em] \arrow[r, shift right=1.5em] \arrow[r, shift left=.5em] \arrow[r, shift left=1.5em] & X_2 \arrow[l] \arrow[l, shift right=1em] \arrow[l, shift left=1em] \arrow[r] \arrow[r, shift left=1em] \arrow[r, shift right=1em] & X_1 \arrow[l, shift left=.5em] \arrow[l, shift right=.5em] \arrow[r, shift left=.5em] \arrow[r, shift right=.5em] & X_0 \arrow[l]
\end{tikzcd}
\]
with $X_0$ still a set, so that for all $n$ the map
\begin{equation}\label{eq:segal-map} X_n \to X_1 \times_{X_0} \cdots \times_{X_0} X_1\end{equation}
induced on weighted limits from the spine inclusion $\Delta^1 \vee \cdots \vee \Delta^1 \to \Delta^n$, is a weak equivalence in a suitable sense. Segal points out that Grothendieck has observed that in the case where the spaces $X_n$ are discrete, these so-called Segal maps are isomorphisms if and only if $X$ is isomorphic to the nerve of a category.

In this section, we introduce various models of $(\infty,1)$-categories many of which are inspired by this paradigm. Before these models make their appearance in \S\ref{ssec:models}, we begin in \S\ref{ssec:toen} with an abbreviated tour of an  axiomatization due to To\"{e}n that characterizes a homotopy theory of $(\infty,1)$-categories.  
In \S\ref{ssec:cosmoi}, we then restrict our attention to four of the six models that are better behaved in the sense of providing easy access to the $(\infty,1)$-category $\Fun(A,B)$ of functors between $(\infty,1)$-categories $A$ and $B$.  Each of these models satisfy a short list of axioms that we exploit in \S\ref{sec:indep} to sketch a natively ``model-independent'' development of the category theory of $(\infty,1)$-categories.

\subsection{An axiomatization of the homotopy theory of $(\infty,1)$-categories}\label{ssec:toen}

 The homotopy theory of $\infty$-groupoids is freely generated under homotopy colimits by the point. We might try to adopt a similar ``generators and relations'' approach to build the homotopy theory of $(\infty,1)$-categories, taking the generators to be the category $\DDelta$, which freely generates simplicial spaces. The relations assert that the natural maps 
\begin{equation}\label{eq:relations} \Delta^1 \vee \cdots \vee \Delta^1 \to \Delta^n \qquad\qquad \iso \to \Delta^0\end{equation} induces equivalences upon mapping into an $(\infty,1)$-category. This idea motivates Rezk's complete Segal space model, which is the conceptual center of the To\"{e}n axiomatization of a model category $\cM$ whose fibrant objects model $(\infty,1)$-categories.

For simplicity we assume that $\cM$ is a combinatorial simplicial model category. In practice, these assumptions are relatively mild: in particular, if $\cM$ fails to be simplicial it is possible to define a Quillen equivalent model structure on $\cM^{\DDelta^\op}$ that is simplicial \cite{dugger-replacing}. The model category $\cM$ should be equipped with a functor $C \colon \DDelta \to \cM$ so that $C(0)$ represents a free point in $\cM$ while $C(1)$ represents a free arrow. This cosimplicial object is required to be a \textbf{weak cocategory} meaning that the duals of the Segal maps are equivalences
\[ C(1) \cup_{C(0)} \cdots \cup_{C(0)}  C(1) \wto C(n).\]
We state To\"{e}n's seven axioms without defining all the terms because to do so would demand too long of an excursion, and refer the reader to  \cite{toen} for more details.

\begin{thm}[{To\"{e}n \cite{toen}}] Let $\cM$ be a combinatorial simplicial model category equipped with a functor $C \colon \DDelta \to \cM$ satisfying the following properties.
\begin{enumerate}
\item Homotopy colimits are universal over \textbf{0-local objects}, those $X$ so that \[\Map(*,X) \wto \Map(C(1),X).\]
\item Homotopy coproducts are disjoint and universal.
\item $C$ is an \textbf{interval}: meaning $C(0)$ and the $C$-geometric realization of $\iso$ are contractible. 
\item For any weak category $X \in \cM^{\DDelta^\op}$ so that $X_0$ and $X_1$ are 0-local, $X$ is equivalent to the \v{C}ech nerve of the map $X_0 \to |X|_{\mathrm{c}}$.
\item For any weak category $X \in \cM^{\DDelta^\op}$ so that $X_0$ and $X_1$ are 0-local, the homotopy fiber of $X \to \RR\Hom(C,|X|_{\mathrm{c}})$ is contractible.
\item The point  and interval define a \textbf{generator}: $f \colon X \to Y$ is a weak equivalence in $\cM$ if and only if $\Map(C(0),X) \wto \Map(C(0),Y)$ and $\Map(C(1),X) \wto \Map(C(1),Y)$.
\item $C$ is \textbf{homotopically fully faithful}: $\DDelta([n],[m]) \wto \Map(C(n),C(m))$
\end{enumerate}
Then the functor $X \mapsto \Map(C(-),X)$ defines a right Quillen equivalence from $\cM$ to the model structure for complete Segal spaces on the category of bisimplicial sets.
\end{thm}

A similar axiomatization is given by Barwick and Schommer-Pries as a specialization of an axiomatization for $(\infty,n)$-categories \cite{BSP}.

\subsection{Models of $(\infty,1)$-categories}\label{ssec:models}

We now introduce six models of $(\infty,1)$-cat\-e\-gor\-ies each of which arise as the fibrant objects in a model category that is Quillen equivalent to all of the others. Two of these models --- the quasi-categories and the Kan complex enriched categories --- have been presented already in Theorems \ref{thm:quasimodel} and \ref{thm:bergner}. 

A \textbf{Segal category} is a Reedy fibrant bisimplicial set $X \in \cat{sSet}^{\DDelta^\op}$ so that the Segal maps \eqref{eq:segal-map} are trivial fibrations and $X_0$ is a set.\footnote{In \cite[\S 7]{DKS-homotopy} the Reedy fibrancy condition, which implies that the Segal maps are Kan fibrations, is dropped and the Segal maps are only required to be weak equivalences.}

\begin{thm}[{Hirschowitz-Simpson \cite{hirschowitz-simpson, simpson-homotopy}, Pellissier \cite{pellissier}, Bergner \cite{bergner-segal}}] There is a cartesian closed model structure on the category of bisimplicial sets with discrete set of objects whose
\begin{itemize}
\item cofibrations are the monomorphisms
\item fibrant objects are the Segal categories that are Reedy fibrant as simplicial spaces
\item weak equivalences are the DK-equivalences (in a suitable sense).
\end{itemize}
\end{thm}

A \textbf{complete Segal space} is similarly a Reedy  fibrant bisimplicial set $X \in \cat{sSet}^{\DDelta^\op}$ so that the Segal maps \eqref{eq:segal-map} are trivial fibrations. In this model, the discreteness condition on $X_0$ is replaced with the so-called \textbf{completeness} condition, which is again most elegantly phrased using weighted limits: it asks either that the map $\{\iso, X\} \to \{ \Delta^0,X\} \cong X_0$ is a trivial fibration or that the map $X_0 \to \{\iso, X\}$ is an equivalence. Intuitively this says that the spacial structure of $X_0$ is recovered by the $\infty$-groupoid of $\{\iso,X\}$ of isomorphisms in $X$. 

\begin{thm}[{Rezk \cite{rezk}}] There is a cartesian closed model structure on the category of bisimplicial sets whose 
\begin{itemize}
\item cofibrations are the monomorphisms
\item fibrant objects are the complete Segal spaces
\item weak equivalences are those maps $u \colon A \to B$ so that for every complete Segal space $X$, the maps $X^B \to X^A$ are weak homotopy equivalences of simplicial sets upon evaluating at 0.
\end{itemize}
\end{thm}

A \textbf{marked simplicial set} is a simplicial set with a collection of marked edges containing the degeneracies; maps must then preserve the markings. A quasi-category is naturally a marked simplicial set whose marked edges are precisely the isomorphisms, described in Definition \ref{defn:quasi-iso}.

\begin{thm}[{Verity \cite{verity-weak-i}, Lurie \cite{lurie-topos}}] There is a cartesian closed model structure on the category of marked simplicial sets whose 
\begin{itemize}
\item cofibrations are the monomorphisms
\item fibrant objects are the naturally marked quasi-categories
\item weak equivalences are those maps $A \to B$ so that for all naturally marked quasi-categories $X$ the map $X^B \to X^A$ is a homotopy equivalence of maximal sub Kan complexes.
\end{itemize}
\end{thm}

A \textbf{relative category} is a category equipped with a wide subcategory of weak equivalences. A morphism of relative categories is a homotopical functor. A weak equivalence of relative categories is a homotopical functor $F \colon (\cC,\we) \to (\cD,\we)$ that induces a DK-equivalence on hammock localizations $\sL^H(\cC,\we) \to L^H(\cD,\we)$. 

\begin{thm}[{Barwick-Kan \cite{BK-relative}}] There is a model structure for relative categories whose
\begin{itemize}
\item weak equivalences are the relative DK-equivalences just defined
\end{itemize}
and whose  cofibrations and fibrant objects are somewhat complicated to describe.
\end{thm}

Each of these model categories, represented in the diagram below by their subcategories of fibrant objects,  are Quillen equivalent, connected via right Quillen equivalences as displayed below:\footnote{The right Quillen equivalences from relative categories are in \cite{BK-relative}. The Quillen equivalences involving complete Segal spaces, Segal categories, and quasi-categories can all be found in \cite{JT}. Proofs that the homotopy coherent nerve defines a Quillen equivalence from simplicial categories to quasi-categories can be found in \cite{lurie-topos} and \cite{ds-mapping}. A zig-zag of Quillen equivalences between simplicial categories and Segal categories is constructed \cite{bergner-three}. The right Quillen equivalence from naturally marked quasi-categories to the Joyal model structure can be found in \cite{lurie-topos} and \cite{verity-weak-i}.
}
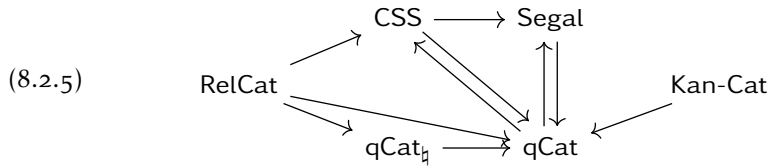
\begin{equation}\label{eq:1-comparison}
\begin{tikzcd}[row sep=small]
& \cat{CSS} \arrow[r] \arrow[ddr, shift left=.25em] & \cat{Segal} \arrow[dd, shift left=.25em] \\ \cat{RelCat} \arrow[drr] \arrow[ur] \arrow[dr] & & & \cat{Kan}\text{-}\cat{Cat} \arrow[dl] \\
& \cat{qCat}_\natural \arrow[r] & \cat{qCat} \arrow[uul, shift left=.25em] \arrow[uu, shift left=.25em]
\end{tikzcd}
\end{equation}

A nice feature of the simplicial category and relative category models is that their objects and morphisms are strictly-defined, as honest-to-goodness enriched categories in the former case and honest-to-goodness homotopical categories in the latter. From this vantage point it is quite surprising that they are Quillen equivalent to the weaker models. But there are some costs paid to obtain this extra strictness: neither model category is cartesian closed, so both contexts lack a suitable internal hom, whereas the other four models --- the quasi-categories, Segal categories, complete Segal spaces, and naturally marked quasi-categories --- all form cartesian closed model categories. Consequently, in each of these models the $(\infty,1)$-categories define an exponential ideal: if $A$ is fibrant and $X$ is cofibrant, then $A^X$ is again fibrant  and moreover the maps induced on exponentials by the maps \eqref{eq:relations} are weak equivalences.

\subsection{$\infty$-cosmoi of $(\infty,1)$-categories}\label{ssec:cosmoi}

From the cartesian closure of the model categories for quasi-categories, Segal categories, complete Segal spaces, and naturally marked quasi-categories, it is possible to induce a secondary enrichment, in the sense of Definition \ref{defn:Vmodelcat}, on these model categories:

\begin{thm}[{\cite[2.2.3]{RV4}}] The model structures for quasi-categories, complete Segal spaces, Segal categories, and naturally marked quasi-categories  are all enriched over the model structure for quasi-categories.
\end{thm}

The following definition of an $\infty$-cosmos collects together the properties of the fibrant objects and fibrations and weak equivalences between them in any model category that is enriched over the Joyal model structure and in which the fibrant objects are also cofibrant:

\begin{defn}[$\infty$-cosmos]\label{qcat.ctxt.cof.def}
An $\infty$-\textbf{cosmos} is a simplicially enriched category $\cK$ whose 
\begin{itemize}
\item objects we refer to as the \textbf{$\infty$-categories} in the $\infty$-cosmos, whose
\item hom simplicial sets $\Fun(A,B)$ are all  quasi-categories, 
\end{itemize} and that is equipped with a specified subcategory of \textbf{isofibrations}, denoted by ``$\fto$'',
satisfying the following axioms:
 \begin{enumerate}[label=(\alph*)]
    \item\label{qcat.ctxt.cof:a} (completeness) As a simplicially enriched category,  $\cK$ possesses a terminal object $1$, cotensors $A^U$ of  objects $A$ by all\footnote{For most purposes, it suffices to require only cotensors with finitely presented simplicial sets (those with only finitely many non-degenerate simplices).} simplicial sets $U$, and pullbacks of isofibrations along any functor.\footnote{For the theory of homotopy coherent adjunctions and monads developed in \cite{RV2}, limits of towers of isofibrations are also required, with the accompanying stability properties of \ref{qcat.ctxt.cof:b}. These limits are present in all of the $\infty$-cosmoi we are aware of, but will not be required for any results discussed here.}
    \item\label{qcat.ctxt.cof:b} (isofibrations) The class of isofibrations contains the isomorphisms and all of the functors $!\colon A \fto 1$ with codomain $1$; is stable under pullback along all functors; and if $p\colon E\fto B$ is an isofibration in $\cK$ and $i\colon U\hookrightarrow V$ is an inclusion of  simplicial sets then the Leibniz cotensor $\widehat{\{i, p\}}\colon E^V\fto E^U\times_{B^U} B^V$ is an isofibration. Moreover, for any object $X$ and isofibration $p \colon E \fto B$, $\Fun(X,p) \colon \Fun(X,E) \fto \Fun(X,B)$ is an isofibration of quasi-categories.
\end{enumerate}
The underlying category of an $\infty$-cosmos $\cK$ has a canonical subcategory of equivalences, denoted by ``$\wto$'', satisfying the two-of-six property. A functor $f \colon A \to B$ is an \textbf{equivalence} just when the induced functor $\Fun(X,f) \colon \Fun(X,A) \to \Fun(X,B)$ is an equivalence of quasi-categories for all objects $X \in \cK$.  The  \textbf{trivial fibrations}, denoted by ``$\fwto$'', are those functors that are both equivalences and isofibrations.
It follows from \ref{qcat.ctxt.cof.def}\ref{qcat.ctxt.cof:a}-\ref{qcat.ctxt.cof:b} that: 
 \begin{enumerate}[label=(\alph*), resume]
    \item\label{qcat.ctxt.cof:c} (cofibrancy) All objects are \textbf{cofibrant}, in the sense that they enjoy the left lifting property with respect to all trivial fibrations in $\cK$. 
    \[
    \begin{tikzcd} & E \arrow[d, two heads, "\wr"] \\ A \arrow[ur, dashed, "\exists"] \arrow[r] & B
    \end{tikzcd}
    \]
    \item\label{qcat.ctxt.cof:d} (trivial fibrations) The trivial fibrations define a subcategory containing  the isomorphisms; are stable under pullback along all functors; and the Leibniz cotensor $\widehat{\{i,p\}}\colon E^V\fwto E^U\times_{B^U}  B^V$ of an isofibration $p\colon E\fto B$ in $\cK$ and a monomorphism $i\colon U\hookrightarrow V$ between presented simplicial sets   is a trivial fibration when $p$ is a trivial fibration in $\cK$ or $i$ is trivial cofibration in the Joyal model structure on $\cat{sSet}$. Moreover, for any object $X$ and trivial fibration $p \colon E \fwto B$, $\Fun(X,p) \colon \Fun(X,E) \fwto \Fun(X,B)$ is a trivial fibration of quasi-categories.
\item\label{qcat.ctxt.cof:e} (factorization) Any functor $f \colon A \to B$ may be factored as $f = p j$ 
\[
\begin{tikzcd}
& N_{\mathrm{f}} \arrow[dr, two heads, "p"] \arrow[dl, bend right, two heads, "q"', "\sim" sloped] & \\ A \arrow[rr, "f"'] \arrow[ur, "\sim" sloped, "j"'] & & B
\end{tikzcd}
\]
where $p \colon N_{\mathrm{f}} \fto B$ is an isofibration and $j \colon A \wto N_{\mathrm{f}}$ is right inverse to a trivial fibration $q \colon N_{\mathrm{f}} \fwto A$. % (see \refIV{lem:Brown.fact}).
\end{enumerate}
\end{defn}

It is a straightforward exercise in enriched model category theory to verify that these axioms are satisfied by the fibrant objects in any model category that is enriched over the Joyal model structure on simplicial sets, at least when all of these objects are cofibrant. Consequently:

\begin{thm}[{Joyal-Tierney, Verity, Lurie, Riehl-Verity \cite{RV4}}] \label{thm:cosmoi} The full subcategories $\cat{qCat}$, $\cat{CSS}$, $\cat{Segal}$, and $\cat{qCat}_\natural$ all define $\infty$-cosmoi.
\end{thm}

Moreover, each of the model categories referenced in Theorem \ref{thm:cosmoi} are closed monoidal model categories with respect to the cartesian product. It follows that each of these four $\infty$-cosmoi are \textbf{cartesian closed} in the sense that they satisfy the extra axiom:
 \begin{enumerate}[label=(\alph*)]
 \setcounter{enumi}{5}
    \item\label{qcat.ctxt.cof:f} (cartesian closure) The product bifunctor $-\times - \colon \cK \times \cK \to \cK$ extends to a simplicially enriched two-variable adjunction
\[ \Fun(A \times B,C) \cong \Fun(A, C^B) \cong \Fun(B,C^A).\]
\end{enumerate}

A \textbf{cosmological functor} is a simplicial functor $F \colon \cK \to \cL$ preserving the class of isofibrations and all of the limits enumerated in Definition \ref{qcat.ctxt.cof.def}\ref{qcat.ctxt.cof:a}. A cosmological functor is a \textbf{biequivalence} when it is:
\begin{enumerate}[label=(\alph*)]
\item surjective on objects up to equivalence: i.e., if for every $C \in \cL$, there is some $A \in \cK$ so that $FA\simeq C \in \cL$.
\item a local equivalence of quasi-categories: i.e., if for every pair $A,B \in \cK$, the map  $\Fun(A,B) \wto \Fun(FA,FB)$ is an equivalence of quasi-categories.
\end{enumerate}
The inclusion $\cat{Cat}\hookrightarrow\cat{qCat}$ defines a cosmological functor but not a biequivalence, since it fails to be essentially surjective. Each right Quillen equivalence of 
\[
\begin{tikzcd}
 \cat{CSS} \arrow[r] \arrow[dr, shift left=.25em] & \cat{Segal} \arrow[d, shift left=.25em] \\ 
 \cat{qCat}_\natural \arrow[r, shift left=.25em] & \cat{qCat} \arrow[ul, shift left=.25em] \arrow[u, shift left=.25em] %\arrow[l, shift left=.25em, dashed]
\end{tikzcd}
\]
defines a cosmological biequivalence.

%The underlying quasi-category functor $\cat{qCat}_\natural \to \cat{qCat}$ is an isomorphism when restricted to these subcategories of fibrant objects, with inverse the ``natural marking'' functor, which accounts for the functor that does not appear in \eqref{eq:1-comparison}. The functor $\cat{CSS} \to \cat{Segal}$ commutes with the underlying quasi-category functors. Other composites are not displayed
%\end{proof}

As we shall discover in the next section, Theorem \ref{thm:cosmoi} together with additional observation ---  that the $\infty$-cosmoi of quasi-categories, Segal categories, complete Segal spaces, and naturally marked simplicial sets are biequivalent --- forms the lynchpin of an approach to develop the basic theory of $(\infty,1)$-categories in a model-independent fashion. In fact, most of that development takes places in a strict 2-category that we now introduce.

\begin{defn}[the homotopy 2-category of $\infty$-cosmos] The \textbf{homotopy 2-cat\-e\-gory} of an $\infty$-cosmos $\cK$ is a strict 2-category $\ho\cK$ so that 
\begin{itemize}
\item the objects of $\ho\cK$ are the objects of $\cK$, i.e., the $\infty$-\textbf{categories};
\item the 1-cells $f \colon A \to B$ of $\ho\cK$ are the vertices $f \in \Fun(A,B)$ in the mapping quasi-categories of $\cK$, i.e., the $\infty$-\textbf{functors};
\item a 2-cell  
$
\begin{tikzcd}
A \arrow[r, bend left, "f"] \arrow[r, bend right, "g"'] \arrow[r, phantom, "\scriptstyle\Downarrow\alpha"] & B
\end{tikzcd}$ in $\ho\cK$, which we call an $\infty$-\textbf{natural transformation}, is represented by a 1-simplex $\alpha \colon f \to g \in \Fun(A,B)$, where a parallel pair of 1-simplices in $\Fun(A,B)$ represent the same 2-cell if and only if they bound a 2-simplex whose remaining outer face is degenerate.
\end{itemize}
Put concisely, the homotopy 2-category is the 2-category $\ho\cK$ defined by applying the homotopy category functor $h \colon \cat{qCat} \to \cat{Cat}$ to the mapping quasi-categories of the $\infty$-cosmos; the hom-categories in $\ho\cK$  are defined by the formula  \[\Hom(A,B):= h\Fun(A,B)\] to be the homotopy categories of the mapping quasi-categories in $\cK$. 
\end{defn}

As we shall see in the next section, much of the theory of $(\infty,1)$-categories can be developed simply by considering them as objects in the homotopy 2-category using an appropriate weakening of standard 2-categorical techniques. A key to the feasibility of this approach is the fact that the standard 2-categorical notion of equivalence, reviewed in Definition \ref{defn:equivalence} below, coincides with the representably-defined notion of equivalence present in any $\infty$-cosmos. The proof of this result should be compared with Quillen's Proposition \ref{prop:structured-we}. 

\begin{prop}\label{prop:equiv-is-equiv}
An $\infty$-functor $ f\colon A \to B$ is an equivalence in the $\infty$-cosmos $\cK$ if and only if it is an equivalence in the homotopy 2-category $\ho\cK$.
\end{prop}
\begin{proof}
By definition, any equivalence $f \colon A \wto B$ in the $\infty$-cosmos induces an equivalence $\Fun(X,A) \wto \Fun(X,B)$ of quasi-categories for any $X$, which becomes an equivalence of categories $\Hom(X,A) \wto \Hom(X,B)$ upon applying the homotopy category functor $\ho \colon \cat{qCat} \to \cat{Cat}$. Applying the Yoneda lemma in the homotopy 2-category $\ho\cK$, it follows easily that $f$ is an equivalence in the standard 2-categorical sense.

Conversely, as the map $\iso \to \Delta^0$ of simplicial sets is a weak equivalence in the Joyal model structure, an argument similar to that used to prove Proposition \ref{prop:quasi-cyl} demonstrates that the cotensor $ B^\iso$ defines a path object for the $\infty$-category $B$. 
\[
\begin{tikzcd} & B^\iso \arrow[dr, two heads, "{(p_1,p_0)}"] \\ B \arrow[rr, "\Delta"'] \arrow[ur, "\sim" sloped] & & B \times B
\end{tikzcd}
\]
It follows from the two-of-three property that any $\infty$-functor that is isomorphic in the homotopy 2-category to an equivalence in the $\infty$-cosmos is again an equivalence in the $\infty$-cosmos. Now it follows immediately from the two-of-six property for equivalences in the $\infty$-cosmos and the fact that the class of equivalences includes the identities, that any 2-categorical equivalence is an equivalence in the $\infty$-cosmos.
\end{proof}

A consequence of Proposition \ref{prop:equiv-is-equiv} is that any cosmological biequivalence in particular defines an biequivalence of homotopy 2-categories, which explains the choice of terminology.

\section{Model-independent $(\infty,1)$-category theory}\label{sec:indep}

We now develop a small portion of the theory of $\infty$-categories in any $\infty$-cosmos, thereby developing a theory of $(\infty,1)$-categories that applies equally to quasi-cate\-gories, Segal categories, complete Segal spaces, and naturally marked quasi-categories.  The definitions of the basic $(\infty,1)$-categorical notions presented here might be viewed as ``synthetic,'' in the sense that they are blind to which model is being considered, in contrast with the ``analytic'' theory of quasi-categories first outlined in Joyal's \cite{joyal-quasi} and later greatly expanded in his unpublished works and Lurie's \cite{lurie-topos, lurie-algebra}. In \S\ref{ssec:adj-equiv}, we introduce adjunctions and equivalences between $\infty$-categories, which generalize the notions of Quillen adjunction and Quillen equivalence between model categories from \S\ref{ssec:derived-adj} and \S\ref{ssec:quillen-equiv}. Then in \S\ref{ssec:inf-limits}, we develop the theory of limits and colimits in an $\infty$-category, which correspond to the homotopy limits and colimits of \S\ref{sec:holim}.

Our synthetic definitions specialize in the $\infty$-cosmos of quasi-categories to notions that precisely recapture the Joyal-Lurie analytic theory; the proofs that this is the case are not discussed here, but can be found in \cite{RV1, RV4}. Considerably more development along these lines can be found in \cite{RV-scratch}.

\subsection{Adjunctions and equivalences}\label{ssec:adj-equiv}

In any 2-category, in particular in the homotopy 2-category $\ho\cK$ of an $\infty$-cosmos, there are standard definitions of adjunction or equivalence, which allow us to define adjunctions and equivalences between $\infty$-categories.

\begin{defn}\label{defn:adjunction} An \textbf{adjunction} between $\infty$-categories consists of:
\begin{itemize}
\item a pair of $\infty$-categories $A$ and $B$;
\item a pair of $\infty$-functors $f \colon B \to A$ and $u \colon A \to B$; and
\item a pair of $\infty$-natural transformations $\eta \colon \id_B \To uf$ and $\epsilon \colon fu \To \id_A$
\end{itemize}
so that the triangle equalities hold:
\[
\begin{tikzcd}[column sep=small] & B \arrow[d, phantom, "\scriptstyle\Downarrow\epsilon"] \arrow[dr,  "f" description] \arrow[rr, equals] & \arrow[d, phantom, "\scriptstyle\Downarrow\eta"] & B \arrow[drr, phantom, "="] && B \arrow[d, phantom, "="] &  & B \arrow[dr, "f"'] \arrow[rr, equals] &\arrow[d, phantom, "\scriptstyle\Downarrow\eta"] & B \arrow[d, phantom, "\scriptstyle\Downarrow\epsilon"] \arrow[dr, "f"] & \arrow[drr, phantom, "="] &  & B \arrow[d, phantom, "="] \arrow[d, bend left, "f"] \arrow[d, bend right, "f"'] \\ 
A \arrow[ur, "u"] \arrow[rr, equals] & ~& A \arrow[ur, "u"']  & ~&~ & A \arrow[u, bend left, "u"] \arrow[u, bend right, "u"']  & & &  A \arrow[rr, equals] \arrow[ur,  "u" description] &~ & A &&  A
\end{tikzcd}
\]
\end{defn}%The left-hand equality of pasting diagrams asserts that $u\epsilon \cdot \eta u = \id_u$, while the right-hand equality asserts that $\epsilon f \cdot f\eta = \id_{\mathrm{f}}$.

We write $f \dashv u$ to assert that the $\infty$-functor $f \colon B \to A$ is \textbf{left adjoint} to the $\infty$-functor $u \colon A \to B$, its \textbf{right adjoint}.

\begin{defn}\label{defn:equivalence} An \textbf{equivalence} between $\infty$-categories consists of:
\begin{itemize}
\item a pair of $\infty$-categories $A$ and $B$;
\item a pair of $\infty$-functors $f \colon B \to A$ and $g \colon A \to B$; and
\item a pair of natural isomorphisms  $\eta \colon \id_B \cong gf$ and $\epsilon \colon fg \cong \id_A$.
\end{itemize}
An \textbf{$\infty$-natural isomorphism} is a 2-cell in the homotopy 2-category that admits a vertical inverse 2-cell.
\end{defn}

We write $A \simeq B$ and say that $A$ and $B$ are \textbf{equivalent} if there exists an equivalence between $A$ and $B$. The direction for the $\infty$-natural isomorphisms comprising an equivalence is immaterial. Our notation is chosen to suggest the connection with adjunctions conveyed by the following exercise.

\begin{exc}\label{exc:adjoint-equivalence} In any 2-category, prove that:
\begin{enumerate}
\item\label{itm:adj-comp} Adjunctions compose: given adjoint $\infty$-functors
\[
\begin{tikzcd}
C \arrow[r, bend left, "f'"] \arrow[r, phantom, "\perp"] & B \arrow[r, bend left, "f"] \arrow[r, phantom, "\perp"] \arrow[l, bend left, "u'"] & A \arrow[l, bend left, "u"] & \rightsquigarrow & C \arrow[r, bend left, "ff'"] \arrow[r, phantom, "\perp"] & A \arrow[l, bend left, "u'u"]
\end{tikzcd}
\]
the composite $\infty$-functors are adjoint.
\item\label{itm:adj-equiv} Any equivalence can always be promoted to an \textbf{adjoint equivalence} by modifying one of the $\infty$-natural isomorphisms. That is, show that the $\infty$-natural isomorphisms in an equivalence can be chosen so as to satisfy the triangle equalities. Conclude, that if $f$ and $g$ are inverse equivalences then $f \dashv g$ and $g \dashv f$.
\end{enumerate}
\end{exc}

The point of Exercise \ref{exc:adjoint-equivalence} is that there are various diagrammatic 2-categorical proofs that can be taken off the shelf and applied to the homotopy 2-category of an $\infty$-cosmos to prove theorems about adjunctions and equivalence between $(\infty,1)$-categories.

\subsection{Limits and colimits}\label{ssec:inf-limits}

We now introduce definitions of limits and colimits for diagrams valued inside an $\infty$-category. We begin by defining terminal objects, or as we shall call them ``terminal elements,'' to avoid an over proliferation of the generic name ``objects.''

\begin{defn}\label{defn:terminal-element} A \textbf{terminal element} in an $\infty$-category $A$ is a right adjoint $t \colon 1 \to A$ to the unique $\infty$-functor $! \colon A \to 1$. Explicitly, the data consists of:
\begin{itemize}
\item an element $t \colon 1 \to A$ and
\item a $\infty$-natural transformation $\eta \colon \id_A \Rightarrow t!$ whose component $\eta t$ at the element $t$ is an isomorphism.\footnote{If $\eta$ is the unit of the adjunction $! \dashv t$, then the triangle equalities demand that $\eta t =\id_t$.  However, by a 2-categorical trick, to show that such an adjunction exists, it suffices to find a 2-cell $\eta$ so that $\eta t$ is an isomorphism.} 
\end{itemize}
\end{defn}

Several basic facts about terminal elements can be deduced immediately from the general theory of adjunctions.

\begin{exc} $\quad$
\begin{enumerate}[label=(\roman*)]
\item\label{itm:term-ii} Terminal elements are preserved by right adjoints and by equivalences.
\item\label{itm:term-iii} If  $A' \simeq A$ then $A$ has a terminal element if and only if $A'$ does.
\end{enumerate}
\end{exc}

Terminal elements are limits of empty diagrams. We now turn to limits of generic diagrams whose indexing shapes are given by 1-categories. For any $\infty$-category $A$ in an $\infty$-cosmos $\cK$, there is a 2-functor $A^{(-)} \colon \cat{Cat}^\op \to \ho\cK$ defined by forming simplicial cotensors with nerves of categories. Using these simplicial cotensors, if $J$ is a 1-category and $A$ is an $\infty$-category, the \textbf{$\infty$-category of $J$-indexed diagrams in $A$} is simply the cotensor $A^J$.\footnote{More generally, this construction permits arbitrary simplicial sets as indexing shapes for diagrams in an $\infty$-category $A$. In either case, the elements of $A^J$ are to be regarded as homotopy coherent diagrams along the lines of Remark \ref{rmk:coherent-diagram-qcat}.}

%We have a 2-category $\cat{sSet}$ of simplicial sets, extending in the evident way the definition of the homotopy 2-category $\cat{qCat} \subset \cat{sSet}$ of quasi-categories. The 2-category of categories sits as a full subcategory $\cat{Cat}\subset\cat{qCat} \subset \cat{sSet}$, with categories identified with the simplicial sets defining their nerves. In this way, diagrams indexed by categories are among the diagrams indexed by simplicial sets. 

\begin{rmk}\label{rmk:diagram-inf-cat} In the cartesian closed $\infty$-cosmoi of  Definition \ref{qcat.ctxt.cof.def}\ref{qcat.ctxt.cof:f}, we also permit the indexing shape $J$ to be another $\infty$-category, in which case the internal hom $A^J$ defines the \textbf{$\infty$-category of $J$-indexed diagrams in $A$}. The development of the theory of limits indexed by an $\infty$-category in a cartesian closed $\infty$-cosmos entirely parallels the development for limits indexed by 1-categories, a parallelism we highlight by conflating the notation of \ref{qcat.ctxt.cof.def}\ref{qcat.ctxt.cof:a} and \ref{qcat.ctxt.cof.def}\ref{qcat.ctxt.cof:f}.
\end{rmk}

In analogy with Definition \ref{defn:terminal-element}, we have:

\begin{defn}\label{defn:all-limits} An $\infty$-category $A$ \textbf{admits all limits of shape $J$} if the constant diagram $\infty$-functor $\Delta \colon A \to A^J$, induced by the unique $\infty$-functor $!\colon  J \to 1$, has a right adjoint:
\[
\begin{tikzcd}
A \arrow[r, bend left, "\Delta"] \arrow[r, phantom, "\perp"] & A^J \arrow[l, bend left, "\lim"]
\end{tikzcd}
\]
\end{defn}

From the vantage point of Definition \ref{defn:all-limits}, the following result is easy:

\begin{exc} Using the general theory of adjunctions, show that a right adjoint $\infty$-functor $u \colon A \to B$ between $\infty$-categories that admit all limits of shape $J$  necessarily preserves them, in the sense that the $\infty$-functors
\[
\begin{tikzcd}
A^J \arrow[d, "\lim"'] \arrow[r, "u^J"] & B^J \arrow[d, "\lim"] \arrow[dl, phantom, "\cong"] \\ A \arrow[r, "u"'] & B
\end{tikzcd}
\]
commute up to isomorphism.
\end{exc}

The problem with Definition \ref{defn:all-limits} is that it is insufficiently general: many $\infty$-categories will have certain, but not all, limits of diagrams of a particular indexing shape. With this aim in mind, we will now re-express Definition \ref{defn:all-limits} in a form that permits its extension to cover this sort of situation. For this, we make use of the 2-categorical notion of an \textbf{absolute right lifting}, which is the ``op''-dual (reversing the 1-cells but not the 2-cells) of the notion of absolute right Kan extension introduced in Definition \ref{defn:Kanext}.

\begin{exc}\label{exs:adj-as-abs-lifting} Show that in any 2-category, a 2-cell $\epsilon \colon fu \To \id_A$ defines the counit of an adjunction $f \dashv u$ if and only if
\[
\begin{tikzcd} \arrow[dr, phantom, "\scriptstyle\Downarrow\epsilon"  pos=.9] & B \arrow[d, "f"] \\ A \arrow[ur, "u"] \arrow[r, equals] & A
\end{tikzcd}
\]
defines an absolute right lifting diagram.
\end{exc}

Applying Exercise \ref{exs:adj-as-abs-lifting}, Definition \ref{defn:all-limits} is equivalent to the assertion that the \textbf{limit cone}, our term for the counit of $\Delta \dashv \lim$,  defines an absolute right lifting diagram:
\begin{equation}\label{eq:all-limits-abs-lifting}
\begin{tikzcd} \arrow[dr, phantom, "\scriptstyle\Downarrow\epsilon" pos=.9] & A \arrow[d, "\Delta"] \\ A^J \arrow[ur, "\lim"] \arrow[r, equals] & A^J
\end{tikzcd}
\end{equation}
Recall that the appellation ``absolute'' means ``preserved by all functors,'' in this case by restriction along any $\infty$-functor $X \to A^J$. In particular, an absolute right lifting diagram \eqref{eq:all-limits-abs-lifting} restricts to define an absolute right lifting diagram on any subobject of the $\infty$-category of diagrams. This motivates the following definition.

\begin{defn}[limit]\label{defn:limit} A \textbf{limit} of a $J$-indexed diagram in $A$ is an absolute right lifting of the diagram $d$ through the constant diagram $\infty$-functor $\Delta \colon A \to A^J$
\begin{equation}\label{eq:lim-diagram-defn}
\begin{tikzcd} \arrow[dr, phantom, "\scriptstyle\Downarrow\lambda"  pos=.9] & A \arrow[d, "\Delta"] \\ 1\arrow[ur, "\lim d"] \arrow[r, "d"'] & A^J
\end{tikzcd}
\end{equation}
 the 2-cell component of which  defines the \textbf{limit cone} $\lambda \colon \Delta \lim d \To d$.
\end{defn}

If $A$ has all $J$-indexed limits, then the restriction of the absolute right lifting diagram \eqref{eq:all-limits-abs-lifting} along the element $d \colon 1\to A^J$ defines a limit for $d$. Interpolating between Definitions \ref{defn:limit} and \ref{defn:all-limits}, we can define a \textbf{limit of a family of diagrams} to be an absolute right lifting of the family $d \colon K \to A^J$ through $\Delta \colon A \to A^J$. For instance:

\begin{thm}[{\cite[5.3.1]{RV1}}]\label{thm:totalization} For every cosimplicial object in an $\infty$-category that admits an coaugmentation and a splitting, the coaugmentation defines its limit. That is, for every $\infty$-category $A$, the $\infty$-functors
\[
\begin{tikzcd} \arrow[dr, phantom, "\scriptstyle\Downarrow\lambda"  pos=.9] & A \arrow[d, "\Delta"] \\
A^{\DDelta_{\bot}} \arrow[r, "\mathrm{res}"'] \arrow[ur, "\ev_{[-1]}"] & A^{\DDelta}
\end{tikzcd}
\]
define an absolute right lifting diagram.
\end{thm}
Here $\DDelta$ is the usual simplex category of finite non-empty ordinals and order-preserving maps. It defines a full subcategory of $\DDelta_+$, which freely appends an initial object $[-1]$, and this in turn defines a  subcategory of $\DDelta_{\bot}$, which adds an ``extra degeneracy'' map between each pair of consecutive ordinals. Diagrams indexed by $\DDelta \subset \DDelta_+\subset \DDelta_{\bot}$ are, respectively, called \textbf{cosimplicial objects}, \textbf{coaugmented cosimplicial objects}, and \textbf{split cosimplicial objects}. The limit of a cosimplicial object is often called its \textbf{totalization}.

\begin{proof}[Proof sketch]
In $\cat{Cat}$, there is a canonical 2-cell
\[
\begin{tikzcd}
\DDelta \arrow[r, hook] \arrow[d, "!"'] \arrow[dr, phantom, "\scriptstyle\Uparrow\lambda" pos=.2] & \DDelta_{\bot} \\
\1 \arrow[ur, "{[-1]}"'] & ~
\end{tikzcd}
\]
 because $[-1] \in \DDelta_{\bot}$ is initial. This data defines an absolute right extension diagram that is moreover preserved by any 2-functor, because the universal property of the functor $[-1] \colon \1 \to \DDelta_{\bot}$ and the 2-cell $\lambda$ is witnessed by a pair of adjunctions. The 2-functor $A^{(-)} \colon \cat{Cat}^\op \to \ho\cK$ converts this into the absolute right lifting diagram of the statement.
\end{proof} 
  
The most important result relating adjunctions and limits is of course:

\begin{thm}[{\cite[5.2.13]{RV1}}]\label{thm:RAPL} Right adjoints preserve limits.
\end{thm}
Our proof will closely follow the classical one. Given a diagram $d\colon 1 \to A^J$ and a right adjoint $u \colon A \to B$ to some $\infty$-functor $f$, a cone with summit $b \colon 1 \to B$ over $u^J d$ transposes to define a cone with summit $fb$ over $d$, which factors uniquely through the limit cone. This factorization transposes back across the adjunction to show that $u$ carries the limit cone over $d$ to a limit cone over $u^Jd$.

\begin{proof}
Suppose that $A$ admits limits of a diagram $d\colon 1\to A^J$ as witnessed by an absolute right lifting diagram \eqref{eq:lim-diagram-defn}. Since adjunctions are preserved by all 2-functors, an adjunction $f \dashv u$ induces an adjunction $f^J \dashv u^J$. We must show that 
\[
\begin{tikzcd} \arrow[dr, phantom, "{\scriptstyle\Downarrow\lambda}"  pos=.9] & A \arrow[d, "\Delta"] \arrow[r, "u"] & B \arrow[d, "\Delta"] \\ 1\arrow[ur, "\lim d"] \arrow[r, "d"'] & A^J \arrow[r, "u^J"'] & B^J
\end{tikzcd}
\]
is again an absolute right lifting diagram. Given a square
\[
\begin{tikzcd}
X \arrow[rr, "b"] \arrow[d, "{!}"'] & \arrow[d, phantom, "\scriptstyle\Downarrow\chi"] & B \arrow[d, "\Delta"]\\
1 \arrow[r, "d"'] & A^J \arrow[r, "u^J"'] & B^J
\end{tikzcd}
\]
we first ``transpose across the adjunction,'' by composing with $f$ and the counit. 
\[
\begin{tikzcd}
X \arrow[rr, "b"] \arrow[d, "{!}"'] & \arrow[d, phantom, "\scriptstyle\Downarrow\chi"] & B \arrow[d, "\Delta"] \arrow[r, "f"] & A \arrow[d, "\Delta"] \arrow[dr, phantom, "="] & X \arrow[d, "{!}"'] \arrow[r, "b"] \arrow[drr, phantom, "{\scriptstyle\exists !\Downarrow\zeta}" pos=.2, "\scriptstyle\Downarrow\lambda"  pos=.8] & B \arrow[r, "f"] & A \arrow[d, "\Delta"] \\
1 \arrow[r, "d"'] & A^J \arrow[rr, bend right, equals, "\scriptstyle\Downarrow\epsilon^J"] \arrow[r, "u^J"] & B^J \arrow[r, "f^J"] & A^J & 1 \arrow[urr, "\lim d" description] \arrow[rr, "d"'] & & A^J
\end{tikzcd}
\]
The universal property of the absolute right lifting diagram $\lambda \colon \Delta \lim \To d$ induces a unique factorization $\zeta$, which may then be ``transposed back across the adjunction'' by composing with $u$ and the unit.
\[
\begin{tikzcd}[column sep=2.2em]
X \arrow[d, "{!}"'] \arrow[r, "b"] \arrow[drr, phantom, "{\scriptstyle\exists !\Downarrow\zeta}" pos=.2, "\scriptstyle\Downarrow\lambda" pos=.8] & B \arrow[rr, bend left, equals, "\scriptstyle\Downarrow\eta"'] \arrow[r, "f"] & A \arrow[d, "\Delta"] \arrow[r, "u"'] & B \arrow[d, "\Delta"] \arrow[dr, phantom, "="] & X \arrow[rr, "b"] \arrow[d, "{!}"'] & \arrow[d, phantom, "\scriptstyle\Downarrow\chi"] & B \arrow[d, "\Delta"] \arrow[r, "f"']  \arrow[rr, bend left, equals, "\scriptstyle\Downarrow\eta"'] & A \arrow[d, "\Delta"] \arrow[r, "u"'] & B \arrow[d, "\Delta"] \\ 1 \arrow[urr, "\lim d" description] \arrow[rr, "d"'] & & A^J \arrow[r, "u^J"'] & B^J & 1 \arrow[r, "d"'] & A^J \arrow[r, "u^J"] \arrow[rr, equals, bend right, "\scriptstyle\Downarrow\epsilon^J"] & B^J \arrow[r, "f^J"] & A^J \arrow[r, "u^J"'] & B^J
\end{tikzcd}
\]
\[
= \quad \begin{tikzcd}
X \arrow[d, "{!}"'] \arrow[rr, "b"] \arrow[drr, phantom, "\scriptstyle\Downarrow\chi"] & & B \arrow[d, "\Delta"] \arrow[rr, bend left, equals] & & B \arrow[d, "\Delta"] \arrow[dr, phantom, "="] & X \arrow[rr, "b"] \arrow[d, "{!}"'] & \arrow[d, phantom, "\scriptstyle\Downarrow\chi"] & B \arrow[d, "\Delta"]\\
1 \arrow[r, "d"'] & A^J \arrow[r, "u^J"] \arrow[rr, bend right, equals, "\scriptstyle\Downarrow\epsilon^J"] & B^J \arrow[rr, bend left, equals, "\scriptstyle\Downarrow\eta^J"'] \arrow[r, "f^J" description] & A^J \arrow[r, "u^J"'] & B^J &
1 \arrow[r, "d"'] & A^J \arrow[r, "u^J"'] & B^J
\end{tikzcd}
\]
Here the second equality is a consequence of the 2-functoriality of the simplicial cotensor, while the third is an application of a triangle equality for the adjunction $f^J \dashv u^J$. The pasted composite of $\zeta$ and $\eta$ is the desired factorization of $\chi$ through $\lambda$. 

The proof that this factorization is unique, which again parallels the classical argument, is left to the reader: the essential point is that the transposes defined via these pasting diagrams are unique.
\end{proof}

Colimits are defined ``co''-dually, by reversing the direction of the 2-cells but not the 1-cells. There is no need to repeat the proofs however: any $\infty$-cosmos $\cK$ has a co-dual $\infty$-cosmos $\cK^\co$ with the same objects but in which the mapping quasi-categories are defined to be the opposites of the mapping quasi-categories in $\cK$.

\section{Epilogue}\label{sec:epilogue}

A category $\cK$ equipped with a class of ``weak equivalences'' $\we$ --- perhaps saturated in the sense of containing all of the maps inverted by the Gabriel-Zisman localization functor or perhaps merely generating the class of maps to be inverted in the category of fractions --- defines a ``homotopy theory,'' a phrase generally used to refer to the associated homotopy category together with the homotopy types of the mapping spaces, as captured for instance by the Dwyer-Kan hammock localization.  We have studied two common axiomatizations of this abstract notion: Quillen's model categories, which present homotopy theories with all homotopy limits and homotopy colimits, and $(\infty,1)$-categories, which might be encoded using one of the models introduced in \S\ref{sec:models} or worked with ``model-independently'' in the sense outlined in \S\ref{sec:indep}.

From the point of view of comparing homotopy categories, the model-indepen\-dent theory of $(\infty,1)$-categories  has some clear advantages: equivalences between homotopy theories are directly definable (see Definition \ref{defn:equivalence}) instead of being presented as zig-zags of DK- or Quillen equivalences. The formation of diagram categories (see Remark \ref{rmk:diagram-inf-cat}) is straightforward and homotopy limit and colimit functors become genuine adjoints (see Definition \ref{defn:all-limits}) and homotopy limits and colimits become genuine limits and colimits --- at least in the sense appropriate to the theory of $(\infty,1)$-categories. So from this vantage point it is natural to ask: ``Do we still need model categories?''\footnote{See \href{https://mathoverflow.net/questions/78400/do-we-still-need-model-categories}{mathoverflow.net/questions/78400/do-we-still-need-model-categories}.} While some might find this sort of dialogue depressing in our view it does not hurt to ask.

Chris Schommer-Pries has suggested a useful analogy to contextualize the role played by model categories in the study of homotopy theories that are complete and cocomplete:
\[
\begin{cases} \text{model~category} :: (\infty,1)\text{-category} \\ \text{basis} :: \text{vector~space} \\ \text{local~coordinates} :: \text{manifold}
\end{cases}
\]
A precise statement is that combinatorial model categories present those $(\infty,1)$-cate\-gor\-ies that are complete and cocomplete and more generally (locally) presentable; this result is proven in \cite[A.3.7.6]{lurie-topos} by applying a theorem of Dugger \cite{dugger-combinatorial}.\footnote{Morally, in the sense discussed in \S\ref{ssec:functoriality}, all model categories are Quillen equivalent to locally presentable ones. More precisely, the result that every ``cofibrantly generated'' (in a suitable sense of this term) model category is Quillen equivalent to a combinatorial one has been proven by Raptis and Rosicky to be equivalent to a large cardinal axiom called Vop\v{e}nka's principle \cite{rosicky}.} In general having coordinates are helpful for calculations. In particular, when working inside  a particular homotopy theory as presented by a model category, you also have access to the non-bifibrant objects. For instance, the Bergner model structure of \S\ref{ssec:bergner} is a useful context to collect results about homotopy coherent diagrams, which are defined to be maps from the cofibrant (and not typically fibrant) objects to the fibrant ones (which are not typically cofibrant).

But Quillen himself was somewhat unsatisfied with the paradigm-shifting abstract framework that he introduced, writing:
\begin{quote}
This definition of the homotopy theory associated to a model category is obviously unsatisfactory. In effect, the loop and suspension functors are a kind of primary structure on $\Ho\cM$ and the families of fibration and cofibration sequences are a kind of secondary structure since they determine the Toda bracket \ldots. Presumably there is higher order structure on the homotopy category which forms part of the homotopy theory of a model category, but we have not been able to find an inclusive general definition of this structure with the property that this structure is preserved when there are adjoint functors which establish an equivalence of homotopy theories.
--- Quillen \cite{quillen} pp.~3-4.
\end{quote}
Quillen was referring to a model category that is pointed, in the sense of having a zero object, like the roll played by the singleton space in $\cat{Top}_*$. A more modern context for the sort of stable homotopy theory that Quillen is implicitly describing is the category of spectra, the $(\infty,1)$-category of which has many pleasant properties collected together in the notion of a \emph{stable $\infty$-category}. We posit that these notions, which are the subject of the next chapter, might fulfill Quillen's dream.

\end{document}